\documentclass[11pt,reqno]{amsart}
\usepackage{amsmath}
\usepackage{amssymb, latexsym, amsfonts, amscd, amsthm, mathrsfs, enumerate, esint}

\usepackage[usenames,dvipsnames]{color}
\usepackage[all]{xy}
\usepackage{graphicx}
\usepackage{epsfig}
\usepackage{hyperref}

\numberwithin{equation}{section}

\definecolor{purple}{rgb}{0.9,0,0.8}

\definecolor{gray}{rgb}{0.7,0.7,0.7}

\newcommand{\abbr}[1]{{\sc\lowercase{#1}}}

\newtheorem{thm}{Theorem}[section]

\newtheorem{lem}[thm]{Lemma}
\newtheorem{ppn}[thm]{Proposition}

\newtheorem{conj}[thm]{Conjecture}
\theoremstyle{definition}
\newtheorem{defn}[thm]{Definition}
\newtheorem{ex}[thm]{Example}
\newtheorem{remark}[thm]{Remark}
\newtheorem{assumption}[thm]{Assumption}

\newcommand{\beq}{\begin{equation}}
\newcommand{\eeq}{\end{equation}}

\newcommand{\bbeta}{\beta}
 
\newcommand{\ep}{\epsilon}

\newcommand{\bpi}{{\boldsymbol \pi}}

\newcommand{\B}{\mathbb{B}}

\newcommand{\bC}{\mathbb{C}}

\newcommand{\E}{\mathbb{E}}
\newcommand{\bE}{\mathbb{E}}

\newcommand{\bG}{\mathbb{G}}
\newcommand{\blG}{\boldsymbol{G}}
\newcommand{\G}{\mathbb{G}}
\newcommand{\bH}{\mathbb{H}}

\newcommand{\M}{\mathbb{M}}
\newcommand{\N}{\mathbb{N}}

	\renewcommand{\P}{\mathbb{P}}	
\newcommand{\bP}{\mathbb{P}}
\newcommand{\bPi}{\boldsymbol{\Pi}}

\newcommand{\R}{\mathbb{R}}
\newcommand{\bS}{\mathbb{S}}
\newcommand{\bT}{\mathbb{T}}

\newcommand{\Z}{\mathbb{Z}}
\newcommand{\bZ}{\mathbb{Z}}

\newcommand{\cC}{\mathcal{C}}

\newcommand{\cE}{\mathcal{E}}

\newcommand{\cL}{\mathcal{L}}
\newcommand{\cM}{\mathcal{M}}
\newcommand{\cN}{\mathcal{N}}

\newcommand{\sss}{\mathsf{s}}

\newcommand{\st}{\mathsf{t}}

\newcommand{\wt}{\widetilde}
\newcommand{\wh}{\widehat}

\newcommand{\grad}{\nabla}
\newcommand{\I}{\mathbf{1}}

\newcommand{\CN}{\mathsf{C}_{\mathsf n}}
\newcommand{\CV}{\mathsf{C}_{\mathsf v}}
\newcommand{\Cd}{\mathsf{C}_{\mathsf D}}
\newcommand{\CP}{\mathsf{C}_{\mathsf P}}
\newcommand{\CPw}{\mathsf{C}'_{\mathsf P}}
\newcommand{\CB}{\mathsf{C_0}}
\newcommand{\Cbu}{\mathsf{C}_\star}

\begin{document}
\author[A.\ Dembo]{Amir Dembo}
\author[R.\ Huang]{Ruojun Huang}
\author[T.\ Zheng]{Tianyi Zheng}
\address{Department of Mathematics, Stanford University. Building 380, 
Sloan Hall, Stanford, CA 94305, USA.}
\address{Department of Statistics, Stanford University. Sequoia Hall, 390 Serra Mall, Stanford, CA 94305, USA.}
\address{Department of Mathematics, University of California, San Diego. 9500 Gilman Dr, La Jolla, CA 92093, USA.}

\date{\today}
\thanks{This research was supported in part by NSF grant DMS-1613091.}



\title[Random walks among time increasing conductances: heat kernel estimates]{Random Walks Among Time Increasing Conductances: \\
Heat Kernel Estimates}
\maketitle

\begin{abstract}
For any graph having a suitable uniform Poincar\'e inequality
and volume growth regularity, we establish two-sided Gaussian 
transition density estimates and parabolic Harnack inequality, 
for constant speed continuous time random walks evolving via
time varying, uniformly elliptic conductances, provided the 
vertex conductances (i.e. reversing measures), increase in time.
Such transition density upper bounds apply for
discrete time uniformly lazy walks, with the matching lower 
bounds holding once the parabolic Harnack inequality is proved.
\end{abstract}

\begin{section}{Introduction} 
One of the most studied models for random walks in disordered 
media is the (random) conductance model, based on a locally 
finite, connected, non-oriented, graph $\bG=(V,E)$ equipped with a (random) collection 
of symmetric \emph{strictly positive}, edge conductances 
$\bPi := \{ \pi(x,y) > 0 : (x,y) \in E\}$ 
(for example, see \cite{MB,Kum} and references therein).
We consider here random walks 
among non-random but time-varying, edge-wise 
Borel measurable,
conductances $\bPi_t$.
In particular, taking $\bPi_t=\bPi_{[t]}$,
the discrete time (simple) random walk (\abbr{dtrw}) $\{X_{n}\}$ 
on the corresponding sequence of weighted graphs 
$\blG_{n}=(\bG,\bPi_n)$ is
constructed as a \emph{time-in-homogeneous} $V$-valued 
Markov chain, having for $t \in \N$ the transition probabilities
\begin{align}
K_t(x,y):=\frac{\pi_t(x,y)}{\pi_t (x)},\quad 
(x,y)\in E,\,t \ge 0,\label{n-kernel}
\end{align}
each of which is reversible with respect to the vertex conductances
\begin{align}\label{dfn:v-cond}
\pi_{t}(x):=\sum_{\{y: (x,y) \in E\}}\pi_{t}(x,y) \,.
\end{align}
Any such $\pi_t(\cdot)$ is a 
$\sigma$-finite measure on $V$
and 
we denote by $\cM_+(V)$ those 
$\sigma$-finite measures which are
bounded away from zero,  
namely functions $\mu: V \mapsto (0,\infty)$ with 
$\underline{\mu}:=\inf_x \mu(x) > 0$.

More generally, starting at $K_{n,n} = I$, any 
transition probabilities $\{K_\ell\}$ on $V$, 
induce the time-inhomogeneous transitions 
\begin{align}\label{def:k-mn}
K_{k-1,n}(x,z) := \bP(X_n=z|X_{k-1}=x) = 
\sum_{y\in V} K_{k}(x,y)K_{k,n}(y,z),\ k \le n\,.\quad 
\end{align}
Similarly, for $\{K_t\}$ of \eqref{n-kernel},
the constant speed random walk (\abbr{csrw}),
is the $V$-valued stochastic process of \abbr{RCLL} 
sample path $t\mapsto Y_{t}$ that at the arrival times
$\{T_n\}$ of an auxiliary unit rate Poisson process, 
jumps to $y \ne x$ according to  
\begin{align}
\bP(Y_{T_n} = y | Y_{T_n^{-}}=x) = K_{T_n}(x,y) \,,\qquad
(x,y)\in E,\,n \in \N \,.\label{def:csrw} 
\end{align}
The \abbr{csrw} is thus a 
time-inhomogeneous Markov process. Its transition probabilities 
$K_{s,t}(x,z) = \bP(Y_t=z|Y_s=x)$ satisfy the inhomogeneous semi-group property 
$K_{s,u}K_{u,t}=K_{s,t}$ for any $s<u<t$. Furthermore, such $\{K_{s,t}(x,z)\}$
solve with (initial) condition 
$u(t,x)= {\bf 1}_{\{z=x\}}$, the backward equation
\begin{align}
u(s,x)& := \E [u(t,Y_t) | Y_s=x] = e^{-(t-s)} u(t,x) 
+ \int_s^t e^{-(\xi-s)} d\xi \sum_{y \in V} K_\xi(x,y) u(\xi,y) \,, 
\label{def:k-st}
\end{align}
which upon setting 
$(\cL_s f) (x) := \sum_{y \in V} (f(y) - f(x)) K_s(x,y)$, amounts 
to the distributional solution of $\partial_{-s} u = \cL_s u(s,\cdot)$ 
(see \eqref{dfn:sol-cyl}).

Any conductance model, having $\bPi$ independent of $t$, is a reversible, 
time-homogeneous network, of reversing measure 
$\{ \pi(x): x \in V \}$. Time varying $\bPi_t$ for which $\pi_t(x)=\pi(x)$ 
are independent of $t$, retain this reversibility (even though they form
time-inhomogeneous transitions $K_{s,t}$). In contrast, 
as soon as $\{ \pi_t(x): x \in V\}$ changes with time ($t$), 
the dynamics associated with \eqref{def:k-mn} or with \eqref{def:k-st}
become genuinely non-reversible. Nevertheless, it has been suggested in \cite{ABGK} that some universality applies
for the recurrence versus transience of such dynamics.
Specifically, \cite[Conj. 7.1]{ABGK} conjectures that 
if both conductance models 
corresponding to $\blG_0$ and $\blG_\infty$ are recurrent, 
or alternatively, both $\blG_0$ and $\blG_\infty$ are transient,
then the same holds for the dynamic of non-decreasing 
$\{n \mapsto \blG_n\}$, namely the \abbr{DTRW} evolving according 
to \eqref{def:k-mn}. Indeed, using flows to construct 
suitable sub or super-harmonic functions, such universality 
is established in \cite[Sec. 5]{ABGK} when $\bG=\bT$ is a tree 
in the recurrent case, and when $\G=\N$ in the transient case, 
even allowing for conductances $\bPi_n$ adapted to the path
$\{X_k, k \le n\}$. In contrast, \cite[Sec. 6]{ABGK} shows
that such universality fails for randomly adapted, increasing in 
time, conductances on $\bG=\bZ^2$, whereas \cite[Ex. 3.5 and 3.6]{ABGK}
demonstrates such failure in the non-adapted and non-monotone setting
(even on the trivial tree $\bG=\bZ$).

The intuition behind \cite[Conj. 7.1]{ABGK} owes to the equivalence
between conductance models and electrical networks, yielding key 
comparisons such as Rayleigh monotonicity principle (due to which
the random walk on any sub-graph $\bG'$ of a recurrent $\bG$ 
must also be recurrent). Lacking such comparisons for the
time-varying conductances of \eqref{n-kernel}, 
we instead 
seek alternative analytic tools, such as, establishing  
the relevant (two-sided) Gaussian 
transition density estimate under certain geometric 
assumptions on the underlying graphs. 
That is, to show that for some $C$ finite, suitable measures
$\mu_{s,t}(\cdot)$, all $x,y \in V$ and $t-s \ge d(x,y)$,
\begin{align}
\text{\abbr{(GHKU)}}\quad  \quad K_{s,t}(x,y) &\le
\frac{C \, \mu_{s,t}(y)}{\pi_s(\B(y,\sqrt{t-s}))}\exp\left(-  \frac{d(x,y)^{2}}{C(t-s)}\right),\label{ghk-ubd}\\
\text{\abbr{(GHKL)}} \quad \quad    K_{s,t}(x,y) &\ge 
\frac{C^{-1} \, \mu_{s,t}(y)}{\pi_s(\B(y,\sqrt{t-s}))}\exp\left(- \frac{C d(x,y)^{2}}{t-s}\right)
\,,
\label{ghk-lbd}
\end{align} 
where
$d(x,y)=d_{\bG}(x,y)$ and
$\B(y,r):=\{z\in V:\,d(y,z)\le r\}$ denote the 
graph distance in $\bG$ and the corresponding 
$\bG$-ball, respectively.
Compared to classical Gaussian heat kernel estimates (\abbr{GHKE}), the novel feature in \eqref{ghk-ubd} and \eqref{ghk-lbd} is the presence of the time dependent reference measures $\mu_{s,t}$. Indeed, a key difficulty for applying analytic methods to the time-varying conductance model is that the time-inhomogeneous transition probabilities $\{K_{s,t}\}$ do not admit any common invariant measure. 
While in principle one has the freedom to choose $\mu_{s,t}(\cdot)$,
we shall take
\begin{align} 
\mu_{s,t}(y):=(\pi_sK_{s,t})(y)=\sum_{x\in V}\pi_s(x)K_{s,t}(x,y), 
\qquad t\ge s\ge 0 \,.
\label{evol-meas}
\end{align}
These measures play the role of invariant measures
by satisfying the relation
\[
\mu_{s,t}(y) = \sum_{x \in V} \mu_{s,v}(x) K_{v,t} (x,y) 
\,, \qquad \forall s \le v \le t \,,
\]
with the \abbr{GHKU} and \abbr{GHKL} thus viewed as Gaussian heat kernel estimates of $\{K_{s,t}\}$ with respect to such evolving reference measure. 
We say that $\{\blG_t\}$ satisfies the \emph{uniform volume doubling (\abbr{vd})} condition, 
if 
\begin{equation}
\sup_{t,r \ge 0} \; \sup_{x \in V} \Big\{ \; \frac{\pi_t(\B(x,2r))}{\pi_t(\B(x,r))} \; \Big\} \le \Cd < \infty \,,
\label{vd}
\end{equation}
and further motivate our specific choice \eqref{evol-meas} by noting that for
$\{\blG_t\}$ satisfying such \abbr{VD} condition, if both \eqref{ghk-ubd} 
and \eqref{ghk-lbd} are to hold, then one \emph{must
take} for $\mu_{s,t}$ in the
\abbr{GHKU} and \abbr{GHKL}, 
up to a universal constant, the measures from \eqref{evol-meas}. 
(Indeed, for $t=s$ compare \eqref{ghk-ubd} and
\eqref{ghk-lbd} at $x=y$, while for $t>s$
bound $\pi_s K_{s,t}$ below using \eqref{ghk-lbd}
and the fact that $d(x,y)^2/(t-s) \le 1$ on 
$\B(y,\sqrt{t-s})$, then rely on \eqref{vd} at
$r_k = 2^k \sqrt{t-s}$, $k \ge L$, when bounding $\pi_s K_{s,t}$ above 
via \eqref{ghk-ubd}, 
to deduce that necessarily 
$(\pi_s K_{s,t})/\mu_{s,t} \in [c_\star^{-1},c_\star]$ 
for some universal $c_\star<\infty$).

In this work we mainly focus on the case that for some 
$\pi_0 \in \cM_+(V)$, the reversing measures of $\{K_t\}$ form a
\emph{pointwise non-decreasing} sequence $t \mapsto \pi_t(x)$, 
of positive functions on $V$.
It is a delicate issue that one must impose certain constraints on the measures $\pi_t$: 
for if $t\mapsto\pi_{t}(x)$ are allowed to oscillate, 
then anomalous behavior may occur (cf. \cite{ABGK,HK,SZ3}). Our main result for \abbr{CSRW} is such two sided Gaussian estimates with respect to $\mu_{s,t}$, under uniform volume doubling and Poincar\'e inequalities. Necessary definitions are listed as follows.

\begin{defn}\label{def-const}
We say that the \emph{uniform Poincar\'{e} inequality} holds 
for $\pi_t$-reversible $K_t$, if  
\begin{align}
\inf_{f_\star \in \R} \Big\{ \sum_{x\in\B(x_{0},r)}|f(x) - f_{\star}
|^{2} 
\pi_t(x) \Big\} 
\le \CP\, r^{2}\sum_{x,y\in\B(x_{0},2r)}(f(x)-f(y))^{2}K_t(x,y)\pi_t(x)  \,,\label{poincare}
\end{align}
for some $\CP<\infty$ and all $f:V\to\R$, $x_{0}\in V$, $t,r \ge 0$. 
\newline
$\bullet$ The \emph{uniform volume growth with $v(r)$ doubling} applies
for $\bG=(V,E)$ and $\pi_t : V \to \R_+$, 
if for some $\CV<\infty$,
\begin{align}
\CV^{-1} \, \le \frac{\pi_t(\B(x,r))}{v(r)} \le \CV\,, \qquad 
\forall x\in V, \;\; \forall r, t \ge 0 \,,
\label{d-set}
\end{align}
where $v(\cdot)$ is non-decreasing, $v(2r) \le \CV \, v(r)$ 
and $v(0) = v(1)=1$.  In particular, then 
$\pi_t(x) \in [\CV^{-1},\CV]$ for all $t,x$ 
and  the uniform \abbr{vd} condition holds
(with $\Cd=\CV^3$ in \eqref{vd}).
\newline
$\bullet$ 
We call Markov kernels $\{K_t\}$  
\emph{uniformly lazy} and weighted graphs $\{\blG_t\}$ \emph{uniformly elliptic},
if respectively,
\begin{align}
\alpha_l &:= \inf_t \;\; \inf_{x \in V} \;\; \big\{ K_t(x,x) \big\} > 0 \,,
\label{dfn:unif-lazy}
\\
\label{dfn:unif-elliptic}
\alpha_e &:= \inf_t \inf_{(x,y) \in E} \big\{ K_t(x,y) \big\} > 0\,,
\end{align}
where if $\{\blG_t\}$ is uniformly lazy, as in \eqref{dfn:unif-lazy},
then in particular $(x,x) \in E$ for all $x \in V$. For
uniformly 
elliptic and lazy $\{\blG_t\}$, set $\bar \alpha:=\alpha_l \wedge \alpha_e$
(with $\bar \alpha=\alpha_e$ when \abbr{csrw} concerned).
\end{defn}

\begin{thm}\label{thm-main-CSRW} \emph{[two-sided Gaussian estimates for \abbr{CSRW}]}
\newline
Consider \abbr{CSRW} associated with \eqref{n-kernel} such that
$t \mapsto \pi_t(x) \in \cM_+(V)$ is non-decreasing. 
%
Assume, in the sense of Definition \ref{def-const}, that $\{\blG_t\}$ is uniformly elliptic with constant $\bar\alpha$, of uniform volume 
growth $v(r)$ with doubling constant $\CV$, satisfying the uniform Poincar\'{e} inequality with constant $\CP$. 
Then there exists
$\Cbu = \Cbu (
\CV,\bar \alpha,\CP)$ finite,
such that for $\mu_{s,t}$ of \eqref{evol-meas}, 
and all $t -s \ge d(x,y)$,
\begin{align}\label{ghke}
\frac{\Cbu^{-1} \mu_{s,t}(y)}{v(\sqrt{t-s})} \exp\Big(
-\frac{\Cbu d(x,y)^2}{t-s} \Big)
\le K_{s,t}(x,y) \le \frac{\Cbu \mu_{s,t}(y)}{v(\sqrt{t-s})}
\exp\Big( - \frac{d(x,y)^2}{\Cbu (t-s)}\Big) \,.
\end{align}
\end{thm}

\medskip
Theorem \ref{thm-main-CSRW} is a direct consequence of 
Theorem \ref{harnack} and the detailed Theorem \ref{thm-main},
with the latter also providing 
our results in the more technically involved discrete time setting.
Note that if $\pi_t(\cdot)=\pi_0(\cdot)$ is  
independent of $t$, then also $\mu_{s,t}(y)=\pi_0(y)$ 
for all $t \ge s$ (see Remark \ref{mon-mu-m-n}), 
and our Gaussian transition density estimates 
take the usual form 
of the time-homogeneous setting. More generally, the same applies
whenever $\{\mu_{s,t}\}$ of \eqref{evol-meas} are $c$-stable 
with respect to the strictly positive 
$\sigma$-finite 
measures $\pi_{s}$ on $V$, as in \cite[Defn 1.10]{SZ2}. That is, 
whenever for some $c$ finite
\begin{equation}
c^{-1}\le\frac{\mu_{s,t}(y)}{\pi_{s}(y)}\le c \,, \qquad 
\forall t\ge s \ge 0, \quad 
y\in V\,.\label{eq:stab}
\end{equation}
Subject to such $c$-stability, considering \eqref{ghk-ubd}--\eqref{ghk-lbd} 
for $y=x$ yields that 
\begin{align}\label{rec:goe-cnd}
\int_s^\infty K_{s,t}(x,x) dt  < \infty \qquad  \Longleftrightarrow \qquad 
\int_{0}^\infty \frac{dt}{\pi_s(\B(x,\sqrt{t}))} \, < \infty \,.
\end{align} 
Starting at $X_s=x$, the \abbr{lhs} of \eqref{rec:goe-cnd} amounts 
to a finite expected total occupation time of state $x$, hence its transience 
for either \abbr{csrw} or \abbr{dtrw}
$\{t \mapsto \blG_t\}$. If in addition
$\sup_s \{h_s\} > 0$ implies $\inf_s \{h_s\} > 0$ for 
$h_s :=\P(X_{t} \ne x, \, \forall t>s | X_s =x)$,
it thereby answers the transience versus recurrence question raised in \cite{ABGK}.
We believe that 
$\{\mu_{s,t}\}$
of Theorem \ref{thm-main-CSRW} (and for \abbr{DTRW} as in Theorem \ref{thm-main})
are
all within a uniform constant of $\pi_{0}$. That is,
\begin{conj}\label{conj-stab}
If $\{\blG_t\}$, of non-decreasing 
$t \mapsto \pi_t(x)$, is uniformly elliptic, of
uniform volume growth $v(r)$ with $v(r)$ doubling, and satisfies the 
uniform Poincar\'{e} inequality, then for 
the corresponding \abbr{csrw} or uniformly lazy \abbr{DTRW}, 
$\inf_{t,x} \{\mu_{0,t}(x)\} \ge 1/c\left(\CP,\CV,\bar \alpha\right)>0$.
\end{conj}

\medskip
In proving Theorem \ref{thm-main-CSRW}, our steps are {\bf (I)} first establishing an on-diagonal upper bound (i.e. \abbr{GHKU} for $x=y$), 
{\bf (II)} proving the full \abbr{GHKU}, but without the term $\mu_{s,t}$, {\bf (III)} establishing a parabolic Harnack inequality (see Theorem \ref{harnack}), and 
{\bf (IV)} obtaining the two sided estimates stated from the Harnack inequality 
and a-priori weaker Gaussian upper bound. 
This road-map is well-established in the literature of heat kernel estimates. However,
in the time-varying setting each step requires overcoming difficulties brought by the changing conductances, in particular, by the lack of reversibility. We discuss below in more details our methods for each step.

Our approach to on-diagonal upper bound is through the so called the Nash profile (for its definition see \eqref{nash-profile}). In the time homogeneous setting, Coulhon \cite{C1,C2, C3} systematically derives 
sharp upper bound on $\left\Vert K_{0,n}\right\Vert _{1\to\infty}$
out of Nash type inequalities.  
Our on-diagonal upper-bound, namely the \abbr{rhs} of \eqref{ghk-ubd}
for $x=y$, is a special case of the general framework of 
Section \ref{sect:nash}, where time dependent reference measures 
such as $\pi_n$ or $\mu_n:=\mu_{0,n}$ of \eqref{evol-meas},
are used to obtain these upper bounds from the corresponding Nash profiles. 
We present this method, which is of independent interest, for the case of 
\abbr{DTRW}. It can be worked out directly for \abbr{CSRW}. Alternatively, 
the \abbr{CSRW} results can be deduced from the corresponding ones for
\abbr{DTRW}, as we do here.

Recall that for Markov operator $K$ on $V$ and its 
invariant $\sigma$-finite measure $\pi$, the map
\begin{equation}\label{def:Kf}
f \mapsto (Kf)(x):=\sum_{y\in V} K(x,y)f(y) \,,
\end{equation}
satisfies $(K f)^2 \le K f^2$ for bounded $f$, thereby 
extending to non-expanding map on $L^2(\pi)$ having the 
non-negative definite Dirichlet form 
\begin{align*}
\cE_{K,\pi}(f,g):=\langle f- Kf,g\rangle_{\pi} 
\end{align*}
with $L^2(\pi) \subseteq \mbox{Dom} (\cE_{K,\pi})$. The Nash profile of such  
$(K,\pi)$ is 
\begin{align}
\cN_{K,\pi}(\st):=\sup\left\{ \frac{\|f\|_{L^{2}(\pi)}^{2}}{\cE_{K,\pi}(f,f)}:
0<\|f\|_{L^{1}(\pi)}^{2}\le \st \|f\|_{L^{2}(\pi)}^{2} < \infty \right\} .
\label{nash-profile}
\end{align}
Since $\st \mapsto \cN_{K,\pi}(\st)$ is non-decreasing,  
setting $\cN_{K,\pi}(\infty)=\infty$ yields Nash inequality 
\begin{align*}
\|f\|_{L^{2}(\pi)}^{2}\le\cN_{K,\pi}\left(\|f\|_{L^{1}(\pi)}^{2}/\|f\|_{L^{2}(\pi)}^{2}\right)\cE_{K,\pi}(f,f)\,,
\end{align*}
for any non-zero $f \in L^2(\pi)$. Further, by Cauchy-Schwarz inequality,
\[
\|f\|_{L^{1}(\pi)}^{2}\le\pi(\mbox{supp}f)\|f\|_{L^{2}(\pi)}^{2}
\]
so $\cN_{K,\pi}(\st)$ plays the role of $L^{2}$-isoperimetric profile, 
where $\st$ acts as the volume. 
 
With the uniform Poincar\'e inequality providing an explicit upper bound on 
the Nash profile of weighted graph $\blG_t$ in terms of the
doubling function $v(r)$ of 
\eqref{d-set} (see our derivation of \eqref{eq:nash-profile-vd}),
the application of Section \ref{sect:nash}
most relevant here is as follows
(c.f. Lemma \ref{P+VG-to-Nash}).
\begin{thm}\label{nash-inc} 
Suppose $K_t$ have reversible measures $\pi_t \in \cM_+(V)$
with $t \mapsto \pi_t(x)$ non-decreasing for each $x \in V$
and the non-decreasing $\sss \mapsto N(\sss)$ is such that 
for some finite $\CN$, $\sss_0$,
\begin{equation}\label{eq:nash-profile}
N(\sss) \ge \sup_t \{ \cN_{K_{t}^2,\pi_{t}}(\sss) \} \,, \qquad
\inf_{\sss \ge \sss_0} \Big\{ \frac{N(\CN \, \sss)}{N(\sss)} \Big\} \ge 2  \,.
\end{equation}
(a). For the dynamics \eqref{def:k-mn} and any $s \le t$ 
one has the on-diagonal upper bound 
\begin{equation}\label{eq:diag-ubd-nash}
\sup_{x,y\in V}\Big\{ 
\frac{K_{s,t}(x,y)}{\pi_{t}(y)} \Big\} \le \CN' \psi\Big(\frac{t-s}{3}\Big)\,,
\end{equation}
for $\psi(\st):=1/F^{-1}(\st;c_\star,N(\cdot))$, 
and some $\CN'=\CN'(\CN,\sss_0/c_\star) < \infty$, where
$c_\star = \underline{\pi_0}$ and 
\begin{equation}\label{eq:N-psi}
F(u;a,N(\cdot)) := \int_a^u \frac{N(\sss)}{\sss} \, d\sss \,.
\end{equation}
(b). For the 
dynamics \eqref{def:k-st} 
replace $\cN_{K_t^2,\pi_t}(\sss)$ by $2 \cN_{K_t,\pi_t}(\sss)$ 
in \eqref{eq:nash-profile}, with the \abbr{rhs} of 
\eqref{eq:diag-ubd-nash} having the expectation 
over $\frac{1}{3}$Poisson$(2(t-s))$ law of the corresponding $\psi(\cdot)$.
\end{thm}

We next remark on related works \cite{SZ1,SZ2} and \cite{DHMP}.
General time-inhomogeneous transitions \{$K_{k,n}(x,y)$\} 
that satisfy \eqref{def:k-mn} for some finite state space $V$,
are considered in \cite{SZ1,SZ2}. Aiming at merging 
for such transitions, namely the suitable convergence 
to zero of $|K_{0,n}(x,y)-K_{0,n}(x',y)|$ as $n \to \infty$,
\cite{SZ1,SZ2} develop in this context analytic tools 
such as the Nash and log-Sobolev inequalities, where a key 
assumption of \cite{SZ2} is that the Markov transitions $\{ K_{0,n} \}$ 
yield $\mu_{0,n}$ as in \eqref{evol-meas} which are $c$-stable 
with respect to some $\pi_0 \in \cM_+(V)$. 
The $c$-stability condition \eqref{eq:stab}
is in general difficult to verify, as it requires fine understanding 
of the transition probabilities $K_{s,t}$.
Theorem $\ref{nash-inc}$, and more generally Proposition \ref{nash-inc-gen}, are phrased in a general framework that allows more flexible choices of reference measures, see Example \ref{stab-ex} and \ref{inc-ex}. In particular, 
under the non-decreasing assumption on $t\mapsto\pi_t(x)$, we do not 
require $c$-stability. The method of Nash profiles applies also 
when $N(\cdot)$ 
grows too slowly for the \abbr{rhs} of \eqref{eq:nash-profile}, 
for example when $N(e^{\sss})$ is a doubling function
(see Remark \ref{regularity}).

Under the same assumption that $n\mapsto\pi_{n}(x)$ is non-decreasing, 
evolving sets are used in \cite{DHMP} for deriving the heat 
kernel on-diagonal upper bound (ie. \abbr{ghku} for $x=y$)
for uniformly lazy \abbr{dtrw}
from $L^{1}$-isoperimetry property of $\{\blG_n\}$. 
\noindent
Via a different approach, Theorem {\ref{nash-inc}} strengthens the main result 
of \cite{DHMP}, see Example \ref{inc-diag-ubd}. More precisely, we recover the same on-diagonal upper bound 
in the more general setting of \eqref{def:k-mn}, while  
replacing the assumed $\kappa_m \sss^{-1/d}$ lower bound on the 
$L^1$-isoperimetric profile for the weighted graph $\blG_m$,
with having only 
$\cN_{K_m^2,\pi_m}(\sss) \le \kappa_m^{-2} \sss^{2/d}$
(c.f. Lemma \ref{N-L} for a comparison between the 
Nash and $L^1$-isoperimetric 
profiles which follows from the Cheeger inequality).
We note in passing that Theorem \ref{nash-inc} can 
even be applied to certain non-local Markov transition 
kernels, as demonstrated by Example \ref{nonlocal}.

\medskip
Due to lack of reversibility for $K_{s,t}$ when $t>s$, 
the off-diagonal upper bound 
is technically more involved.
Nevertheless, in Section \ref{sect:gauss} we adapt the technique of \cite[Section 2]{HS} for deriving
off-diagonal Gaussian upper bounds via complex interpolation. 
The complex interpolation method requires two input bounds: a bound on the 
$2\to\infty$ norm of $K_{s,t}$ and a bound on the $2\to 2$ norm of the perturbed kernel $K^\theta_{s,t}$ (see \eqref{eq:Ks}) with respect to appropriate 
reference measures, see Proposition \ref{prop:hs}. The $2\to\infty$ norm bound is provided by the Nash profile method 
in Section \ref{sect:nash}. The bound on $2\to2$ norm is often referred to as the Gaffney lemma, which we prove 
for time non-decreasing $\pi_t$ in Lemma \ref{gaffney} (based on 
Lemma \ref{gaffney-ineq}). The complex interpolation method is especially 
suited for our purpose, since it never invokes reversibility and
incorporates well having time dependent reference measures $\pi_s$ (for each
term $K_s$).

\medskip
So far what we have discussed applies equally well to \abbr{CSRW} and \abbr{DTRW}. 
The \abbr{GHKL} turns out to be more difficult without reversibility, even if 
just to obtain an on-diagonal lower estimate. Specializing the setting 
of Theorem \ref{nash-inc}, back to that of weighted graphs $\{\blG_t\}$ 
that satisfy the uniform Poincar\'{e} inequality and uniform volume doubling,
our approach is thus to first
establish a Parabolic Harnack Inequality (\abbr{phi}), then derive the full two-sided \abbr{GHKE} from it.
To this end, we first introduce the notations needed for stating
such parabolic Harnack inequality. 
We call $u(\cdot,\cdot) \ge 0$ on a \emph{time-space cylinder}
\begin{equation}\label{dfn:cyl}
Q:= Q(t_{1},t_{2};z,R)=\left[t_{1},t_{2}\right]\times\B(z,R)
\end{equation}
a (non-negative) solution to the 
(backward) heat equation, if 
\begin{equation}\label{dfn:sol-cyl}
\partial_{-s} u(s,x)=\sum_{y} K_s(x,y) u(s,y) - u(s,x), 
\qquad \forall (s,x) \in Q \,,
\end{equation}
for some non-negative boundary values (for $u$) outside $Q$. For bounded
range $K_s$ we have that \eqref{dfn:sol-cyl} holds on 
$\B(z,R-r_0)$ even when restricting the sum to $\B(z,R)$.
Here $r_0=1$, so such solution is uniquely 
specified by 
$\{ u(s,x) : d(z,x)=R$ or $s=t_2\}$.
For \emph{discrete time} we take $s \in \N$ and 
$\partial_{-s} u(s,\cdot):=u(s-1,\cdot)-u(s,\cdot)$, whereas 
for \abbr{csrw} we assume \abbr{wlog} 
that $s \mapsto u(s,x)$ is absolutely continuous, so
$\partial_{-s} u$ 
exists a.e. and \eqref{dfn:sol-cyl} 
interpreted as a distributional identity 
via integration by parts.

We say that the \abbr{PHI} holds for \eqref{dfn:sol-cyl},
if for any $0<\theta_1 \le \theta_2<\theta_3 \le \theta_4$ 
some $\gamma=\gamma(\theta_i) \in (0,1)$,
any $T \ge (\theta_4 R)^2$ and 
solution $u$ of \eqref{dfn:sol-cyl} 
on time-space cylinder $Q(T-(\theta_4 R)^{2},T;z,8 R)$, we have 
\begin{equation}\label{eq:harnack}
(\theta_{2i-1} R)^2 <  \tau_i \le (\theta_{2i} R)^2, 
x_1,x_2 \in \B(z,R) \quad 
\Longrightarrow \quad 
u(T-\tau_2,x_2) \ge \gamma u(T-\tau_1,x_1) \,,
\end{equation}
further restricting \eqref{eq:harnack} in the discrete case 
to $\tau_2 \ge \tau_1 + d(x_1,x_2)$.
\begin{remark}\label{rmk:str-pos}
If $u(\cdot,\cdot)$ satisfies 
\eqref{dfn:sol-cyl} 
then so does $a u(\cdot,\cdot) + b$.
Considering $b \downarrow 0$ we deduce that it suffices 
to prove the \abbr{PHI} only for strictly positive solutions.
\end{remark}

Recall \cite{Del} that for \abbr{CSRW} on time-invariant conductances, 
the \abbr{PHI} is equivalent to uniformly elliptic conductances satisfying 
both the \abbr{PI} and \abbr{VD}. Our next result extends this to 
time-varying, non-decreasing vertex conductances $t \mapsto \pi_t(\cdot)$. 
\begin{thm}\label{harnack}\emph{[parabolic Harnack inequality]}
\newline
Suppose $\{\blG_t\}$ of 
non-decreasing $t \mapsto \pi_t(x) \in \cM_+(V)$ and 
$\CB := \sup_{t,x} \{\frac{\pi_t(x)}{\pi_0(x)}\}$ finite,
is uniformly elliptic, satisfying the uniform volume doubling 
condition and the uniform Poincar\'{e} inequality.
Then, the \abbr{PHI} holds for the 
continuous time heat equation 
\eqref{dfn:sol-cyl} and some $\gamma=\gamma(\CP, \Cd, \alpha_e, \CB)$ positive.
\end{thm}
 
\smallskip
For time-invariant conductances the \abbr{PHI} implies 
H\"older regularity of 
(non-negative) solutions of the heat equation
(see \cite[Pages 227-228]{Del}). This extends to our setting, 
yielding the 
H\"older regularity of $(s,x) \mapsto K_{s,t}(x,z)$ 
under the conditions of Theorem \ref{harnack}.
\begin{ppn}\label{ppn:holder}
The \abbr{PHI} implies existence of 
$h(\gamma)>0$ such that for any 
$z\in V$, $R \ge 1$, $T \ge 4 R^2$ and 
solution $u \ge 0$ of
\eqref{dfn:sol-cyl} on $Q=Q(T-4R^2, T; z, 8R)$, 
if $y_j \in \B(z,R)$ and $(T-s_j) \in [R^2,4R^2]$, $j=1,2$, then  
\begin{align}\label{eq:holder}
|u(s_2,y_2)-u(s_1,y_1)|\le (4/R)^{h}
(|s_2-s_1|^{1/2} \vee d(y_1,y_2))^h 
\sup_Q \, \{u\}
\,.
\end{align}
\end{ppn}

Parabolic Harnack inequalities, Gaussian estimates and H\"older regularity of solutions of the heat equation have a long history.
Aiming at a-priori H\"{o}lder continuity for solutions of the heat equation
\begin{equation}
\label{eq:he-rd}
\partial_{t}u(t,x)=\cL u(t,x), \qquad t \ge 0, x \in \M \subseteq \R^d \,,
\end{equation}
on a Riemannian manifold $\M$, with a divergence form operator 
$$
\cL u :=\sum_{i,j=1}^{d}
\partial_{x_{i}} \Big( a_{ij}(t,x) \partial_{x_{j}} u \, \Big) 
$$
having symmetric, measurable, uniformly elliptic matrix of 
coefficients $\{ a_{ij}(\cdot) \}$,
the study of heat kernel asymptotics for the corresponding diffusion on $\M$,  
goes back at least to works of De Giorgi, Nash, Moser in mid-century.
The characterization of two-sided Gaussian Heat Kernel 
Estimates (\abbr{GHKE}), for the solutions of \eqref{eq:he-rd} 
(namely, the diffusion analog of \abbr{vsrw}),
in terms of Poincar\'e Inequality (\abbr{PI}), plus the volume doubling (\abbr{vd})
property, and their equivalence
to the Parabolic Harnack Inequality (\abbr{PHI}),
are established independently by \cite{Gr,SC}. Such results 
have later been derived in 
\cite{St1,St2} for time-dependent, strongly local Dirichlet 
forms on metric measure spaces (subject to the existence of a 
time-invariant Radon measure in the underlying topological space).
However, strongly local Dirichlet forms as in \cite{St1,St2} can have 
no jumps (nor killing). In particular, this assumption excludes
the uniformly elliptic (and lazy)
random walks on a (static) graph $\blG_0$, for which such
equivalence between \abbr{GHKE}, \abbr{PHI} and \abbr{PI+VD}
is proved in \cite{Del}. See also \cite{BC}, which proves a 
similar equivalence for \abbr{csrw} on non-elliptic (static) 
graph $\blG_0$, when the \abbr{GHKE}, \abbr{PHI} and 
\abbr{PD+VI} are suitably restricted (to large balls).
One direction we pursue here, is to extend this
graph part of the theory, by obtaining the
\abbr{ghku} (with $\mu_{s,t}$ replaced by 
$\pi_t$), for both dynamics of \eqref{def:k-mn}
and \eqref{def:k-st},
allowing for genuinely time-varying, non-decreasing
$\{\pi_t(x): x \in V\}$.
In a related context, 
the two-sided Gaussian heat kernel estimates are already provided 
in \cite{DD,GOS} for continuous-time symmetric rate random walks 
on $\bZ^d$ having time-dependent, uniformly elliptic 
jump rates $c_t(x,y)$ (i.e. the so-called variable 
speed random walk \abbr{VSRW}; c.f. \cite{MO} for the same 
in certain degenerate cases lacking uniform ellipticity). Indeed, 
the treatment of time-varying \abbr{VSRW} is much 
simpler than both \abbr{DTRW} and \abbr{csrw} 
since any \abbr{vsrw} has the time-invariant reversing measure 
$\{ \pi_t(x)=1: x \in V, t \ge 0 \}$. Similar reversible 
situation applies in \cite{CGZ} where two-sided Gaussian heat 
kernel bounds are stated (without a detailed proof), 
for the \abbr{dtrw} of \eqref{def:k-mn}, 
provided $\{\pi_{n}(x)=\pi_0(x)\}$ is constant in time 
and a uniform Sobolev inequality holds.

The approach to establish \abbr{PHI} from volume doubling and Poincar\'{e} 
inequalities, as in \cite{Gr} and \cite{SC} (that we adapt for
proving Theorem \ref{harnack} in Section \ref{sect:harnack}), relies on taking the time derivative 
of the logarithm of the heat kernel. Having a discrete-time 
version of such a step, is a well known open challenge. 
This difficulty can be circumvented by first deriving the \abbr{HKE}-s and then deducing
the \abbr{PHI} from them, see \cite{FS}. Indeed, for time-invariant conductances, as in \cite{Del}, one compares the transition probabilities of 
\abbr{DTRW} to those of the \abbr{CSRW}, thereby obtaining the Gaussian estimates for the \abbr{DTRW}, which in turn yield the \abbr{PHI}. However,
such a comparison with the \abbr{CSRW} is not available in our
time-varying setting. Alternatively, in \cite{HS2}, 
Hebisch and Saloff-Coste prove the \abbr{PHI} 
for discrete-time dynamic, directly, 
from a scale invariant elliptic Harnack inequality and local Sobolev inequalities. Unfortunately, it is unclear what should be the analogous
elliptic objects to study in the time-varying setting. For these reasons, we are only able to establish parabolic Harnack inequalities in the case of \abbr{CSRW}.

We summarize our main results for both \abbr{CSRW} and \abbr{DTRW} as follows. 
As mentioned before, under the assumption of uniform Poincar\'e inequalities, uniform ellipticity, and uniform volume growth doubling conditions, we derive a \abbr{GHKU} without the term $\mu_{s,t}$
for both \abbr{CSRW} and uniformly lazy \abbr{DTRW}; then the improved \abbr{GHKU} 
with respect to $\mu_{s,t}$ and the
matching \abbr{GHKL} are both obtained as a consequence of the \abbr{PHI}. 
\begin{thm}\label{thm-main} \emph{[two-sided Gaussian \abbr{hke}]}
\newline
Consider either \abbr{CSRW}
or a uniformly lazy \abbr{DTRW} associated with \eqref{n-kernel},
for non-decreasing $t \mapsto \pi_t(x)$ and $\{\blG_t\}$ of uniform volume 
growth $v(r)$ with $v(r)$ doubling.\\
(a). Let $I(r)=r^2$ on $[0,1]$, and for $r>1$ set
$I(r)=r (\log r + 1)$ 
for the \abbr{csrw}, while $I(r)=\infty$ for 
the \abbr{dtrw}. Then, the uniform Poincar\'{e} inequality 
yields that for some finite $C=C(\CP,\CV,\alpha_l)$,
\begin{equation}\label{ghku-pi}
K_{s,t}(x,y) \le
\frac{C}{v(\sqrt{t-s})} \exp \Big\{- 
\frac{(t-s)}{C} I\big(\frac{d(x,y)}{t-s}\big)\Big\} 
\,,
\qquad \forall x,y \in V, \; t \ge s \ge 0 \,.
\end{equation}
(b) Suppose \eqref{ghku-pi} and the 
\abbr{PHI} hold (so for \abbr{DTRW}, the graphs $\{\blG_t\}$ are uniformly elliptic
and lazy).
Then, for $\mu_{s,t}$ of \eqref{evol-meas}, some 
$\Cbu = \Cbu (C,\gamma,\CV,\bar \alpha)$ finite
and all $t -s \ge d(x,y)$,
\begin{align}\label{ghke}
\frac{\Cbu^{-1} \mu_{s,t}(y)}{v(\sqrt{t-s})} \exp\Big(
-\frac{\Cbu d(x,y)^2}{t-s} \Big)
\le K_{s,t}(x,y) \le \frac{\Cbu \mu_{s,t}(y)}{v(\sqrt{t-s})}
\exp\Big( - \frac{d(x,y)^2}{\Cbu (t-s)}\Big) \,.
\end{align}
\end{thm}

\begin{remark}\label{mon-mu-m-n}
In the discrete time setting of \eqref{def:k-mn},
if $t \mapsto\pi_t (x)$ is non-decreasing at each $x \in V$, 
then pointwise $\pi_s K_s = \pi_s \le \pi_{s+1}$, hence   
$\mu_{s,t} = (\pi_s K_s) K_{s+1,t} \le \mu_{s+1,t}$ for $\mu_{s,t}$
of \eqref{evol-meas}. That is,
\begin{equation}\label{eq:mon-mu-n}
\mu_t(x):=\mu_{0,t}(x) 
\le \mu_{s',t} (x) \le \mu_{s,t}(x) \le  \mu_{t,t}(x) = \pi_t(x) \,,
\quad \forall t\ge s\ge s'\ge 0,\;\; x \in V 
\end{equation}
To verify that the same applies in the setting of the 
continuous time evolution \eqref{def:csrw}, recall that 
$\mu_{s,t}$ is then the expected value over
$N\!\!\sim$Poisson($t-s$) and jump times
$s < T'_{1} < \cdots < T'_N \le t$ of the
value $\mu^{(\omega)}_{0,N}$ for a discrete time dynamic 
starting at $\pi_s$ and using the random 
$\{K_{T'_{m}}\}$ in \eqref{def:k-mn}.
With $s' \in (T'_L,T'_{L+1}) \cap (s,t)$ for some 
$0 \le L \le N$, clearly $\mu_{s',t}$ 
exceeds the expected value of $\mu^{(\omega)}_{L,N}$ for the 
corresponding discrete time dynamic, so 
by the monotonicity of the expectation, \eqref{eq:mon-mu-n} 
applies also for any continuous time evolution \eqref{def:csrw}
with non-decreasing  $t \mapsto \pi_t(\cdot)$.
By the same reasoning, for both \abbr{dtrw} and \abbr{csrw},
if $t \mapsto \pi_t(x)$ is 
non-increasing, then so is $s\mapsto \mu_{s,t} (x)$.
In particular, for 
the special case of $\pi_t(x)=\pi(x)$ independent of $t$
we have that $\mu_{s,t}=\pi$ and Theorem \ref{thm-main} recovers 
(under uniform Poincar\'e inequality and uniform volume growth $v(r)$ with $v(r)$ doubling), 
the Gaussian upper bound for \abbr{DTRW} stated in \cite[Sec. 7]{CGZ}.
\end{remark}

In view of \eqref{eq:mon-mu-n}, upon verifying Conjecture \ref{conj-stab}  
the \abbr{rhs} of \eqref{rec:goe-cnd} should provide a criterion 
for transience/recurrence of \abbr{csrw} in terms of the volume 
growth of $\blG_0$ (and upon proving the discrete time \abbr{PHI},
the same would apply for uniformly elliptic and lazy \abbr{DTRW}).

\begin{remark} Without monotonicity of $n \mapsto \pi_n$, even for
$\{\pi_n\}$ that are $c$-stable \abbr{wrt} the function
$\nu_0 (x) \equiv 1$ on $\bG=\bZ_{\ge0}$, the reference
$\mu_n=\mu_{0,n}$ may be non-comparable with $\pi_n$. For example, 
fixing $\eta,\ep>0$ let $\pi_n(x,x+1)=1 + (-1)^{n+x} \eta$ 
with $\pi_n(x,x)=1+\ep \I_{\{n+x \text{ odd}\}}$ when $x>0$  
and $\pi_n(0,0)=\pi_n(2,2) \pi_n(0,1)/2$
(to assure that $K_n(0,0)=K_n(2x,2x)$ for any $n,x$).
Classifying states into types $A$ or $B$ according to 
$n+X_n$ being even or odd, respectively, yields an
$\{A,B\}$-valued homogeneous Markov chain  
of invariant measure $[m_A,m_B]=[3(1+\ep),3+\ep]$. The 
process $\{X_n\}$ has drift $\eta$ at the $A \mapsto A$ 
moves with opposite drift at $B \mapsto B$ moves. Consequently, 
$\{X_n\}$ has asymptotic speed $v=\eta\ep/(3+2\ep)$ to the right.
In particular, for some $C$ finite and any $y\in\Z_{\ge0}$ 
we have for all $n \ge C y/v$, the fast decay 
$$ 
\mu_{n}(y)=\sum_{x\ge 0} K_{0,n}(x,y) 
\le C\sqrt{n}e^{-(nv-y)^2/(Cn)} \,.
$$
\end{remark}

In Section \ref{sect:proof-ppn} we treat a perturbative regime. Specifically, we show that the \abbr{GHKU} 
of Theorem \ref{nash-inc} and Theorem \ref{thm-main}(a) 
apply as soon as $t \mapsto e^{a_t} \pi_t(\cdot)$ is non-decreasing 
for some 
non-decreasing $t \mapsto a_t$ such that
\begin{align}\label{cond:pert}
A := \sup_{t\ge 0}  \{ a_{2t+1} - a_t \}  < \infty \,.
\end{align}
Further, we get the matching \abbr{GHKL} 
if \eqref{cond:pert} applies for
\begin{align}\label{pertur-regime}
a_t =\sup_\ell \sup_{0=s_0<\cdots<s_\ell=t} \;\; \big\{ \sum_{i=0}^{\ell-1} 
\rho_{\bpi} (s_i,s_{i+1}) \big\} \,, \qquad 
\rho_{\bpi}(s,s') :=
\sup_{x \in V} \Big|\log \frac{\pi_{s'}(x)}{\pi_{s}(x)} \Big| 
\end{align}
(considering for \abbr{dtrw} only $s_i \in \N$). In particular,
for \abbr{csrw} with absolutely continuous $s \mapsto \pi_s(x)$ we 
have in \eqref{pertur-regime} 
absolutely continuous $a_t$ such 
that a.e. $\partial_t a_t = \sup_{x \in V} |\partial_t \log \pi_t(x)|$.
\begin{ppn}\label{perturbative}
Suppose $\{\blG_t\}$ uniformly elliptic of uniform volume growth $v(r)$ with $v(r)$ doubling, has the uniform Poincar\'{e} inequality and 
some $\{a_t\}$ satisfies \eqref{cond:pert}--\eqref{pertur-regime}.
\newline
(a) The \abbr{GHKU} holds for either \abbr{csrw} or 
uniformly lazy \abbr{dtrw}, without 
$\mu_{s,t}(\cdot)$, 
in \eqref{ghk-ubd}, 
and with some $C=C\left(A,\CP,\CV, \bar \alpha\right)$ finite.   \\
(b) The matching \abbr{GHKL} holds for \abbr{CSRW}, and 
subject to the discrete time \abbr{phi}, also for 
\abbr{dtrw}.
\end{ppn}
\begin{remark} 
Starting at a uniformly elliptic $\blG_0$ of 
volume growth $v(r)$ with $v(r)$ doubling that 
satisfies the Poincar\'{e} inequality, 
Proposition \ref{perturbative} yields the 
matching \abbr{ghke} for the \abbr{csrw} on
$\pi_t(x,y) = \pi_0(x,y) e^{h_t(x,y)}$,
whenever $\sup_{t} \{ \|h_t\|_\infty \}$ and 
$\sup_t \{ (t+1) \|\partial_t h_t\|_\infty \}$ are finite.
In particular, this setting allows us to have \emph{forever oscillating} 
$t \mapsto \pi_t(x)$.
\end{remark}
While in Proposition \ref{perturbative} we have $a_n = O(\log n)$,
we next show that no such Gaussian estimates hold universally 
when $a_n$ grows as $O(n^{1/2+\iota})$ for some $\iota>0$. It is 
interesting to find a sharp threshold in the context of Proposition \ref{perturbative}, and in particular to determine 
whether $a_n \le O(n^{1/2})$ suffices for such Gaussian density bounds.

\begin{ppn}\label{ppn:conv} 
\noindent For any $\iota>0$, there exist uniformly bounded,
uniformly elliptic, and uniformly lazy, time-varying 
edge-conductances on $\bZ$, with 
\begin{align}
\limsup_{n\to\infty} \, \{ n^{-(1/2+\iota)} a_n \} < \infty\,,
\label{conter-conv}
\end{align}
such that 
neither \eqref{ghk-ubd} nor \eqref{ghk-lbd} hold
for the corresponding \abbr{DTRW} $\{X_{n}\}$.
\end{ppn}

\medskip

The rest of the paper is organized as follows. 
Section \ref{sect:nash} explores a general framework using evolving reference measures for obtaining on-diagonal transition probability upper bounds from Nash
profiles of underlying graphs and can be read independently of the 
rest of this paper. Section \ref{sect:gauss} adapts to our time-inhomogeneous setting the perturbation-interpolation technique of \cite[Section 2]{HS} 
for deriving off-diagonal upper bounds (from a given on-diagonal upper 
bound), concluding with the \abbr{ghku} of Theorem \ref{thm-main}(a).
We establish in Section \ref{sect:harnack} the \abbr{phi}
of Theorem \ref{harnack} 
and the regularity estimate
of Proposition \ref{ppn:holder}.
Section \ref{sect:conseq} then 
complete the derivation of Theorem \ref{thm-main}, whereas
Section \ref{sect:proof-ppn} deals with the 
perturbative regime of Propositions \ref{perturbative} and \ref{ppn:conv}.

\end{section}

\begin{section}{Nash Inequalities} \label{sect:nash} 

Recall the definition of the Nash profile $\mathcal{N}_{K,\pi}$ as in \eqref{nash-profile}. 
Given a dynamic \eqref{def:k-mn} for 
Markov kernels $\{K_n\}$ the Nash method relies on finding auxiliary 
Markov kernels $Q_{n}$ reversible for some $\nu_{n} \in \cM_+(V)$, having 
useful Nash profiles $\cN_{Q_n,\nu_n}(\st)$ as well as
the following contraction properties.
\begin{assumption}\label{nash-assum} 
Markov kernels $Q_{n}$ are reversible for $\nu_{n} \in \cM_+(V)$ 
and for any $f \in L^2(\nu_n)$
\begin{equation}
\left\Vert K_{n}f\right\Vert _{L^{2}(\nu_{n-1})}^{2} + \cE_{Q_n,\nu_n}(f,f) 
\le \left\Vert f\right\Vert_{L^{2}(\nu_{n})}^{2} \,.
\label{eq:increment}
\end{equation}
Further, for any $f \in L^1(\nu_n)$ 
\begin{align}
\left\Vert K_{n}f\right\Vert _{L^{1}(\nu_{n-1})}\le\left\Vert f\right\Vert _{L^{1}(\nu_{n})}.\label{contraction}
\end{align}
In particular $K_n$ must be a bounded operator from 
$L^p(\nu_n)$ to $L^p(\nu_{n-1})$ for $p=1,2$.
\end{assumption}
We proceed to provide two canonical examples (of pairs $Q_n,\nu_n$), for 
which Assumption \ref{nash-assum} holds. 
\begin{ex}\label{stab-ex} 
If $\nu \in \cM_+(V)$ 
and Markov kernel $K$ are such that $\mu=(\nu K) \in \cM_+(V)$, 
then $f \mapsto Kf$ of \eqref{def:Kf} extends uniquely to the non-negative, 
bounded linear map $K_{\mu\to\nu}:L^{2}(\mu) \to L^{2}(\nu)$. Its dual 
$K_{\nu\to\mu}^{\star}:L^{2}(\nu)\to L^{2}(\mu)$ then satisfies
\begin{equation}\label{dfn:dual}
\left\langle h, K_{\nu \to \mu}^{\star} g \right\rangle _{\mu}
= \left\langle K h, g \right\rangle _{\nu},\quad \forall g\in L^{2}(\nu),
h \in L^{2}(\mu)\,,
\end{equation}
with the self-adjoined non-negative operator $Q=K^\star K$ such that
\begin{align}\label{eq:dual}
\left \langle h, Q f \right \rangle_{\mu} = \left \langle K h, K f 
\right \rangle_\nu ,
\quad \forall f,h \in L^2(\mu) \,.
\end{align}
Taking $h=\delta_x$ and $f=\delta_y$ in \eqref{eq:dual}, we further see that 
\begin{equation}\label{dfn:Q}
Q(x,y) := \frac{1}{\mu(x)} \sum_z \nu(z) K(z,x) K(z,y) \,,
\end{equation}
is a $\mu$-reversible, Markov transition kernel.
Further, for $f \in L^2(\mu)$ we have from \eqref{eq:dual} that
\begin{equation}
\cE_{Q,\mu}(f,f)= \left\Vert f\right\Vert_{L^{2}(\mu)}^{2}
- \left\Vert K f\right\Vert _{L^{2}(\nu)}^{2} \,.
\label{eq:l2-ident}
\end{equation}
Since $\mu=(\nu K)$ we also have that
\begin{align}
\left\Vert K f\right\Vert _{L^{1}(\nu)} 
 = \sum_x \nu(x) | \sum_y K(x,y) f(y) | & \le 
\sum_{x,y} \nu(x) K(x,y) |f(y)| \nonumber \\
& = \sum_y (\nu K)(y) |f(y)| =
\left\Vert f\right\Vert _{L^{1}(\mu)}.
\label{eq:l1-ineq}
\end{align}
Thus, Assumption \ref{nash-assum} holds for $\mu_{n}:=\mu_{0,n}$
of \eqref{evol-meas} and the corresponding $\mu_n$-reversible 
Markov kernels $Q_n := K_{\mu_{n-1} \to \mu_n}^{\star} K_n$, 
provided 
\begin{equation}\label{unif-elip}
\mu_n \in \cM_+(V)\,, \qquad \qquad n=0,1,2,\ldots  \,.
\end{equation}
Starting at any $\mu_0 \in \cM_+(V)$, one has \eqref{unif-elip} 
for uniformly lazy walks, where 
$\underline{\mu_n} \ge \widehat{\alpha} \, \underline{\mu_{n-1}}$ 
are strictly positive, since 
$$ 
\widehat{\alpha} := \inf_{n,y} \sum_x K_n(x,y) \ge \alpha_l 
$$
of \eqref{dfn:unif-lazy}, and by induction having per $n,y$ 
only finitely many $x \in V$ for which 
$K_n(x,y) > 0$ guarantees the finiteness of $\mu_n(y)$.

If $\{\mu_{0,n}\}$ are $c$-stable (see \eqref{eq:stab}), 
then similarly to the considerations of \cite{SZ2},
one may estimate the Nash profile $\cN_{Q_{n},\mu_n}(\st)$ 
in terms of say $\cN_{Q_1,\mu_1}(\cdot)$. However,
not withstanding Conjecture \ref{conj-stab}
we have no systematic way towards such $c$-stability, 
without which we have little control on  
$\cN_{Q_{n},\mu_n}(\st)$. 
\end{ex}

\begin{ex}\label{inc-ex} 
If the Markov kernel $K$ has an invariant measure $\pi \in \cM_+(V)$
then considering Example \ref{stab-ex} for $\nu=\pi$ results with
$\mu = (\nu K) = \pi$. Suppose now that $K_{n}$ have invariant measures
$\pi_n \in \cM_+(V)$ such that $n \mapsto \pi_n(x)$
are non-decreasing. Then $n \mapsto \|K f\|_{L^p(\pi_{n})}$ are
non-decreasing for $p=1,2$, hence from \eqref{eq:l2-ident}-\eqref{eq:l1-ineq} 
we deduce that Assumption \ref{nash-assum} holds for $\nu_{n}=\pi_{n}$.
If further $K_n$ is $\pi_n$-reversible, as in Theorem \ref{nash-inc}, 
then from \eqref{dfn:Q} we see that $Q_{n}=K_{n}^2$.
\end{ex}

In view of Example \ref{inc-ex}, part (a)
of Theorem \ref{nash-inc} is a 
special case of our next proposition dealing with the more general 
setting of Assumption \ref{nash-assum}. 

\begin{ppn}\label{nash-inc-gen}
Suppose in addition to Assumption \ref{nash-assum}, that for non-decreasing 
$\sss \mapsto N(\sss)$  
\begin{align}\label{eq:uniform-nash}
N(\sss) \ge \sup_k \{ \mathcal{N}_{Q_k,\nu_k}(\sss) \}  & \,, 
\qquad \qquad  
\inf_k \{ \,\underline{\nu_k} \,\} \ge c_\star > 0 \,,
\\
&
\inf_{\sss \ge \sss_0} \Big\{ \frac{N(\CN \, \sss)}{N(\sss)} \Big\} \ge 2 \,.
\label{eq:regular-nash}
\end{align}
Then the
bound \eqref{eq:diag-ubd-nash} holds
with $\nu_n$ instead of $\pi_n$. 
%
\end{ppn}

Turning to the proof of Proposition \ref{nash-inc-gen}, note that
fixing hereafter $\{\nu_n\} \subseteq \cM_+(V)$, part of Assumption 
\ref{nash-assum} is having bounded operators   
$K_n : L^2(\nu_n) \to L^2(\nu_{n-1})$, and 
hence the dual (adjoint) non-negative operators 
$$
(K_n)^\star_{\nu_{n-1} \to \nu_{n}}: L^2(\nu_{n-1}) \to L^2(\nu_n)
$$
as in \eqref{dfn:dual}, using   
$K_n^\star$ for $(K_n)^\star_{\nu_{n-1} \to \nu_{n}}$ whenever the choice of 
$\{\nu_n\}$ is clear. For any $0 \le m <n$, 
\begin{equation*}
K_{m,n}^{\star} := K_n^{\star} K_{n-1}^{\star} \cdots K^\star_{m+1} \,,
\end{equation*}
is the adjoint of $K_{m,n} : L^{2}(\nu_{n}) \to L^{2}(\nu_{m})$
of \eqref{def:k-mn}, about which we have the following bound.
\begin{lem}\label{diff-eq}
Under Assumption \ref{nash-assum}, if $N_k(\st) \ge \cN_{Q_k,\nu_k}(\st)$ 
is non-decreasing in $\st$, 
then for any $0 \le m  \le n$,
\begin{equation}\label{eq:1to2-bd}
\psi_{n}(n-m)^{\frac{1}{2}} \ge 
\left\Vert K_{m,n}\right\Vert _{L^{1}(\nu_{n})\to L^{2}(\nu_{m})}=\left\Vert K_{m,n}^{\star}\right\Vert _{L^{2}(\nu_{m})\to L^{\infty}(\nu_{n})} \,,
\end{equation}
where starting at any $\psi_n(0) \ge 1/\underline{\nu_n}$ 
we inductively set $\psi_n(\cdot)$ such that for 
$F(\cdot)$ of \eqref{eq:N-psi},
\begin{equation}
\frac{1}{\psi_{n}(j+1)} := F^{-1}\Big(1;\frac{1}{\psi_n(j)},N_{n-j}(\cdot)\Big), 
\qquad j=0,\ldots,n-1 \,.
\label{eq:DE}
\end{equation}
\end{lem}
\begin{proof} Fix $g : V \to \R$ 
such that $\left\Vert g\right\Vert _{L^{1}(\nu_{n})}=1$ and let
\[
J_{n}(j) :=\left\Vert K_{m,n} g \right\Vert _{L^{2}(\nu_m)}^{2} 
\qquad \mbox{with} \qquad 
j = n - m  \,.
\]
Since $K_{n,n}=I$ we have that
\begin{equation}\label{eq:jn-bd-0}
J_{n}(0)=\left\Vert g \right\Vert _{L^{2}(\nu_{n})}^{2}\le
\frac{1}{\underline{\nu_{n}}} \sum_x [\,\nu_n(x) |g(x)|\,]^2 \le
\frac{1}{\underline{\nu_{n}}} \le \psi_n(0) 
\,,  
\end{equation}
is finite (since $\nu_n \in \cM_+(V)$). In particular, 
here $L^1(\nu_n) \subseteq L^2(\nu_n) \subseteq L^\infty(\nu_n)$  
and the identity on the \abbr{rhs} of \eqref{eq:1to2-bd} 
follows by duality between $L^1(\nu_n)$ and $L^\infty(\nu_n)$. 
Turning to prove inductively for $j=1,2,\ldots$
the inequality on its \abbr{lhs}, recall that $K_{m-1,n}=K_m K_{m,n}$ 
hence considering \eqref{eq:increment} for $f_m := K_{m,n} g$ 
in $L^2(\nu_{m})$ and the definition \eqref{nash-profile} of 
$\cN_{Q_m,\nu_m} \le N_m$
we find that 
\begin{align*}
J_{n}(j)-J_{n}(j+1) & = \Vert f_m \Vert^2_{L^2(\nu_m)} - 
\Vert K_{m} f_m \Vert^2_{L^2(\nu_{m-1})} \ge \cE_{Q_m,\nu_m}(f_m,f_m)
\ge \frac{\Vert f_m \Vert^2_{L^2(\nu_m)}}{N_m(\st)} 
\end{align*}
provided $\st$ is such that 
\[
\Vert f_m \Vert^2_{L^1(\nu_m)} \le \st \Vert f_m \Vert^2_{L^2(\nu_m)} \;.
\] 
Next, iterative consideration of \eqref{contraction} down from $n$ to $m$
yields that 
$$
\Vert f_m \Vert_{L^1(\nu_m)} \le \Vert g \Vert_{L^1(\nu_n)}=1 \;,
$$
hence by the definition of $J_n(j)$ we can use $\st=1/J_n(j)$ to deduce that 
\[
J_{n}(j)-J_{n}(j+1)\ge H_{n-j}(J_n(j)) \,, \qquad j = 0,\ldots,n-1 \,,
\]
where $H_{m}(u):=u/N_m(1/u)$.
Since $\st \mapsto N_{m}(\st)$ is non-decreasing, so is the 
positive $u \mapsto H_m(u)$ and consequently the 
piece-wise linear interpolation of $J_n(\cdot)$ 
to $[0,n]$ satisfies
\[
- \frac{d}{du} \Big\{ J_n(u) \Big\} 
\ge H_{n-j}(J_n(u)) \,, \qquad \forall u \in (j,j+1) \,.
\] 
It is easy to verify that 
\[
-\frac{d}{du} \Big\{ \psi_n(u) \Big\} 
= H_{n-j}(\psi_n(u)) \,, \qquad \forall u \in (j,j+1) 
\]
for the continuous interpolation of $\psi_n(\cdot)$ of \eqref{eq:DE} according to $\frac{1}{\psi_n(u)}:=F^{-1}(u-j; \frac{1}{\psi_n(j)}, N_{n-j}(\cdot))$ 
for $u\in(j, j+1)$, which by
\eqref{eq:jn-bd-0} starts at $\psi_n(0) \ge J_n(0)$.
Thus, the continuous $\psi_n(u)-J_n(u)$ is
non-negative on $[0,n]$ and in particular
$J_{n}(j)\le\psi_{n}(j)$ for all $0 \le j \le n$. This 
holds whenever $\Vert g \Vert_{L^1(\nu_n)}=1$, so
$\psi_n(j)$ control the relevant operator norms.   
\end{proof}

\begin{remark}\label{regularity}
With $\underline{\nu_k}>0$, Proposition \ref{nash-inc-gen} is equivalent to 
the operator norm bound
\[
\left\Vert K_{m,n}^{\star}\right\Vert _{L^{1}(\nu_{m})\to L^{\infty}(\nu_{n})} 
=
\left\Vert K_{m,n}\right\Vert _{L^{1}(\nu_{n})\to L^{\infty}(\nu_{m})}
\le \CN' \psi\left(\frac{n-m}{3}\right)  
\,.
\]

Even without \eqref{eq:regular-nash}, by Lemma \ref{diff-eq} we get from 
\eqref{eq:uniform-nash} that 
\begin{equation}\label{eq:bd-psi-half}
\left\Vert K_{m,n}\right\Vert _{{L^{1}(\nu_{n})}\to L^{\infty}(\nu_{m})}\le c^{-1/2}_\star
\left\Vert K_{m,n}\right\Vert _{L^{1}(\nu_{n})\to L^{2}(\nu_{m})}\le c^{-1/2}_\star\psi^{1/2}(n-m) \,.
\end{equation}
If $\psi(\st)\le \psi(\delta \st)^2$ for some $\delta>0$, 
this already yields the bound \eqref{eq:diag-ubd-nash} of Proposition \ref{nash-inc-gen} (albeit with $\psi(\delta \cdot)$ instead of $\psi(\cdot/3)$). For 
example, replacing the (fast) growth 
assumption \eqref{eq:regular-nash} by 
\begin{equation}\label{eq:slow-nash}
\sup_{\sss \ge c_\star} \Big\{ \frac{N(\sss^2)}{N(\sss)} \Big\}  
\le \CN 
\end{equation}
yields for $\delta=1/(1+2\CN)>0$ and $u=1/\psi(\delta \st)$ that 
\[
\int_{c_\star}^{u^2} \frac{N(\sss)}{\sss}d\sss \le   
(1+ 2\CN) \int_{c_\star}^{u} \frac{N(\sss)}{\sss}d\sss \,,
\]
or equivalently, that 
$\psi(\st) \le \psi(\delta \st)^2$.
The bound \eqref{eq:slow-nash} applies for 
$N(\sss)=A (\log (1+\sss))^\beta$ with $\beta>0$, 
yielding the stretched exponential decay $\psi(\st)= A' \exp(-b\st^{1/(1+\beta)})$ 
(see \cite{BPS} and the references therein for more details about
this type of Nash profiles).
\end{remark}
Unlike the setting of Remark \ref{regularity},
for our main focus the Nash profile $N(\cdot)$ induced by the doubling function $v(r)$ has at least polynomial growth, 
namely satisfies \eqref{eq:regular-nash}, so we improve the bound 
\eqref{eq:bd-psi-half} by adapting next
the argument of \cite[proof of Theorem 2.3]{SZ2}.
\begin{lem}\label{A_n}
In the setting of Lemma \ref{diff-eq} if $A_{n} \ge 0$ 
are such that for $n \in [m,N]$ 
\begin{equation}
A_{n}^2 \le \sup_{m \le \ell \le n} 
\Big\{ \frac{A_{\ell}}{\psi_{n}(n-\ell)} \Big\} \; , 
\label{eq:A_psi}
\end{equation}
then 
\[
M_N := \sup_{n\in [m,N]}\big\{ A_{n}\left\Vert K_{m,n}\right\Vert _{L^{1}(\nu_{n})\to L^{\infty}(\nu_{m})}\big\}\le 1 \,.
\]
\end{lem}
\begin{proof} Considering \eqref{eq:1to2-bd} at some 
$\ell \in [m,n]$ we have by Lemma \ref{diff-eq} that 
\begin{align*}
\left\Vert K_{m,n}^{\star}\right\Vert _{L^{1}(\nu_{m})\to L^{\infty}(\nu_{n})}
& \le\left\Vert K_{m,\ell}^{\star}\right\Vert _{L^{1}(\nu_{m})\to L^{2}(\nu_{\ell})}\left\Vert K_{\ell,n}^{\star}\right\Vert _{L^{2}(\nu_{\ell})\to L^{\infty}(\nu_{n})} \\
& \le 
\left\Vert K_{m,\ell}^{\star}\right\Vert _{L^{1}(\nu_{m})\to L^{2}(\nu_{\ell})}
\psi_{n}(n-\ell)^{\frac{1}{2}}.
\end{align*}
Recall that a Markov kernel $K_{m,\ell}$ is a contraction 
from $L^{\infty}(\nu_{\ell})$ to $L^\infty(\nu_m)$ 
(as $\nu_{\ell}$ is strictly positive on $V$). By duality, 
the adjoint $K_{m,\ell}^{\star}$ is thus 
a contraction from $L^{1}(\nu_{m})$ to $L^{1}(\nu_{\ell})$.
Moreover, for any $\ell$ and $f_\ell$
\[
\left\Vert f_\ell \right\Vert _{L^{2}(\nu_{\ell})}^{2}
\le\left\Vert f_\ell \right\Vert _{L^{\infty}(\nu_{\ell})}\left\Vert 
f_\ell \right\Vert _{L^{1}(\nu_{\ell})} 
\]
and taking $f_\ell = K^{\star}_{m,\ell} \, g \,$ for arbitrary 
$g \in L^1(\nu_m)$ we get that
\begin{align*}
\left\Vert K_{m,\ell}^{\star}\right\Vert^2_{L^{1}(\nu_{m})\to L^{2}(\nu_{\ell})} & \le\left\Vert K_{m,\ell}^{\star}\right\Vert _{L^{1}(\nu_{m})\to L^{\infty}(\nu_{\ell})}\left\Vert K_{m,\ell}^{\star}\right\Vert _{L^{1}(\nu_{m})\to L^{1}(\nu_{\ell})}
\le\left\Vert K_{m,\ell}^{\star}\right\Vert _{L^{1}(\nu_{m})\to L^{\infty}(\nu_{\ell})} \,.
\end{align*}
By the definition of $M_N$, we thus deduce upon choosing the optimal $\ell$ from \eqref{eq:A_psi}, that for any $n \in [m,N]$, 
\begin{align*}
A_n^2 \left\Vert K_{m,n}^{\star}\right\Vert^2_{L^{1}(\nu_{m})\to L^{\infty}(\nu_{n})} & \le A_n^2 \left\Vert K_{m,\ell}^{\star}\right\Vert _{L^{1}(\nu_{m})
\to L^{\infty}(\nu_{\ell})} \psi_{n}(n-\ell) \\
 & \le A_n^2 M_N A_{\ell}^{-1}\psi_n (n-\ell) \le M_N  \,.
\end{align*}
Finally, taking the supremum over $n$ in the \abbr{LHS}, we find that
$M_N \le 1$, as claimed.
\end{proof}

Starting with an upper bound $N_k(\sss)$ on the Nash profile functions, 
one merely applies Lemmas \ref{diff-eq} and \ref{A_n}. That is,
first solving the map \eqref{eq:DE} corresponding
to $N_{k}(\sss)$ in order to get an upper bound on $\psi_{n}(\cdot)$, 
from which by \eqref{eq:A_psi} one deduces the diagonal upper bound $1/A_{n}$. 
This is much simplified in the presence of
uniform bounds, as in Proposition \ref{nash-inc-gen}. 

\begin{proof}[Proof of Proposition \ref{nash-inc-gen}]
Considering Lemma \ref{diff-eq} for $N_k(\cdot)=N(\cdot)$ and 
$\psi_n(0)=1/c_\star$, we have the bound \eqref{eq:1to2-bd} for 
the non-increasing $\psi_n(t)=\psi(t)=1/F^{-1}(t;c_\star,N(\cdot))$.
We claim that  
\begin{equation}\label{eq:psi-ratio-bd}
\CN' := \sup_{n \ge 0} 
\Big\{ \frac{\psi(n/4)}{\psi(n/3)} \Big\}^2 < \infty 
\end{equation}
in which case $A_n=1/(\CN' \psi(n/3))$ and $\ell=[3n/4]$ satisfy
\eqref{eq:A_psi}, with Lemma \ref{A_n} completing the proof 
of the proposition. Next, setting $\CN'=C^2$ and $u=1/\psi(n/4)$, 
we get \eqref{eq:psi-ratio-bd} upon showing that 
$1/\psi(n/3) \le C u$, or alternatively that for any $u \ge c_\star$  
\[
h(u) := 3 \int_{u}^{C u} \frac{N(\sss)}{\sss}d\sss -  
\int_{c_\star}^{u} \frac{N(\sss)}{\sss}d\sss \ge 0 \,.
\]
To this end, recall \eqref{eq:regular-nash}
that for $C \ge \CN > 1$ and 
$u \ge \sss_0$ we assumed that 
$N(C u) \ge 2 N(u)$, hence
\[
h'(u) = \frac{3N(Cu)}{u} - \frac{4N(u)}{u} \ge \frac{2N(u)}{u} \ge 0 \,.
\]
Finally, with $u \mapsto N(u)$ non-decreasing, 
$h(u) \ge N(u) ( 3 \log C - \log (u/c_\star) )$ is non-negative
for any $C \ge (\sss_0/c_\star)^{1/3}$ and all $u \in [c_\star,s_0]$.
\end{proof}

\noindent
Coupling the dynamics of \eqref{def:k-mn} and \eqref{def:k-st} 
we deduce part (b) out of part (a) of Theorem \ref{nash-inc}.
\begin{proof}[Proof of Theorem \ref{nash-inc}(b)]
From the arrival times 
$\{T_n'\}$ of an auxiliary Poisson process of rate $2$, we
construct $\{Y_t\}$ obeying \eqref{def:k-st}
by independently 
censoring the jump at each time $T'_n$ with probability $1/2$
and proceeding at the non-censored $\{T_n\} \subseteq \{T'_n\}$
as in \eqref{def:csrw}. Fixing a realization $\omega = \{T'_n\}$, 
at most $N_t-N_s$ jump attempts are made in $[s,t]$ by the 
corresponding dynamics $\{X_n\}$ of \eqref{def:k-mn} having
the $1/2$-uniformly lazy transitions 
$K_n^{(\omega)} := \frac{1}{2} K_{T'_n} + \frac{1}{2} I$.
Recall that the Dirichlet form 
of any $\pi$-reversible $K$ has the symmetric form 
\begin{align}
\cE_{K,\pi}(f,g)=\frac{1}{2}\sum_{x,y\in V}(f(x)-f(y))(g(x)-g(y))\pi(x) K(x,y)
\label{dirichlet}
\end{align}
and under \eqref{dfn:unif-lazy} we have that 
$Q(x,y) \ge \alpha_l K(x,y)$ for $Q=K^2$ and all $x,y \in V$. Thus, 
in the reversible case $\cE_{K^2,\pi}(f,f) \ge \alpha_l \, \cE_{K,\pi}(f,f)$
for all $f$ and consequently, 
\begin{align}\label{eq:nash-profile-vd2}
\cN_{K,\pi}(\sss) \ge 
\alpha_l \, \cN_{K^2,\pi}(\sss)  
\,, \qquad \forall \sss\,.
\end{align}
This applies for $K=K_n^{(\omega)}$, $\alpha_l=1/2$ and 
the non-decreasing $n \mapsto \pi_{T'_n}(x)$, hence by part (a) of 
Theorem \ref{nash-inc}, the bound \eqref{eq:diag-ubd-nash} holds for
$\{X_n\}$ with $\psi(\cdot)$ as stated in part (b).
That is, 
\begin{align*}
\frac{K^{(\omega)}_{N_s,N_t}(x,y)}{\pi_t(y)} \le 
\frac{K^{(\omega)}_{N_s,N_t}(x,y)}{\pi_{T'_{N_t}}(y)}
\le \CN' \psi\Big(\frac{N_t-N_s}{3}\Big) \,.
\end{align*}
To complete the proof, note that 
$K_{s,t}(x,y)$ is the expected value of $K^{(\omega)}_{N_s,N_t}(x,y)$ 
over $\omega$, whereas $N_t-N_s \sim$ Poisson$(2(t-s))$.
\end{proof}

For the remainder of this section, we detail three additional
situations in which 
good upper bounds on the Nash profiles are available. First,
recalling that the Poincar\'e inequality 
together with uniform volume growth for $v(r)$ doubling, 
provide an upper bound on the Nash profile, in the following lemma we deduce 
from Theorem \ref{nash-inc} that 
in the context of Theorem \ref{thm-main} 
the on-diagonal 
upper bound \eqref{eq:diag-ubd-nash} holds 
for $\psi(\st)=C'/v(\sqrt{\st})$.
\begin{lem}\label{P+VG-to-Nash} 
Suppose $\bG$ and non-decreasing $t \mapsto \pi_t(x)$
satisfy uniform volume growth $v(r)$ as in \eqref{d-set}, for 
$v(r)$ doubling. Further, suppose
the $\pi_t$-reversible $K_t$ satisfy the 
uniform \abbr{PI}
\eqref{poincare}, and in case of \eqref{def:k-mn} such 
$\{K_t\}$ are also $\alpha_l$-uniformly lazy.
Then, for
 $C'(\CP,\CV, \alpha_l)$ finite, 
\begin{equation}\label{eq:diag-ubd-nash-rev}
\left\Vert K_{s,t}\right\Vert _{L^{1}(\pi_t)\to L^{\infty}(\pi_s)}
= \sup_{x,y\in V}\Big\{ 
\frac{K_{s,t}(x,y)}{\pi_t(y)}\Big\} \le \frac{C'}{v(\sqrt{t-s})} 
\quad \qquad \forall s \in [0,t] \,.
\end{equation}
\end{lem}
\begin{proof} In case of the dynamic \eqref{def:k-mn}, with 
\eqref{eq:nash-profile-vd2} applicable for $K_t$, by 
Theorem \ref{nash-inc}(a) it suffices to show that if 
$\pi$-reversible $K$ satisfies the Poincar\'e inequality 
\eqref{poincare} and the uniform growth 
assumption \eqref{d-set} for volume doubling $v(r)$, then for 
$\beta=2$,
$v^{-1}(\sss) := \inf \{r \ge 1 : v(r) \ge \sss \}$ and 
some $C(\CP,\CV)$ finite, 
\begin{equation}\label{eq:nash-profile-vd} 
\alpha_l \, N(\sss) := C \big( v^{-1} (C \sss) \big)^{\beta} \ge 
\cN_{K,\pi}(\sss) \,, \qquad \forall \sss \,.
\end{equation} 
Indeed, $v^{-1} (\sss/\CV^{k}) \le 2^{-k} v^{-1} (\sss)$ so 
$N(\CV \sss) \ge 4 N(\sss)$, namely
\eqref{eq:regular-nash} holds with $\CN=\CV$. Moreover,
splitting the integral \eqref{eq:N-psi} 
for such non-decreasing $\sss \mapsto N(\sss)$, to 
intervals $[u/\CV^{k+1}, u/\CV^{k}]$, $k \ge 0$ 
shows that the \abbr{rhs} of \eqref{eq:N-psi} is 
dominated by the largest intervals and hence 
for some $C'=C'(C/\alpha_l,\CV)$ finite
and any $\st \ge 1/3$,
\begin{equation}\label{eq:psi-value}
\psi(\st)=\frac{1}{F^{-1}(\st;\underline{\pi_0},N(\cdot))}
\le \frac{C'}{v((3\st)^{1/\beta})} \,.
\end{equation}
Thus, multiplying $C'$ by $\CN'$ we get
\eqref{eq:diag-ubd-nash-rev} as
a consequence of \eqref{eq:diag-ubd-nash}. 
Further, recall Theorem \ref{nash-inc}(b) that for 
the dynamic \eqref{def:k-st} the preceding bound on 
$\psi(\st)$ always hold (with $\alpha_l=1/2$),
and in this context we arrive 
for $N_\lambda \sim$Poisson($\lambda$) and the non-decreasing $v(r) \ge 1$, 
at 
\begin{equation}\label{eq:diag-ubd-nash-dyn}
\left\Vert K_{s,t}\right\Vert _{L^{1}(\pi_t)\to L^{\infty}(\pi_s)}
\le \E \Big[ \frac{C'}{v((N_{2(t-s)})^{1/\beta})} \Big] 
\le \frac{C'}{v((t-s)^{1/\beta})} + C' e^{-(t-s)/\kappa} 
\end{equation}
(as $\P(N_{2\lambda} \le \lambda) \le e^{-\lambda/\kappa}$
for some $\kappa$ finite and all $\lambda \ge 0$). With 
$v(r) e^{-r^2/\kappa}$ uniformly bounded, upon increasing $C'$ 
we thus 
get \eqref{eq:diag-ubd-nash-rev} out of \eqref{eq:diag-ubd-nash-dyn}.

Turning to establish
\eqref{eq:nash-profile-vd}, 
recall that 
from \eqref{dirichlet} 
and the covering argument in the proof of 
\cite[Lemma 2.4]{SC}, it follows  
that for some $c=c(\CP,\CV)$ finite, any $f \in L^2(\pi)$ and $r > 0$ 
\[
\|f-f_r\|^2_{L^2(\pi)} \le c \, r^{2} \cE_{K,\pi}(f,f) \,,
\]
where
\[
f_{r}(x):=\frac{1}{\pi(\mathbb{B}(x,r))}\sum_{y\in\mathbb{B}(x,r)} f(y) \pi(y) \,.
\]
Further, by the uniform volume growth assumption of \eqref{d-set},
\begin{align*}
\left\Vert f_{r}\right\Vert _{L^{\infty}(\pi)} & \le\frac{1}{\inf_{x}\pi(\mathbb{B}(x,r))}\left\Vert f\right\Vert _{L^{1}(\pi)}\le\frac{\CV}{v(r)}\left\Vert f\right\Vert _{L^{1}(\pi)},\\
\left\Vert f_{r}\right\Vert _{L^{1}(\pi)} & \le \CV^{2}\left\Vert f\right\Vert _{L^{1}(\pi)} \,.
\end{align*}
Consequently  
\begin{equation}\label{tianyi-1}
\|f\|_{L^{2}(\pi)}^{2}\le 
2 \|f-f_r\|_{L^2(\pi)}^2 + 2 \|f_r\|_{L^2(\pi)}^2 
\le 2 c \, r^{2}\mathcal{E}_{K,\pi}(f,f) + \frac{2 \CV^{3}}{v(r)}
\| f\|_{L^{1}(\pi)}^{2} \,.
\end{equation}
Setting $C:=2 c + 2 \CV^3$ and 
$r = v^{-1}(C \sss)$, it follows from \eqref{tianyi-1} that 
if $\| f \|^2_{L^{1}(\pi)} \le \sss$ and $\|f\|_{L^2(\pi)}^2 = 1$, then 
\begin{equation*}
\frac{1}{\cE_{K,\pi}(f,f
)} \le C r^2 = \alpha_l \, N(\sss) \,.
\end{equation*}
To complete the proof, recall that $\cN_{K,\pi}(\sss)$ is 
the maximum of $1/\cE_{K,\pi}(f,f)$ over such $f$. 
\end{proof}

One can also obtain upper bounds on the Nash profiles from lower bounds on isoperimetric profiles of the Markov operators. Indeed, there is a tight 
connection between the Nash and isoperimetric
profiles of a Markov chain $Q$ having invariant measure $\pi$. Specifically, recall the $L^{2}$-isoperimetric (or spectral) profile 
of such a chain $(Q,\pi)$ on infinite state space $V$, defined as 
the non-increasing 
\begin{equation}\label{dfn:L2-isop}
\Lambda_{Q,\pi}(u)=\inf\{\lambda_{Q}(\Omega):\Omega\subseteq V,\;\pi(\Omega)\le u\}
\end{equation}
where 
\begin{equation}
\lambda_{Q}(\Omega)=\inf\big\{\mathcal{E}_{Q,\pi}(f,f):\ \mbox{support}(f)
\subseteq \Omega,\; \|f\|_{L^{2}(\pi)}=1\big\}\,.
\label{def-eig}
\end{equation}
In words, $\lambda_{Q}(\Omega)$ is the smallest eigenvalue of the operator
$I-Q$ with Dirichlet boundary condition in $\Omega$. Also recall the 
$L^{1}$-isoperimetric (or conductance) profile 
\[
\Phi_{Q,\pi}(u)=\inf\Big\{ 
\frac{1}{\pi(\Omega)} \sum_{x\in\Omega} \pi(x)Q(x,\Omega^c)\,:
\Omega \subseteq V, \; \pi(\Omega)\le u\Big\}\,.
\]
The $L^2$ and $L^1$ profiles are related via Cheeger's inequality (see \cite{LS}), 
\begin{equation}
\frac{1}{2}\Phi_{Q,\pi}^{2}(u)\le\Lambda_{Q,\pi}(u)\le\Phi_{Q,\pi}(u)\,.
\label{Cheeger}
\end{equation}
As shown next, the Nash profile $\mathcal{N}_{Q,\pi}(\cdot)$ contains 
the same information as 
the $L^{2}$-isoperimetric profile (see \cite[Lemma 2.1]{GMT}
for such a result in case of finite Markov chains; 
the proof for $V$ infinite is provided here for the reader's convenience).
\begin{lem}\label{N-L} For Markov operator $Q$, its $\sigma$-finite
invariant measure $\pi$, and any $u>0$,
\[
\frac{1}{\Lambda_{Q,\pi}(u)}\le\mathcal{N}_{Q,\pi}(u)\le\frac{2}{\Lambda_{Q,\pi}(4u)}\,.
\]
\end{lem}
\begin{proof} By the Cauchy-Schwarz inequality 
$\|f\|_{L^{1}(\pi)}^{2}\le\pi(\Omega)\|f\|_{L^{2}(\pi)}^{2}$
for any $f$ supported within $\Omega$, yielding that 
$\Lambda_{Q,\pi}(u) \ge 1/\cN_{Q,\pi}(u)$ via  
the definitions \eqref{nash-profile} and \eqref{dfn:L2-isop}.
We proceed to show that 
$\Lambda_{Q,\pi}(4u) \le 2/\cN_{Q,\pi}(u)$. Namely, 
that for any $f \in L^2(\pi)$  
\begin{equation}\label{eq:nash-isop}
\|f\|_{L^1(\pi)}^2 = u \|f\|_{L^2(\pi)}^2 
\quad \Longrightarrow 
\quad 
\frac{1}{2} \|f\|^2_{L^2(\pi)} \, \Lambda_{Q,\pi}(4u)  
\le \cE_{Q,\pi}(f,f) \,.
\end{equation}
Indeed, recall from \eqref{dfn:L2-isop} that for $f \in L^2(\pi)$,
\[ 
\|f\|_{L^{2}(\pi)}^{2} \,
\Lambda_{Q,\pi}\big( \pi(\mbox{support}(f))\big) 
\le \cE_{Q,\pi}(f,f) \,.
\]
Next, for $\st \ge 0$ set $f_{\st}:=(f-\st)_+$ 
supported on $\Omega_\st := \{ f > \st \}$. Obviously 
$f_\st \in L^2(\pi)$, with 
$f^{2} - 2 \st |f| \le f_{\st}^{2}$ 
and
$\cE_{Q,\pi}(f_{\st},f_{\st})\le
\cE_{Q,\pi}(f,f)$. Hence, 
\[
\Big[
\|f\|_{L^{2}(\pi)}^{2} - 2\st\|f\|_{L^1(\pi)} \Big] 
\Lambda_{Q,\pi}(\pi(\Omega_\st)) 
\le  
\|f_\st\|^2_{L^2(\pi)} 
\Lambda_{Q,\pi}(\pi(\Omega_\st)) 
\le \cE_{Q,\pi}(f_\st,f_\st) \le \cE_{Q,\pi}(f,f) \;. 
\]
Since $\pi(\Omega_\st)\le \st^{-1}\|f\|_{L^{1}(\pi)}$, 
if $4\st \|f\|_{L^1(\pi)} =  \|f\|^2_{L^2(\pi)}$ 
for $\st=\|f\|_{L^1(\pi)}/(4 u)$ finite, then 
\[
\frac{1}{2} 
\|f\|_{L^{2}(\pi)}^{2} \,
\Lambda_{Q,\pi}(4u) 
\le 
\frac{1}{2} 
\|f\|_{L^{2}(\pi)}^{2} \,
\Lambda_{Q,\pi}(\pi(\Omega_\st)) 
\le \cE_{Q,\pi}(f,f)\,,
\]
as claimed in \eqref{eq:nash-isop}.
\end{proof}

By Lemma \ref{N-L}, any lower bound on the $L^{2}$ or $L^{1}$-isoperimetric
profile can be turned into an upper bound on the Nash profile. 
\begin{ex}\label{inc-diag-ubd}
Consider Example \ref{inc-ex} with $\nu_n=\pi_n$ invariant 
for $K_n$ such that $n\mapsto\pi_{n}(x)$ are non-decreasing 
(and $Q_{n}=K_{n}^{\star} K_{n}$). Suppose also  
that for some $d>0$ and positive $\kappa_n$ 
\begin{equation}
\mathcal{N}_{Q_{n},\pi_{n}}(\sss)\le 4 \alpha^{-1} \kappa_{n}^{-2} (4\sss)^{2/d}\,, 
\qquad \forall \sss,n,
\label{eq:poly}
\end{equation}
where $\gamma_n := \sum_{m=1}^n \kappa_m^2$ are such that for some $c_0 \ge 2$, 
\begin{equation}
\gamma_n \ge c_0 \qquad \Longrightarrow \qquad \exists \ell(n) \qquad 
\frac{1}{3} \le \frac{1+\gamma_{\ell(n)}}{1+\gamma_n} \le \frac{2}{3}  \,.
\label{eq:ga-c}
\end{equation}
It is easy to check that for some $c=c(\alpha, d, \underline{\pi_0})$ finite,
\[
\psi_{n}(j)=c \Big(1 + \gamma_n - \gamma_{n-j} \Big)^{-d/2}\,,
\]
satisfies \eqref{eq:DE} and consequently the bound \eqref{eq:1to2-bd}.
The condition \eqref{eq:ga-c} allows for taking
$A_{n}= c_1^{-1} (1+\gamma_n)^{d/2}$ with $c_1= c (1+c_0)^{d}$
in Lemma \ref{A_n}, thereby concluding that 
\begin{equation}\label{eq:dhmp}
\sup_{x,y \in V} \, \Big\{ \frac{K_{0,n}(x,y)}{\pi_{n}(y)} \Big\} 
\le c_1 \Big(1+ \sum_{m=1}^{n} \kappa_{m}^2 \Big)^{-d/2} \,.
\end{equation}
In particular, the bound \eqref{eq:dhmp} recovers \cite[Theorem 1.2]{DHMP},
proved before for the dynamic \eqref{def:k-mn} with
$(K_t,\pi_t)$ of \eqref{n-kernel}-\eqref{dfn:v-cond}
via evolving sets techniques. More precisely,
\cite{DHMP} assume that the uniform lazy property
\eqref{dfn:unif-lazy} holds and 
$L^1$-isoperimetric profiles such that for some $d>1$ and
positive $\kappa_n$
\begin{equation}\label{eq:isop}
\Phi_{K_n,\pi_n}(\sss) \ge \kappa_n \sss^{-1/d} \,.
\end{equation}
As noted by \cite{DHMP},
if $\sup_{x,n} \{\pi_n(x)\} \le C^d$ then
$\kappa_n \le (\underline{\pi_n})^{1/d} \le C$, so \eqref{eq:ga-c} holds 
(for $c_0 = 2+ 3 C^2$). Recall the left inequality of 
\eqref{eq:nash-profile-vd} that for $Q_n=K_n^2$ in the reversible case
\begin{align*}
\cN_{Q_{n},\pi_{n}}(\sss)\le\alpha_l^{-1} \, \cN_{K_{n},\pi_{n}}(\sss)
\,,\qquad \forall \sss, n \,.
\end{align*}
Finally, by Lemma \ref{N-L} and
Cheeger's inequality \eqref{Cheeger}, 
the isoperimetric bound \eqref{eq:isop} implies 
\[
\cN_{K_{n},\pi_{n}}(\sss) \le 2 \Lambda_{K_n,\pi_n}(4\sss)^{-1} 
\le 4 \Phi_{K_n,\pi_n}(4\sss)^{-2} \le 
4 \kappa_{n}^{-2} (4\sss)^{2/d} \,, \qquad \forall \sss,n \,.
\] 
So, the assumptions of \cite{DHMP} imply both
\eqref{eq:poly} and \eqref{eq:ga-c}, thereby yielding
\eqref{eq:dhmp}.
\end{ex}

Theorem \ref{nash-inc} and Proposition \ref{nash-inc-gen} can be applied to general Markov operators $K_t$, in particular, having non-local jumping kernels. The following example illustrates this point.
\begin{ex}\label{nonlocal}
Let $d(x,y)$ denote the graph distance on a locally finite, connected, infinite, 
non-oriented graph $\mathbb{G}=(V,E)$. Suppose that for each $n\in \N$ the 
non-local Markov kernel $K_n$ is reversible with respect to the 
measure $\pi_n$, such that:
\newline
$\bullet$ For every $x\in V$ the sequence $n\mapsto \pi_n(x)$ is non-decreasing,
and $\pi_n$ satisfies uniform volume growth with $v(r)$ doubling as in \eqref{d-set}.
\newline 
$\bullet$ There exist $A<\infty$ and $\beta \in (0,2)$ such that for any 
$x \ne y \in V$,
\begin{equation}\label{eq:Chen-Kumagai}
\frac{K_n(x,y)}{\pi_n(y)} \ge \frac{d(x,y)^{-\beta}}{A\, v(d(x,y))} \,.
\end{equation}
It is easy to see that upon changing $A$ to $A^2 \, v(2) 2^{\beta} \CV$,
the bound \eqref{eq:Chen-Kumagai} applies also for $K_n^2(x,y)$. Hence,
by \cite[Theorem 3.1]{CK} there exist finite 
$c_i=c_i(\CV,A,\beta)$, $i=1,2$, such that for 
$\theta(r):= r \left( v^{-1}(c_1/r)\right)^{-\beta}$
and the inverse $v^{-1}$ of the function $r\mapsto v(r)$,
the following Nash inequality holds
for every $f\in \mbox{Dom} (\cE_{K^2_n,\pi_n})$ with $\|f\|_{L^{1}(\pi_n)}=1$:
$$
\theta\left(\|f\|^{2}_{L^2(\pi_n)}\right)\le c_2\mathcal{E}_{K^2_n,\pi_n}(f,f)\,.
$$ 
Equivalently, the Nash profile of $(K_n^2,\pi_n)$ then satisfies 
$$
N(\st) :=
c_2\left( v^{-1}(c_1\st)\right)^{\beta} \ge 
\mathcal{N}_{K^2_n,\pi_n}(\st) \,.
$$
As in the derivation of \eqref{eq:psi-value}, 
it then follows from Theorem \ref{nash-inc} 
(now with $\CN = \CV^{\lceil 1/\beta \rceil}$),
that for 
the dynamics \eqref{def:k-mn},
some finite $c_3=c_3(\CV,A,\beta)$ and all integers $s<t$,
$$\sup_{x,y\in V}\Big\{ \frac{K_{s,t}(x,y)}{\pi_t(y)}\Big\} \le \frac{c_3}{v((t-s)^{1/{\beta}})}
\,.
$$
\end{ex}
\end{section}

\begin{section}{Gaussian upper bounds}\label{sect:gauss} 

We adapt the technique of \cite[Section 2]{HS} for deriving
off-diagonal Gaussian upper bounds via complex interpolation 
techniques. Specifically, in this section we work with 
$L^{p}$ spaces of $\bC$-valued functions, with $\cC_{0}(V)$ 
denoting the dense linear subspace of 
finitely supported $\bC$-valued functions. Considering  
$\rho: V \to \R$ such that the non-negative linear operators 
on $\cC_0(V)$ 
\begin{equation}
K_{s,t}^{\theta}f(x):=w_{-\theta}K_{s,t}(w_{\theta}f)(x) \,,\quad
w_{\theta}(x):=e^{\theta \rho(x)}\,,\quad 0 \le s \le t \le T\,, 
\quad \theta \in \R\,,
 \label{eq:Ks}
\end{equation}
have bounded $L^2(\nu_t) \to L^2(\nu_s)$ norms, we study the 
unique continuous extension of 
$K_{s,t}^\theta$
for both continuous and discrete time (where $s,t \in \Z_+$).
Our main example is $\rho(x)=d(x,x_0)$ for the
graph distance $d(x,y)=d_{\bG}(x,y)$ in a locally finite, 
connected graph $\bG$ and a fixed vertex $x_0 \in V$.

For completeness we first prove the following proposition, which
summarizes the interpolation method of \cite{HS}.
\begin{ppn}\label{prop:hs} \emph{(see \cite[Lemma 2.2]{HS})} 
Suppose $\{K^\theta_{s,t}\}$ are as in \eqref{eq:Ks}, with 
non-negative linear operators $K_{s,t} = K_{s,\xi} K_{\xi,t}$,
$s \le \xi \le t$ (with $K_{\xi,\xi}=I$), such that 
\begin{equation}
\left\Vert K_{s,t}\right\Vert _{L^{\infty}(\nu_{t})\to L^{\infty}(\nu_{s})}\le 
1\,,\qquad\forall 0 \le s \le t \le T\,,\label{eq:infty}
\end{equation}
for strictly positive 
$\sigma$-finite measures $\{\nu_s\}$, with $\nu_0 \in \cM_+(V)$. 
Suppose that for $t \ge s \in [0,T]$:\\
(a) For a non-decreasing $t \mapsto a_t$ satisfying \eqref{cond:pert},
the \emph{Gaffney bound}
\begin{align}
\left\Vert K_{s,t}^{\theta}\right\Vert _{L^{2}\left(\nu_{t}\right)\to L^{2}(\nu_{s})} &\le \exp \big( a_t - a_s + \chi(\theta) (t-s) \big) \,,
\quad\forall \, \theta \in \R \,, 
\label{dfn:gaffney-bd} \\
\hbox{with} 
\quad \qquad c_1^{-1} \theta^2 &\le \chi(\theta) \le c_1 \theta^2 \,,  
\quad \qquad \qquad \qquad \forall |\theta| \le \delta_\star \,,
\label{eq:chi-bd}
\end{align}
holds on $\cC_0(V)$, for $|\theta| \mapsto \chi(\theta)$ non-decreasing, 
some $\delta_\star > 0$ and $c_{1}<\infty$. \\
(b) \emph{The $2 \to\infty$ bound}
\begin{equation}
\left\Vert K_{s,t}\right\Vert_{L^{2}(\nu_{t})\to L^{\infty}(\nu_{s})}\le\varphi(t-s) \,,\label{eq:on-diagonal}
\end{equation}
holds with $\tau \mapsto \varphi(\tau) \tau^{\bbeta}$ non-decreasing on $[0,T]$ for some $\bbeta>0$.\\
Then, for some finite $C_{1}(c_{1},\bbeta,\delta_\star)$ and  
$C_{2}=C_{2}(A,\bbeta,\underline{\nu_0},\varphi(1))$,
\begin{equation}\label{eq:bd-hs}
\left\Vert K_{0,T}^{\theta}\right
\Vert _{L^{2}(\nu_{T})\to L^{\infty}(\nu_{0})}\le
C_{2} \, \varphi(T) \exp(C_{1} \chi(\theta) T) \,, \qquad 
 \forall \theta \in \R\,, \quad  T \ge 1 \,.
\end{equation}
\end{ppn}
\begin{remark}\label{rem:contract} 
Clearly, \eqref{eq:infty} holds for Markov transition probabilities
$\{K_{s,t}\}$ and strictly positive $\{\nu_s\}$. 
\end{remark}
\begin{proof} Since $(\underline{\nu})^{1/q}
\left\Vert f\right\Vert _{L^{\infty}(\nu)}\le 
\left\Vert f\right\Vert _{L^{q}(\nu)}$ for any $f$ and $\nu \in \cM_+(V)$, 
necessarily  
\begin{align}\label{eq:pq-to-p-inf}
\left\Vert K_{0,t}^{\theta}\right\Vert_{L^{2p}(\nu_{t})\to L^{\infty}(\nu_0)} 
\le (\underline{\nu_0})^{-1/(2q)}
\left\Vert K_{0,t}^{\theta}\right\Vert_{L^{2p}(\nu_{t})\to 
L^{2q}(\nu_{0})}\,, \qquad \forall p, q, t > 0 \,.
\end{align}
Considering \eqref{dfn:gaffney-bd} and
\eqref{eq:pq-to-p-inf} at $s=0$, $t=T$, $p=q=1$, 
our assumption \eqref{eq:chi-bd} and having
$\varphi(T) \ge \varphi(1) T^{-\bbeta}$ for $T \ge 1$,
yield \eqref{eq:bd-hs} for 
$C_1=1 + c_1 \kappa \bbeta/\delta_\star^2$, 
$C_2=(\underline{\nu_0})^{-1/2} \varphi(1)^{-1}$ and
$|\theta| \ge \delta_\star/(1+\log T)$, where
\[
\kappa := \sup_{t \ge 1} \Big\{ \frac{\log t}{t} (1 + \log t)^2 \Big\} 
< \infty \,.
\]
We proceed to derive \eqref{eq:bd-hs} when 
$|\theta| < \delta_\star/(1+\log T)$ by closely 
following \cite[Subsection 4.2.2]{SC2}. To this end, 
\eqref{eq:infty} and \eqref{eq:on-diagonal} are invariant 
under the re-scaling $\nu_{s;t} := e^{-2(a_t-a_s)} \nu_s$ 
and yield by Riesz-Thorin interpolation theorem that 
\begin{equation}
\left\Vert K_{s,t}\right\Vert _{L^{2p}(\nu_{t})\to L^{\infty}(\nu_{s;t})}\le\varphi(t-s)^{1/p}\,, \qquad \forall p \ge 1 \,.
\label{eq:p-infty}
\end{equation}
To apply Stein's interpolation theorem, consider the 
$\bC$-valued weights $w_{\theta z}(x)=w_\theta (x)^z$ indexed on
the strip $\bS := \{ z = u+ib : u = \Re(z) \in [0,1]\}$.
For fixed $\theta \in \R$ and $m \le \ell$ 
the associated map $z \mapsto K_{m,\ell}^{\theta z}=w_{-\theta z}K_{m,\ell}\, w_{\theta z}$ 
forms an $\bS$-analytic collection  of linear operators on
$\cC_{0}(V)$ such that 
$K_{m,\ell}^{\theta (u+ib)}= w_{-i\theta b} K_{m,\ell}^{\theta u} w_{i\theta b} 
$. With $|w_{\pm i \theta b}(x)|=1$ 
we thus have for $u=0$ the $L^\infty$-contraction 
\begin{equation}
\left\Vert K_{s,t}^{\theta ib}\right\Vert _{L^{\infty}\left(\nu_{t}\right)\to L^{\infty}(\nu_{s;t})}\le\left\Vert K_{s,t}\right\Vert _{L^{\infty}\left(\nu_{t}\right)\to L^{\infty}(\nu_{s;t})}\le 1.\label{eq:infty-b}
\end{equation}
Moving to $\nu_{s;t}$ eliminates the term $a_t-a_s$ 
in \eqref{dfn:gaffney-bd}, so we get 
for $u=1$ the $L^2$-norm bound 
\begin{equation}
\left\Vert K_{s,t}^{\theta (1+ib)}\right\Vert _{L^{2}\left(\nu_{t}\right)\to L^{2}(\nu_{s;t})}\le\left\Vert K_{s,t}^{\theta}\right\Vert _{L^{2}\left(\nu_{t}\right)\to L^{2}(\nu_{s;t})}\le\exp\left(\chi(\theta) (t-s) \right).\label{eq:2-b}
\end{equation}
By Stein's interpolation 
(see \cite{SW}), from 
\eqref{eq:infty-b} and \eqref{eq:2-b} we have for $\lambda = 1/p \in [0,1]$,
\begin{equation*}
\left\Vert K_{s,t}^{\theta \lambda}\right\Vert _{L^{2p}\left(\nu_{t}\right)\to L^{2p}(\nu_{s;t})}\le\exp\left(\lambda \chi(\theta) (t-s) \right)\, 
\end{equation*}
and upon replacing $\theta$ by $\theta/\lambda$, deduce that for all $\theta \in \R$,
\begin{equation}\label{eq:stein1}
\left\Vert K_{s,t}^{\theta(1+ib)}\right\Vert_{L^{2p}\left(\nu_{t}\right)\to L^{2p}(\nu_{s;t})}\le\left\Vert K_{s,t}^{\theta}\right\Vert _{L^{2p}\left(\nu_{t}\right)\to L^{2p}(\nu_{s;t})}\le\exp\left(p^{-1} \chi(\theta p)
(t-s) \right) \,.
\end{equation}
Next, employing \eqref{eq:p-infty} gives that  
\begin{equation}\label{eq:stein0}
\left\Vert K_{s,t}^{\theta ib}\right\Vert _{L^{2p}(\nu_{t})\to L^{\infty}(\nu_{s;t})}\le\left\Vert K_{s,t}\right\Vert _{L^{2p}(\nu_{t})\to L^{\infty}(\nu_{s;t})}\le\varphi(t-s)^{1/p}\,.
\end{equation}
With $\nu_t=\nu_{t;t}$, from 
\eqref{eq:stein1} and \eqref{eq:stein0} we conclude 
by yet another application of Stein's interpolation theorem,
that for $\lambda=p/q \in [0,1]$ and any $q \ge p \ge 1$,
\[
\left\Vert K_{s,t}^{\theta \lambda}\right\Vert _{L^{2p}\left(\nu_{t;t}\right)\to L^{2q}(\nu_{s;t})}\le\varphi(t-s)^{\frac{1-\lambda}{p}}\exp\Big(
\lambda p^{-1} \chi(\theta p) (t-s) \Big)\,.
\]
Considering $\wh{\nu}_t = \nu_{t;T}$ and replacing once 
more $\theta$ by $\theta/\lambda$, we get that  
\begin{align}
\left\Vert K_{s,t}^{\theta}\right\Vert _{L^{2p}\left(\wh{\nu}_{t}\right)\to L^{2q}(\wh{\nu}_{s})} &= 
e^{(a_T-a_t)(1/p-1/q)}
\left\Vert K_{s,t}^{\theta}\right\Vert _{L^{2p}\left(\nu_{t;t}\right)\to L^{2q}(\nu_{s;t})}
\nonumber \\
&\le \Big( e^{a_T-a_t} \varphi(t-s) \Big)^{(1/p-1/q)}\exp\Big(
q^{-1} \chi(\theta q) (t-s) \Big)\,.
\label{eq:p-q}
\end{align}
Next, proceeding 
similarly to the proof of \cite[Lemma 2.2]{HS}, set
$\eta_{j}=c_0 j^{-2}$ such that $\sum_{j \ge 2} \eta_{j}=1$ and
partition $T$ into (non-increasing) 
blocks $\ell_{j}=\left\lfloor \eta_{j} T \right \rfloor \ge 1$
for $2\le j \le m \le \sqrt{c_0 T}$ and $\ell_{1}=T - \sum_{j=2}^{m} \ell_{j}$. 
We further set the corresponding strictly decreasing 
\begin{equation}
t_{j}= \sum_{k=j+1}^{m} \ell_{k}\,,\qquad \qquad 0 \le j\le m \,.
\label{eq:m-j}
\end{equation}
With 
$\underline{\wh{\nu}_0} = e^{-2 a_T} \underline{\nu_0}$,
in view of \eqref{eq:pq-to-p-inf} we get 
for any non-decreasing $q_j \ge q_0=1$, 
\begin{align}
 \left\Vert K_{0,T}^{\theta}\right\Vert _{L^{2}(\nu_{T})\to L^{\infty}(\nu_{0})} &=
\left\Vert K_{0,T}^{\theta}\right\Vert _{L^{2}(\wh{\nu}_{T})\to L^{\infty}(\wh{\nu}_{0})} \nonumber \\
& \le \kappa_0 e^{a_T/q_m} \prod_{j=1}^{m}
\left\Vert K_{t_{j},t_{j-1}}^{\theta}\right\Vert _{L^{2q_{j-1}}(
\wh{\nu}_{t_{j-1}})\to L^{2q_{j}}(\wh{\nu}_{t_{j}})} 
\quad \forall \theta \in \R \,,
\label{eq:prod-bd}
\end{align}
where $\kappa_0 := (1 \wedge \underline{\nu}_0)^{-1}$ is finite.
Further, from \eqref{cond:pert} 
we have that for $A_o=A/\log 2$,
\begin{equation}\label{doub:pert}
a_T-a_t  \le A + A_o \log \Big(\frac{T+1}{t+1} \Big) \le A + A_o \log (T/t)
 \,, \qquad \forall t \in [0,T] \,.
\end{equation}
Recall \eqref{eq:m-j} that $t_{j-1} \ge \ell_j$, hence 
by our interpolation bound \eqref{eq:p-q} and \eqref{doub:pert},
 
\[
\left\Vert K_{t_{j},t_{j-1}}^{\theta}\right\Vert _{L^{2q_{j-1}}({\wh{\nu}}_{t_{j-1}})\to L^{2q_j}({\wh{\nu}}_{t_{j}})}\le 
\Big( e^{A} (T/\ell_j)^{A_o} \, \varphi(\ell_{j}) \Big)^{(1/q_{j-1}-1/q_j)}\exp\Big(
q_j^{-1} \chi(\theta q_j) \ell_j \Big)\,.
\]
Upon plugging these bounds into \eqref{eq:prod-bd},
recalling \eqref{eq:chi-bd}, that $a_T \le 2A + A_o \log T$ for $T \ge 1$,
and our assumption that $\tau \mapsto\varphi(\tau)\tau^{\bbeta}$ 
is non-decreasing on $[0,T]$, 
we find that for $|\theta| \le \delta_\star/q_m$ and 
$\bbeta_\star := \bbeta+A_o$,
\begin{align*}
 \left\Vert K_{0,T}^{\theta}\right\Vert _{L^{2}(\nu_{T})\to L^{\infty}(\nu_{0})}
 & \le 
  \kappa_0 e^{2A} T^{A_0/q_m}
  \, 
 \prod_{j=1}^{m} \Big( (T/\ell_j)^{A_0}
 \varphi(\ell_j)  
 \Big)^{(1/q_{j-1}-1/q_{j})}
 \exp\Big(c_{1} q_j \ell_j \theta^2\Big)\\
 & \le \zeta_T \, \varphi(T) \exp\left(\bbeta_\star \gamma_T + 
 c_{1}^2 \, b_T \,  \chi(\theta) T \right)
\end{align*}
where 
$$
b_T :=  \sum_{j=1}^{m} q_j \frac{\ell_j}{T} \,, \quad
\gamma_T :=  
\sum_{j=1}^m  \Big(\frac{1}{q_{j-1}} - \frac{1}{q_j} \Big) \log(T/\ell_{j}) \,,
\quad 
\zeta_T := \kappa_0 e^{2A} \varphi(1)^{-1/q_m} T^{\bbeta_\star/q_m} \,.
$$ 
Set $q_j:=1+(\log j)_+^{2}$ and maximal
$m(T) \ge 1$ such that $q_m < 1 + \kappa_1^{-1} \log T$,
with $\kappa_1 \ge 2/\log 2$ implying $c_0 T \ge T \ge m^2$.
It yields $q_m \ge 1$ such that $\log T/q_m \le 2 \kappa_1$,
hence $\zeta_T$ is uniformly bounded. 
Further, both $\gamma_T$ and $b_T$ are 
bounded by some universal constant since 
the series $\sum_j (\log j) (1/q_{j-1}-1/q_{j})$
and $\sum_j q_j j^{-2}$ converge. Combined, these facts 
imply that  \eqref{eq:bd-hs} holds for some
$C_1(c_1)$, $C_2(A,\bbeta,\underline{\nu_0},\varphi(1))$ finite
and all $|\theta| < \delta_\star/(1+\log T)$, 
as claimed.
\end{proof}

\noindent
Applying Proposition \ref{prop:hs} yields
the following heat-kernel off-diagonal estimate.
\begin{ppn}\label{prop:gauss-ubd}
Let $\psi(\cdot)$ be such that 
$\tau \mapsto \psi(\tau) \tau^\bbeta$ is non-decreasing 
for some $\bbeta>0$. Suppose 
Markov transition probabilities $\{K_{s,t}\}$ and strictly positive 
$\sigma$-finite measures $\{\nu_s\}$, with $\nu_0 \in \cM_+(V)$ such 
that $(\nu_s K_{s,t}) \le \nu_t$ whenever $T \le s \le t \le 2T$, also 
satisfy the
\emph{$1 \to \infty$ bound}
\begin{equation}
\left\Vert K_{s,t}\right\Vert _{L^{1}(\nu_{t})\to L^{\infty}(\nu_{s})}
\le \psi (t-s) \,,\qquad 0 \le s \le t \le 2T\,.
\label{eq:str-on-diagonal}
\end{equation} 
Taking $\rho(\cdot)=d(\cdot,x)$, suppose the corresponding
$\{K^\theta_{s,t}\}$ of \eqref{eq:Ks} further satisfy 
the Gaffney bound \eqref{dfn:gaffney-bd}--\eqref{eq:chi-bd}
for all 
$t \ge s \in [0,2T]$ (and 
some  
non-decreasing $t \mapsto a_t$ satisfying \eqref{cond:pert}). 
Then, for some
$C_2'(A,\bbeta,{\underline{\nu_0}},\psi(1))$ finite 
and $\kappa \ge 4 C_1(c_1,\bbeta,\delta_\star)$, 
\begin{equation}\label{eq:gauss-ubd}
d(x,y) \le \kappa \delta_\star T\,, \;\; T \ge 1 \quad 
\Longrightarrow \quad 
\frac{K_{0,2T}(x,y)}{\nu_{2T}(y)}\le C'_{2} \, \psi(T)
\exp \Big(-\frac{d^{2}(x,y)}{2 \kappa T} \Big) \,.
\end{equation}
\end{ppn}

\begin{proof} 
By Riesz-Thorin 
interpolation the $1 \to \infty$ bound \eqref{eq:str-on-diagonal} implies
the $2 \to \infty$ bound \eqref{eq:on-diagonal} with 
$\varphi(\tau)=\psi(\tau)^{1/2}$.
Hence, for $\rho(\cdot)=d(\cdot,x)$, we have from Proposition \ref{prop:hs} that 
\begin{align}
\left\Vert K_{0,T}^{\theta}\right\Vert_{L^{2}(\nu_{T})\to L^{\infty}(\nu_{0})}\le C_{2} \psi(T)^{\frac{1}{2}}\exp\left(C_{1} \chi(\theta) T \right)\,,
\quad \forall T \ge 1, \;\; \theta \in \R\,.
\label{2-to-infty}
\end{align}
Considering the adjoint $K_{s,t}^{\star}$ of 
$K_{s,t}:L^{2}(\nu_{t})\to L^{2}(\nu_{s})$, we have 
by duality that 
\begin{align}
\left\Vert K_{T,2T}^{\theta}\right\Vert_{L^{1}(\nu_{2T})\to L^{2}(\nu_{T})} 
&=
\left\Vert w_\theta K_{T,2T}^{\star} w_{-\theta} \right\Vert _{L^{2}(\nu_{T})\to L^{\infty}(\nu_{2T})} \,.
\label{1-to-2}
\end{align}
Setting $\left(K_{s,t}^{\star}\right)^{-\theta}=w_{\theta}K_{s,t}^{\star}w_{-\theta}$,
we further have by duality that 
\begin{align}
\left\Vert K_{s,t}^{\star}\right\Vert _{L^{1}(\nu_{s})\to L^{\infty}(\nu_{t})} & =\left\Vert K_{s,t}\right\Vert _{L^{1}(\nu_{t})\to L^{\infty}(\nu_{s})}\le \psi(t-s),\label{1-to-infty-star}\\
\left\Vert \left(K_{s,t}^{\star}\right)^{-\theta}\right\Vert _{L^{2}(\nu_{s})\to L^{2}(\nu_{t})} & =\left\Vert K_{s,t}^{\theta}\right\Vert _{L^{2}(\nu_{t})\to L^{2}(\nu_{s})}\le\exp\big((t-s) \chi(\theta) \big)\,,
\label{2-to-2-star}
\end{align}
where the identity in \eqref{2-to-2-star} holds since
$\left(K_{s,t}^\star\right)^{-\theta}$
is the adjoint of $K_{s,t}^{\theta} : L^2(\nu_{t}) \to L^2(\nu_{s})$.
Recall that the adjoint $K^\star_{s,t}$ of each Markov operator $K_{s,t}$ 
is a non-negative linear operator. Further, our assumption that
$(\nu_{s} K_{s,t}) \le \nu_t$ for strictly positive $\{\nu_s\}$
when $T \le s \le t \le 2T$, 
yields for $f_{t,y} := (1/\nu_t(y)) \delta_y$
any $y \in V$ and $[-1,1]$-valued $g \in \cC_0(V)$, 
\begin{equation}\label{dual-contr}
|(K_{s,t}^\star g) (y)| = 
|\langle f_{t,y}, K_{s,t}^\star g \rangle_{\nu_t}| 
= |\langle g, K_{s,t} f_{t,y} \rangle_{\nu_{s}}| 
\le
\langle K_{s,t} f_{t,y} \rangle_{\nu_{s}} 
= \frac{(\nu_{s} K_{s,t})(y)}{\nu_t(y)}
\le 1 
\,.
\end{equation}
Thus, $\left\Vert K^\star_{s,t} \right \Vert_{L^\infty(\nu_{s}) \to L^\infty(\nu_t)} \le 1$, with
\eqref{1-to-infty-star} and \eqref{2-to-2-star} allowing us 
to apply Proposition \ref{prop:hs} with $\varphi(\tau)=\psi(\tau)^{1/2}$ 
for the adjoint operators on time interval $[T,2T]$, to get that   
\begin{align}
\left\Vert \left(K_{T,2T}^{\star}\right)^{-\theta}\right\Vert _{L^{2}(\nu_{T})\to L^{\infty}(\nu_{2T})}\le C_{2} \psi(T)^{\frac{1}{2}}
\exp\left(C_{1} \chi(\theta) T\right)\,.
\label{duality}
\end{align}
With $C_2'=C_2^2$, upon combining the latter bound with  
\eqref{2-to-infty} and \eqref{1-to-2} we deduce that  
\begin{equation}\label{eq:op-s-norm}
\left\Vert K_{0,2T}^{\theta}\right\Vert _{L^{1}(\nu_{2T})\to L^{\infty}(\nu_{0})}\le 
C_{2}' \, \psi(T) \exp\left(2 C_{1} \chi(\theta) T \right)\,.
\end{equation}
Since $\rho(y)-\rho(x)=d(y,x)$,
specializing this operator 
bound to test function $f(y)=\delta_{y}$ yields
\begin{equation}\label{eq:kxy-ubd}
K_{0,2T}(x,y) e^{\theta d(x,y)} \le C_2'\, \psi(T) 
\exp\left(2 C_{1} \chi(\theta) T \right)\nu_{2T}(y)  \,\qquad \forall y \in V \,.
\end{equation}
In view of \eqref{eq:chi-bd}, taking $\theta=d(x,y)/(\kappa T) 
\le \delta_\star$ in \eqref{eq:kxy-ubd}, establishes 
the bound 
\eqref{eq:gauss-ubd}.
\end{proof}


The next lemma, is part of the (discrete) integral 
maximum principle of \cite[Theorem 2.2, Proposition 2.3]{CGZ}
and key to our proof of the Gaffney bound \eqref{dfn:gaffney-bd}.

\begin{lem}\label{gaffney-ineq}
Suppose $\pi$ is $\sigma$-finite measure and $K$ is a $\pi$-reversible, 
bounded range Markov transition on $V$. Then, for  
$f$ strictly positive and $u \in \cC_0(V)$,
\begin{equation}\label{eq:cgz-conc2}
2 \langle f u (K u - u) \rangle_\pi
\le \frac{1}{4} \langle u^2 f^{-1} K |\nabla f|^2 \rangle_\pi \,,
\end{equation}
and for $\alpha_l$-uniformly lazy $K$, also
\begin{align}\label{eq:cgz-conc}
\langle f (K u)^2 - f u^2 \rangle_\pi \le 
\frac{1}{4\alpha_l} 
\langle u^2 f^{-1} K |\nabla f|^2 \rangle_\pi \,.
\end{align}
\end{lem}
\begin{proof} For bounded range $K(x,y)$, any $u \in \cC_0(V)$ and $f$, 
the function
\begin{equation}\label{dfn:K-grad-f}
(K (\nabla u \nabla f)) (x) := \sum_{y \in V} K(x,y) (u(y)-u(x)) (f(y)-f(x)) 
\end{equation}
is in $\cC_0(V)$. Following the algebra of \cite[Eq. (2.7)-(2.8)]{CGZ}, 
if such $K(x,y)$ is $\pi$-reversible then
\begin{align} 
 2 \langle f u (Ku - u) \rangle_\pi & = - \langle K (\nabla fu) (\nabla u) 
 \rangle_\pi = 
- \langle f K |\nabla u|^2 \rangle_\pi 
- \langle u K (\nabla u \nabla f) \rangle_\pi \,.
\label{eq:28-cgz}
\end{align}
Further, as in \cite[proof of Theorem 2.2]{CGZ}, for 
strictly positive $\alpha f$,
\begin{align}
- \alpha   
 \langle f K |\nabla u|^2 \rangle_\pi 
- \langle u K (\nabla u \nabla f) \rangle_\pi 
\le \frac{1}{4 \alpha} \langle u^2 f^{-1} K |\nabla f|^2 \rangle_\pi \,,
\label{eq:cgz-thm22}
\end{align}
which for $\alpha=1$ yields \eqref{eq:cgz-conc2} when 
combined with \eqref{eq:28-cgz}. 
Next 
recall as in \cite[Eq. (2.9)]{CGZ}, that 
for $\alpha_l$-uniformly lazy Markov transition $K$ 
and any $u \in \cC_0(V)$, by Cauchy-Schwarz 
\[
(Ku -u)^2(x) \le (1-\alpha_l) (K |\nabla u|^2)(x) \,, \qquad \forall x \in V \,.
\] 
Multiplying by $f(\cdot) \ge 0$ and integrating over the 
$\sigma$-finite measure $\pi$, results with 
\begin{equation}\label{eq:210-cgz}
\langle f (Ku -u)^2 \rangle_\pi 
\le (1-\alpha_l) \langle f K |\nabla u|^2 \rangle_\pi \,.
\end{equation}
For bounded range $K(\cdot,\cdot)$ all functions are in $\cC_0(V)$, 
so combining \eqref{eq:28-cgz}, 
\eqref{eq:210-cgz} and
\eqref{eq:cgz-thm22} 
we have
\begin{align*}
\langle f (Ku)^2 \rangle_\pi - 
\langle f u^2 \rangle_\pi &= 
\langle f (Ku -u)^2 \rangle_\pi 
+ 2 \langle f u (Ku - u) \rangle_\pi \nonumber \\
&\le 
- \alpha_l \langle f K |\nabla u|^2 \rangle_\pi 
- \langle u K (\nabla u \nabla f) \rangle_\pi \le
\frac{1}{4 \alpha_l} \langle u^2 f^{-1} K |\nabla f|^2 \rangle_\pi \,,
\end{align*}
as stated in \eqref{eq:cgz-conc}.
\end{proof}

We proceed to 
establish the Gaffney bound \eqref{dfn:gaffney-bd}
for non-decreasing 
$t \mapsto \pi_t \in \cM_+(V)$ and bounded range 
$\pi_t$-reversible Markov operators $K_t$.
\begin{lem}\label{gaffney} \emph{[The Gaffney lemma]}
Suppose that Markov operators $K_t$:
\newline
(a) have reversible measures $\pi_t \in \cM_+(V)$ 
with $t \mapsto \pi_t(x)$ non-decreasing for any $x \in V$.
\newline
(b) have uniformly bounded range. That is, for some $r_{0}<\infty$ 
\begin{equation}
\left\{ y\in V:\ K_{t}(x,y)>0\right\} \subset \B (x,r_{0}) \,,
\qquad \forall x \in V, t \in \N \,.
\label{bdd-range}
\end{equation}
\newline
(c) in case of discrete time, also 
$\inf_{t,x} \{K_t(x,x)\} \ge \alpha > 0$.

\noindent
Then, the $2 \to 2$ bound \eqref{dfn:gaffney-bd} holds 
for 
$K_{s,t}^\theta : L^2(\pi_t) \to L^2(\pi_s)$ provided
\[
L_\rho := \sup_{x \ne y \in V} \Big\{ \frac{|\rho(x)-\rho(y)|}{d(x,y)} \Big\} < \infty \,,
\] 
where $a_t \equiv 0$, 
for the dynamics \eqref{def:k-st} we have $\delta_\star=1$, 
$\alpha=1$ and $\chi(\cdot)=\zeta(\cdot)$ for
\begin{equation}\label{dfn:zeta}
\zeta(\theta) := \frac{1}{8\alpha} (e^{2 r_0 L_\rho |\theta|}-1)^2 \,,
\end{equation}
whereas for \eqref{def:k-mn} set $\delta_\star = \infty$ and 
$\chi(\theta)=c_1 \theta^2$ with
$c_1 := \frac{1}{2} \sup_{\theta} \{\theta^{-2} \log (1+ 2 \zeta (\theta))\}$ finite. 
\end{lem}
\begin{proof} For the $L^2(\pi_t)$-closure of non-negative 
$K^\theta_{s,t}$
it suffices to get \eqref{dfn:gaffney-bd} for $0 \le g \in \cC_0(V)$, 
namely 
\begin{align}\label{eq:mon:J_t}
e^{2 \chi(\theta) s} \,
\Vert K^\theta_{s,t} \, g \Vert_{L^2(\pi_{s})}^2 &\le 
e^{2 \chi(\theta) t} \Vert \, K_{t,t}^\theta \, g \Vert_{L^2(\pi_t)}^2 \,,
\quad \forall s \in [0,t] \,, \quad  0 \le g \in \cC_0(V) \,.
\end{align}
For $f_s(x):=w_{-2\theta}(x) e^{2 \chi(\theta) s}$ and $u \ge 0$ solving 
\eqref{dfn:sol-cyl} on $[0,t] \times V$, 
with $u_t(\cdot)=u(t,\cdot) \in \cC_0(V)$, let
\begin{align}
E_s (u)&:= \sum_{x \in V} f_s (x) u_s^2(x) \pi_{s}(x) \,, \qquad s \in [0,t]\,.
\label{def:J_t}
\end{align}
In particular, \eqref{eq:mon:J_t} amounts to $E_s(u^{(\infty)}) \le
E_t(u^{(\infty)})$ for 
$u^{(\infty)}_s = K_{s,t} u^{(\infty)}_t$ and 
$u^{(\infty)}_t = w_\theta g \in \cC_0(V)$ 
(which for \abbr{csrw} is absolutely continuous, see \eqref{def:k-st}). For 
large 
enough
$R_k \uparrow \infty$ consider the solution $u^{(k)} \ge 0$ 
of \eqref{dfn:sol-cyl} on $Q:=Q(0,t;z,R_k)$ with 
$u^{(k)}_t=u^{(\infty)}_t$ and 
$u^{(k)} \equiv 0$ outside $Q$ (which correspond
to the transition probabilities \eqref{def:k-mn} or \eqref{def:k-st}, 
killed at exiting $\B(z,R_k)$). By monotone convergence
$E_s(u^{(k)}) \uparrow E_s(u^{(\infty)})$ with equality at 
$s=t$ and $k$ large (c.f. \cite[(3.10)]{Fo} for such argument), 
and we thus proceed to show more generally that   
$s \mapsto E_s(u)$ is non-decreasing on $[0,t]$
for any solution $u$ of \eqref{dfn:sol-cyl}
on finite time-space cylinder $Q$, with zero boundary conditions 
(hence with $u_s \in \cC_0(V)$ at any $s \le t$).
To this end, with $f_s \ge 0$ and non-decreasing $s \mapsto \pi_s(x)$  
bounded on $Q$, clearly 
$f_s u_s^2 \pi_s \le f_t u_t^2 \pi_t 
+ \int_s^t \pi_\xi \partial_{-\xi} (f_\xi u_\xi^2) d \xi$ at each 
$(s,x) \in Q$.
Thus, it suffices to show that (as a distribution in case of \abbr{csrw}), 
\begin{equation}\label{eq:der:E_t}
\Delta E_s (u) := \langle \partial_{-s} (f_s u_s^2) \rangle_{\pi_s} \le 0 \,,
\qquad \hbox{for a.e.} \;\; s \in (0,t] \,. 
\end{equation}
With $f_{s'}$ strictly positive and $\rho$ Lipschitz, 
by the uniform bounded range assumption 
\eqref{bdd-range}, 
\begin{equation}\label{eq:fx-chi}
\frac{1}{8\alpha} f_{s'}^{-1} K_s (\nabla f_{s'})^{2} \le 
\zeta(\theta) f_{s'} \,, \qquad \forall s,s' \ge 0 \,.
\end{equation}
For the 
dynamics of \eqref{def:k-mn}, since $u_{s'}=K_s u_s$
and $f_{s} = e^{2 \chi(\theta)} f_{s'}$ for $s'=s-1$, we have that 
\begin{equation}\label{eq:der:E_t-d}
\Delta E_s(u) = 
\langle f_{s'} u_{s'}^2 - f_s u_s^2 \rangle_{\pi_s}
= \langle f_{s'} (K_s u_s)^2 - f_{s'} u_s^2 \rangle_{\pi_s} 
- \langle h_s u_s^2 \rangle_{\pi_s} \,,
\end{equation}
where $h_s := (e^{2\chi(\theta)} -1) f_{s'}$. 
Similarly, for the 
dynamics of \eqref{def:k-st},
since $\partial_{-s} f_s = - 2 \chi(\theta) f_s$ and 
a.e. $\partial_{-s} u_s = (K_s -I) u_s$ (unless $u_s=0$ by
our zero boundary condition), we find that a.e.
\begin{equation}\label{eq:der:E_t-c}
\Delta E_s(u) = \langle 2 u_s f_s \partial_{-s} u_s + u_s^2 \partial_{-s} f_s 
\rangle_{\pi_s} = 2 \langle u_s f_s (K_s-I) u_s \rangle_{\pi_s} 
- \langle h_s u_s^2 \rangle_{\pi_s} \,,
\end{equation}
now with $h_s:=2 \chi(\theta) f_s$. In view of \eqref{eq:fx-chi}, taking 
$\alpha=1$, $\chi_c(\cdot)=\zeta(\cdot)$ and $s'=s$
in the continuous time setting, while 
$\chi_d (\theta) = c_1 \theta^2 \ge \frac{1}{2} \log \big( 1+ 2 \zeta(\theta) \big)$
in discrete time, yields that in both cases
\[
\frac{1}{4\alpha} 
\langle u_s^2 f_{s'}^{-1} K_s |\nabla f_{s'}|^2 \rangle_{\pi_s} 
\le 
\langle u_s^2 h_s \rangle_{\pi_s} 
\,.
\]
Thus, having $\pi_s$-reversible $K=K_s$, strictly positive
$f=f_{s'}$ and $u=u_s \in \cC_0(V)$, upon combining 
\eqref{eq:cgz-conc2} and \eqref{eq:der:E_t-c},
or \eqref{eq:cgz-conc} and \eqref{eq:der:E_t-d},
we get \eqref{eq:der:E_t}
for both the continuous and discrete time dynamics. 
To complete the proof of the lemma, just confirm that \eqref{eq:chi-bd} 
holds for $\chi_c(\cdot) = \zeta(\cdot)$, 
$\delta_\star=1$ and some $c_1$ finite.
\end{proof}

%

\noindent
\emph{Proof of Theorem \ref{thm-main}(a):}
In case of lazy \abbr{dtrw} it suffices to consider $d(x,y) \le t-s$
where the bound of \eqref{ghku-pi} is merely the 
conclusion of Proposition \ref{prop:gauss-ubd} for $\psi(k)=C'/v(\sqrt{2k})$,
$\delta_\star=\infty$, 
the dynamic \eqref{def:k-mn} for $\{(K_{r+s},\pi_{r+s}): r \in (0,t-s] \}$ 
and $\nu_r=\pi_r$. Indeed, the required $1 \to \infty$ 
bound \eqref{eq:str-on-diagonal} is provided by 
Lemma \ref{P+VG-to-Nash}, whereas the $2 \to 2$ Gaffney bound
of \eqref{dfn:gaffney-bd} is proved for $a_t \equiv 0$
and $\rho(\cdot)=d(\cdot,x)$, 
in Lemma \ref{gaffney}. The same applies for the \abbr{csrw}, except 
that now $\delta_\star=1$ in the $2 \to 2$ Gaffney bound, hence also in \eqref{eq:gauss-ubd}. Nevertheless, in this case \eqref{eq:kxy-ubd} holds 
with $\chi(\theta) \le \frac{1}{4} \exp(4|\theta|)$, so considering 
$\theta = \frac{1}{4} \log( d(x,y)/(2 C_1 T) )$ for $T=(t-s)/2$
yields the stated bound \eqref{ghku-pi}.
\qed

Further use of the integral maximum principle as in \cite[Prop. 2.5]{CGZ} 
yields the following lemma that we shall use in the sequel to strengthen 
the preceding \abbr{ghku}.
\begin{lem}\label{lem:new-weight}
Let $I(r)=r^2$ for the dynamics \eqref{def:k-mn} 
and $I(\cdot)$ as in Theorem \ref{thm-main}(a) for the dynamics \eqref{def:k-st}.
Then, in the setting of Lemma \ref{gaffney}, for 
\begin{align}\label{new-weight}
f_s(x):=\exp\Big(-\eta (s+1) I\big(\rho(x)/(s+1)\big)\Big), 
 \qquad  x\in V,  s \in \Z_+,
\end{align}
the function $s\mapsto E_s(u)$ of \eqref{def:J_t} is non-decreasing  
provided $\inf_x \rho(x) \ge 1$, $u_t \in \cC_0(V)$, and
$\eta \in [0,c_2^{-1}]$ for some 
$c_2 (L_\rho r_0,\alpha)$ finite.
\end{lem}
\begin{proof} Following the proof of Lemma \ref{gaffney},
consider first the discrete dynamic \eqref{def:k-mn}. Then, by
\eqref{eq:cgz-conc} and \eqref{eq:der:E_t-d} 
it suffices to find $c_2<\infty$ such that for $c_2 \eta \in [0,1]$ and
$s'=s-1 \ge 0$ 
\begin{align}\label{new-space-deriv}
\frac{1}{4\alpha} f_{s'}^{-1} (K_s|\nabla f_{s'}|^2) \le h_s \,,
\end{align}
where for the strictly positive $f_s$ of \eqref{new-weight},
at $x \in V$,
\begin{align*}
h_s(x) :=  f_{s}(x) - f_{s'}(x) = f_{s'}(x) \big(e^{\frac{\eta \rho(x)^2}{(s'+1)(s+1)}}-1\big)\,. 
\end{align*}
Next, if $d(x,y)\le r_0$ then $|\rho(x)-\rho(y)|\le L_\rho r_0$ and
as $\rho(x) \ge 1$ also 
$|\rho(x)^2 -\rho(y)^2| \le c_3 \rho(x)$
for $c_3:=L_\rho r_0(2+L_\rho r_0)$ finite. In this case,  
the 
inequality $|e^w-1|\le e^{|w|}-1$ yields
\begin{align*}
|f_{s'}(y)-f_{s'}(x)|\le f_{s'}(x)\Big(e^{\frac{|\eta| |\rho(x)^2-\rho(y)^2|}{s'+1}}-1\Big)
\le f_{s'}(x)\Big(e^{\frac{c_3 |\eta| \rho(x)}{s'+1}}-1\Big)\,.
\end{align*}
Next, recall that for any $\alpha>0$ there exists $b(\alpha)$ finite,
such that $(4\alpha)^{-1}(e^w-1)^2\le e^{b(\alpha)w^2}-1$ 
for any $w \ge 0$. Thus, by 
assumption 
\eqref{bdd-range}, the \abbr{lhs} of \eqref{new-space-deriv} is 
bounded above by
\[
f_{s'}(x)\Big(e^{\frac{b(\alpha) c_3^2 \eta^2 \rho(x)^2}{(s'+1)^2}}-1\Big) 
\le h_s(x)\,,
\]
provided  
non-negative $\eta \le (2 b(\alpha) c_3^2)^{-1}$ is chosen in \eqref{new-weight}. Turning to the dynamic \eqref{def:k-st}, by \eqref{eq:cgz-conc2} and \eqref{eq:der:E_t-c} it similarly suffices to show that 
\[
\sup_{y \in \B(x,r_0)} |f_s(x)-f_s(y)|^2 \le 2 \partial_{s} f_s^2(x) \,,
\]
which for $f_s(\cdot)$ of \eqref{new-weight}, 
$c_3 = L_\rho r_0$ and an 
$L_\rho$-Lipshitz function $\rho(\cdot) \ge 1$, 
follows from 
\[
\sup_{\delta \in (0,r]} \;\; \sup_{|r'-r|\le c_3 \delta}
\Big\{ e^{\eta |I(r') -I(r)|/\delta}-1 \Big\}^2
\le 4 \eta [r I'(r) - I(r)] 
\]
(take $\delta^{-1} =(s+1)$, $r=\rho(x) \delta$ and $r'=\rho(y) \delta$). 
Further, with $r I'(r) - I(r) = r (r \wedge 1)$
and $I'(r)=2 (r \wedge 1) + (\log r)_+$ non-decreasing on $\R_+$, 
it suffices in turn to verify that 
\[
g(r) := \exp\big( \eta c_3 I'( (1+c_3) r)  \big) - 1 - 2 
\sqrt{ \eta r (r \wedge 1) } \le 0 \,, \qquad \forall r > 0 \,.
\]
To this end, note that $g(0)=0$ and it is not hard to check that 
$g'(r) \le 0$ whenever $2 \eta c_3 \le 1$ and
$\sqrt{\eta} c_3 (1+c_3) e 
\le 1$. That is, for any 
non-negative $\eta \le 1/c_2(c_3)$, as claimed.
\end{proof}
\end{section}

\begin{section}{Parabolic Harnack Inequality}\label{sect:harnack}
We adapt here Grigor'yan's approach \cite{Gr} to proving \abbr{PHI} to the case of continuous time heat equation \eqref{dfn:sol-cyl} on graphs associated with the \abbr{CSRW}, when $t\mapsto\pi_t(x)$ are non-decreasing and uniformly bounded.
Building on weighted Poincar\'{e} and $L^{2}$-mean-value inequalities,  
the crucial element of the proof is a \emph{first growth lemma} (here
Lemma \ref{first-gr}). Combining such first growth lemma with the uniform volume doubling condition,
one then derives the \emph{second growth lemma} (here Lemma \ref{second-gr}),
which yields the Harnack inequality by a quite intricate, but by now
classical, argument. We thus proceed with the weighted Poincar\'{e} inequality
of \cite[Prop. 2.2]{Del}.
\begin{ppn}\emph{[Weighted Poincar\'{e} inequality]}\label{weighted}
\newline
Suppose $\pi$ has \abbr{VD} property with constant $\Cd$ and
the Poincar\'{e} inequality with constant $\CP$ holds for 
uniformly elliptic, $\pi$-reversible, Markov transition 
$K(\cdot,\cdot)$ 
on $
E$.
Then, there exist $\CPw(\CP,\Cd,\alpha_e)$ finite,
such that for $\B:=\B (z,2r)$ 
and $\eta(\cdot):=\big\{[1 - d(\cdot,z)/(2r)]_+\big\}^2$,
\begin{align}\label{eq:w-poinc}
\CPw \, r^{2} \,  \langle (\eta_\wedge K) |\nabla f|^2 \rangle_\pi 
\ge \frac{\pi(\bH_f)}{\pi(\B)} \langle \eta f^2 \rangle_\pi 
\qquad \forall r>0, z \in V, f : V \to \R_+ \,
\end{align}
where $\bH_f := \B(z,r) \cap f^{-1}(\{0\})$, 
$\eta_\wedge(x,y) := \eta(x) \wedge \eta(y)$
and $(\eta_\wedge K) |\nabla f|^2$ is as in \eqref{dfn:K-grad-f}.
\end{ppn}
\begin{proof} From \cite[Prop. 2.2]{Del} we have the weighted Poincar\'{e}
inequality 
\[
\frac{\CPw}{4} \, r^{2} \, \langle (\eta_\wedge K) |\nabla f|^2 \rangle_\pi 
\ge 
\langle \eta f^2 \rangle_\pi - \frac{\langle \eta f \rangle_\pi^2}
{\langle \eta \rangle_\pi} \ge 
\frac{\langle \eta {\bf 1}_{\{f=0\}} \rangle_\pi}{\langle \eta \rangle_\pi} 
\langle \eta f^2 \rangle_\pi \,,
\] 
where the right-inequality is merely 
Cauchy-Schwarz for $f {\bf 1}_{f>0}$. Since
the $[0,1]$-valued 
$\eta(\cdot)$ is 
supported on $\B$ and exceeds $1/4$ throughout $\B(z,r)$,
we arrive at \eqref{weighted}.
\end{proof}

The next ingredient is $L^{2}$-mean value inequality 
(denoted \abbr{ML$^2$}), analogous to the one in \cite[Sect. 4.1]{CG} for
uniformly lazy \abbr{DTRW} on time-invariant graph. 
To this end, recall 
first that a $\pi$-reversible
Markov transition $K(x,y)$ satisfies a relative Faber-Krahn (\abbr{FK})
inequality if there exist positive $a,\nu$ such that 
\begin{equation}
\lambda_{K}(\Omega)\ge\frac{a}{r^{2}}\left(\frac{\pi(\B)}{\pi(\Omega)}\right)^{\nu},
\qquad \forall r>0, \,z\in V, \, \Omega \subseteq \B, |\Omega| \ge 1 \,, 
\label{eq:FK}
\end{equation}
where $\lambda_{K}(\Omega)$ of \eqref{def-eig} is the smallest eigenvalue
of $I-K$ with Dirichlet boundary condition in $\Omega$. By \cite[Proposition 2.3]{CG},
the \abbr{FK} inequality (\ref{eq:FK}) follows from the \abbr{vd}
property and Poincar\'{e}  inequality, with constants $a,\nu$
that depend only on $\Cd$ and $\CP$. Proceeding to 
adapt the relevant part of 
\cite[Sect. 4]{CG} to our continuous time-varying setting, 
for Markov kernels $\{K_t\}$ of uniformly bounded range 
(as in \eqref{bdd-range}), we denote by
$\bpi(\cdot)$ the $\sigma$-finite 
measure on $[0,\infty) \times V$ such that  
\[
\bpi(S)= \int_0^{\infty} \big[ \sum_{x\in V}
\mathbf{1}_{\{(t,x)\in S\}} \pi_t(x) \big] \, dt 
\]
and call 
$u:Q \mapsto \R_+$ a \emph{super-solution}
(of the heat equation) on
$Q=Q(t_1,t_2;z,R)$ of \eqref{dfn:cyl} if 
\begin{equation}\label{dfn:sub-sol-cyl}
\partial_{-s} u(s,x) \ge \sum_{y} K_s(x,y) u(s,y) - u(s,x)\,,\quad 
\forall (s,x) \in Q \,,
\end{equation}
for some non-negative boundary values outside $Q$ (restricting 
to $s \in \N$ in discrete time, while for \abbr{csrw} the
inequality is between distributions and holds a.e.).
Similarly, 
$u \ge 0$ is called a \emph{sub-solution} on $Q$ 
when the reversed inequality \eqref{dfn:sub-sol-cyl} 
holds (see
\cite[Sec. 2.2]{Del}). 

\begin{remark}\label{rmk:sub-sol-cvx}
If $u \ge 0$ is a solution of \eqref{dfn:sol-cyl} on $Q$,
it must satisfy there \eqref{def:k-st} for \abbr{csrw} stopped upon exiting  
$\B(z,R)$. 
For any $\Phi(\cdot)$ convex, 
$v=\Phi(u)$ is then absolutely continuous on $Q$, and by Jensen's inequality
has \abbr{LHS} $\le$ \abbr{RHS} in \eqref{def:k-st} (throughout $Q$).
Taking $s \uparrow t$ we deduce that 
$v$ is a sub-solution on $Q$. Likewise, $v=\Phi(u)$ is
a super-solution on $Q$ whenever $\Phi(\cdot)$ is concave.
\end{remark}

\begin{ppn}\emph{[$L^2$-mean value inequality]}\label{csrw-ML1}
\newline
Suppose $t \mapsto \pi_t(x)$ is non-decreasing with $\CB:=\sup_{t,x}
\{\frac{\pi_t(x)}{\pi_0(x)}\}$ finite and the 
$\pi_t$-reversible, Markov operators $K_t$ satisfy
\eqref{bdd-range} and the relative \abbr{FK} inequality 
with same positive $a,\nu$.
Then, for $\vartheta(t):=\max\{t,t^{-1/\nu}\}$, some $C=C(a,\nu,\CB)<\infty$
any $T \ge 2 t \ge 4$, $R>r_0$, $z \in V$ and 
sub-solution $u(\cdot,\cdot)$ on $Q:=Q(T-2t,T;z,R)$ of \eqref{dfn:sub-sol-cyl}, 
\begin{align}
\hbox{\abbr{ML}}^2: 
\quad\quad u^2(T-t,z)\le \frac{C\vartheta(t/R^{2})}{\bpi(Q)}\int_{Q}u^2 d\bpi
\,.
\label{ML2}
\end{align}
\end{ppn}

\begin{proof}
We follow closely the argument
in \cite[Sect. 4]{CG}, starting with the analogue 
of \cite[Corollary 4.7]{CG}. To this end, for any functions $u,g$ on $V$, 
\begin{align*}
|\nabla (gu) |^2 - (\nabla g^2 u) (\nabla u) = u(x)u(y) |\nabla g|^2
\le \frac{1}{2} u^2(x) |\nabla g|^2 + \frac{1}{2} u^2(y) |\nabla g|^2 \,.
\end{align*}
Hence, for any $\pi$-reversible operator $K(x,y)$ on $V$ and $g \in \cC_0(V)$,
\begin{equation}\label{CG-4.8}
2 \cE_{K,\pi}(gu,gu)
+ 2 \langle g^2 u (K u - u) \rangle_\pi 
\le \langle u^2 K |\nabla g|^2 \rangle_\pi 
\end{equation}
(recall \eqref{dirichlet} and the \abbr{lhs} of \eqref{eq:28-cgz}).
Fix any $\eta(s,x)$ supported on finite time-space 
region $[T-2t,T] \times \Omega$ with  
$\eta(T,\cdot) \equiv 0$ and 
$\|(\nabla \eta)^2\|_\infty + \|\partial_s\eta^2\|_\infty \le M$.
Since $s\mapsto\pi_s(x)$ are non-decreasing, from
\eqref{CG-4.8} for $g=\eta_s=\eta(s,\cdot)$ differentiable in $s$, 
any
sub-solution $u_s=u(s,\cdot)$
and the $\pi_s$-reversible $K_s$, we get that 
at a.e. $s \in [T-2t,T]$,
\begin{align*}
& 
 2 \cE_{K_s,\pi_s}(\eta_s u_s,\eta_s u_s) 
+ \partial_{-s}\langle \eta_s^2u_s^2\rangle_{\pi_s}
\le
\langle u_s^2K_s|\nabla\eta_s|^2\rangle_{\pi_s}+\langle u_s^2\partial_{-s}\eta_s^2\rangle_{\pi_s}.
\end{align*}
Integrating both sides over $[T-\tau,T]$ yields the analogue of \cite[Eq. (4.15)]{CG}.
That is, for $\wt{\Omega} = \{ z \in V : d(z,\Omega) \le r_0 \}$ and 
any $\tau \le 2t$,
\begin{equation}\label{eq:CG-4.15}
\|\eta u\|^2_{L^2(\pi_{T-\tau})} 
+ 2 \int_{T-\tau}^{T} \cE_{K_{s},\pi_{s}}(\eta u,\eta u) \, ds
\le 2 M \int_{[T-\tau,T] \times \wt{\Omega}} \, u^2 d\bpi \,. 
\end{equation}
We proceed to adapt the proof from \cite[Sect. 4.4 \& 4.5]{CG} 
of the \abbr{ML$^2$}. Indeed, by the assumed monotonicity of $s\mapsto \pi_s$ and \emph{uniformity} 
of $a$, $\nu$, here the relative Faber-Krahn inequality \eqref{eq:FK}
yields that for any $s \ge0$, $z \in V$, $r>0$ and 
non-empty $\Omega \subseteq \B(z,r)$,
\begin{equation}\label{eq:CG-4.19}
\lambda_{K_s}(\Omega)\ge \Lambda(\pi_s(\Omega)), \qquad 
\Lambda(\xi):=\frac{a}{r^2}\pi_0(\B(z,r))^{\nu}\xi^{-\nu} \,.
\end{equation}
With \eqref{eq:CG-4.15} and \eqref{eq:CG-4.19} 
taking the roles of \cite[Eq. (4.15)]{CG} and \cite[Eq (4.19)]{CG}, respectively,
the proof of \cite[Eq. (4.20)]{CG} 
applies verbatim, upon changing on \cite[page 681]{CG} to 
$I:=\int_\Psi u^2  d\bpi$, $I':=\int_{\Psi'}
\left(u-\theta\right)_+^2 \,d \bpi$,
for a
solution $u(\cdot,\cdot)$, constant $\theta>0$ 
and invoking hereafter the time inversion 
$s \mapsto (T-s)$ on $\Psi' \subseteq \Psi$ and 
all other time-space cylinders from \cite{CG}.
We proceed as in 
\cite[pages 685-687]{CG} to compare via \cite[Eq. (4.20)]{CG} the values 
of $I_{n-1}$ and $I_n:=\int_{\Psi_n}\left(u-\theta_n\right)_+^2d\bpi$,
for $\theta_n=\theta (2-2^{-n})$ and
decreasing cylinders $\Psi_n:=Q(T-2t+n,T-t+t_n;z,R_n)$, 
with $t_n=t 2^{-n}$ and $R_n=\lceil R_{n-1}/2 \rceil$, 
starting at $\Psi_0=Q$.
Iterating 
to $N=\max\{ n : R_n \ge r_0 + 1, t_n \ge 2\}$, we arrive at \cite[Eq. (4.36)]{CG} where
$\beta := \Lambda(1) \le \pi_0(z)^\nu$ by \eqref{eq:FK} for 
$\Omega=\{z\}$ (as $\lambda_{K_0}(\{z\})\le 1$). Setting $M=2$,
$\hat{\tau}=2$ and $\hat{T}-\hat{\tau}=T-t$, consider 
\eqref{eq:CG-4.15} for the sub-solution 
$\wh{u}:=(u-2\theta_N)_+$ (recall Remark \ref{rmk:sub-sol-cvx}),
and $[0,1]$-valued 
$\wh{\eta}$ supported on $[\hat{T}-2\hat{\tau},\hat{T}) \times \Omega$,
such that $\wh{\eta}(T-t,z)=1$ and 
$\|(\nabla \hat{\eta})^2\|_\infty + \|\partial_s\hat{\eta}^2\|_\infty \le M$.
Since $\wt{Q} = [\hat{T}-\hat{\tau},\hat{T}] \times \widetilde{\Omega} 
\subset \Psi_N$ and $\theta_N \le 2 \theta$, we have that 
\begin{align}\label{eq:singleton}
(u(T-t,z)-2\theta)_+^2 \pi_0(z) \le 
 2 M \int_{\wt{Q}} \hat{u}^2 d \bpi \le
2 M I_N \,.
\end{align}
Continuing as in \cite{CG}, we cancel the common power of 
$m(z)=\pi_0(z)$ 
from both sides of \cite[Eq. (4.38)]{CG}, en-route to \cite[Eq. (4.39)]{CG} 
and thereby to \abbr{ML$^2$} by taking $\theta=\frac{1}{3} u(T-t,z)$.     
\end{proof}

\medskip
Having the key ingredients of Prop.
\ref{weighted} and Prop. \ref{csrw-ML1}, 
we now establish the first growth lemma.

\begin{lem}\emph{[First growth lemma]}\label{first-gr}
Suppose $\{\blG_t\}$ are as in Theorem \ref{harnack}.\\
For any $\delta>0$ there exists $\varepsilon=\varepsilon(\delta,\CP,\Cd, \alpha_e,\CB)>0$ such that  
for all $T \ge 6R^2$, $z \in V$ and any positive solution  
$u(\cdot,\cdot)$ of \eqref{dfn:sol-cyl} on  
$Q:=Q\left(T-4R^{2},T;z,2R\right)$, 
\begin{equation}\label{eq:H-large-measure}
\frac{\bpi (Q(T-R^{2},T;z,R) \cap u^{-1}([1,\infty)))}{\bpi(Q(T-R^{2},T;z,R))} \ge \delta
\quad \Longrightarrow \quad
\inf_{Q(T-3R^{2},T-2R^{2};z,R)} u \ge \varepsilon \,.
\end{equation}
\end{lem}
\begin{proof} Fixing $z \in V$ and $T \ge 4 R^2$ set $\B:=\B(z,2R)$ and
$\eta(\cdot):=\big\{[1-d(\cdot,z)/(2R)]_+\big\}^{2}$ as in Prop. \ref{weighted}. Recall Remark \ref{rmk:str-pos} that  
any solution $u>0$ of \eqref{dfn:sol-cyl} on $Q$ can be replaced by
bounded away from zero solutions $u_b = (1-b) u + b$,
without altering the \abbr{lhs} of \eqref{eq:H-large-measure}. Hence, by Remark \ref{rmk:sub-sol-cvx},
\abbr{wlog} we have the associated  
super-solution $\wt{u} := 1 \wedge u \ge b$ for some $b>0$,
and 
uniformly bounded sub-solution $v := - \log \wt{u}$.
Consider the functions 
$F(s):=\langle \eta v_s \rangle_{\pi_s}$ on $[T-4R^2,T]$, 
$\wt{u}_s:= \wt{u}(s,\cdot)$ and $v_s:=v(s,\cdot)$
on $V$ and the subset $\bH_{v_s}=\B(z,R) \cap u(s,\cdot)^{-1}([1,\infty))$ 
of $\B$. Having $s\mapsto\pi_s(\cdot)$ non-decreasing and
$\wt{u} \ge b$ a super-solution, 
it follows from \eqref{dfn:sub-sol-cyl} and the \abbr{lhs} of \eqref{eq:28-cgz}
that as distributions, for a.e. $s$,
\begin{align}\label{def:I-t}
\partial_s F(s)&\ge\langle\frac{\eta}{\wt{u}_s}\partial_{-s}\wt{u}_s\rangle_{\pi_s} \ge \langle \frac{\eta}{\wt{u}_s} K_s \nabla \wt{u}_s 
\rangle_{\pi_s} =
-\frac{1}{2}\langle K_s(\nabla(\frac{\eta}{\wt{u}_s})\nabla\wt{u}_s)\rangle_{\pi_s}.
\end{align} 
Setting $\psi(c)=\frac{1}{2c}$ for $c>0$ and $\psi(0)=1$, recall 
that for any $a,b>0$ and $c,d \ge 0$, 
\begin{align*}
-(\frac{d}{b}-\frac{c}{a})(b-a)\ge\frac{1}{2}(c\wedge d)(\log b-\log a)^2
- |d-c| \psi\Big(\frac{c \wedge d}{|d-c|}\Big) 
\end{align*}
(see \cite[Inequality (1.23)]{SZh}).
For $(x,y) \in E$ and $k \in \Z_+$, 
if $\eta_\wedge = k^2 (2R)^{-2}$ then necessarily 
$|\nabla \eta| \le (2k+1) (2R)^{-2}$,
so $|\nabla \eta| \psi(\eta_\wedge/|\nabla \eta|) \le 
\frac{9}{8} R^{-2}$ (or zero, whenever $d(z,x) \vee d(z,y) > 2R$). Hence, 
upon summing over $\pi_s K_s$ we get as in \cite[proof of (5.7)]{Ba}, that
\begin{align}
- \langle K_s(\nabla(\frac{\eta}{\wt{u}_s})\nabla\wt{u}_s)\rangle_{\pi_s} 
\ge \frac{1}{2}
\langle (\eta_{\wedge} K_s) |\nabla v_s|^{2} \rangle_{\pi_s} -
 \frac{9}{4} R^{-2} \pi_s(\B) \,.   \label{delm-diff}
\end{align}
Next, by Prop. \ref{weighted} for
$ 
v_s \ge 0$ and $r=R$, followed by Cauchy-Schwartz, 
\begin{align*}
 \CPw R^{2}
 \langle (\eta_{\wedge} K_s) |\nabla v_s|^{2} \rangle_{\pi_s}
 & \ge\frac{\pi_s(\bH_{v_s})}{\pi_s(\B)}\langle \eta v_s^2 \rangle_{\pi_s}
 \;\ge\frac{\pi_s(\bH_{v_s})}{\pi_s(\B)^{2}}F(s)^2.
\end{align*}
Plugging this into \eqref{def:I-t}-\eqref{delm-diff} yields 
\begin{align*}
\partial_s F(s)\ge L(s)F(s)^2-D(s)\,, \qquad
L(s):=\frac{R^{-2} \pi_s(\bH_{v_s})}{
\CPw (2 \CB)^2 \pi_0(\B)^{2}},  \qquad     D(s):= \frac{9}{8} R^{-2}\pi_s(\B)\,.
\end{align*}
Following \cite[pg. 67]{Gr}, let $J(t) = F(t)-\int^T_t D(s) ds$ 
and $t_\star := \sup \{ t \le T : J(t) \le 0 \}$. With 
$F$, $L$, $D$ non-negative and $J(T)
\ge 0$, we have on $[t_\star,T]$
that $\partial_s J \ge L F^2 \ge L J^2$ 
and consequently  
$J(t) \le (\int_{t}^T L(s) ds)^{-1}$. Further 
$t \mapsto J(t)$ is non-decreasing and $J(t_\star)=0$,
so this bound extends to all $t$. Thus, on 
$[T-4R^{2},T-R^{2}]$,
\begin{align}\label{eq:bd-I-t}
F(t)
\le 
\Big(\int_{T-R^2}^T L(s)ds\Big)^{-1}+\int_t^T D(s) ds \,.
\end{align}
From the \abbr{LHS} of \eqref{eq:H-large-measure} 
and definition of $\bH_{v_s}$ we have 
$R^{-2} \int_{T-R^2}^T \pi_s(\bH_{v_s}) ds \ge \delta \Cd^{-1} \pi_0(\B)$,
hence the first term on the \abbr{rhs} of \eqref{eq:bd-I-t} is at most
$\delta^{-1} \Cd \CPw (2 \CB)^2 \pi_0(\B)$. The other term 
is at most $7 \CB \pi_0(\B)$, so   
for some $C_1(\CPw,\Cd,\CB)$ finite,  
\begin{align}\label{ubd:I}
F(t)\le C_1 \pi_{0}(\B) \delta^{-1}\,, \qquad \qquad
\forall \, t\in[T-4R^2,T-R^2] \,.
\end{align}
For $t\in [T-3R^2,T-2R^2]$, integrating \eqref{def:I-t} 
on $I_t := [t-R^2,t+R^2]$, yields by \eqref{delm-diff} and \eqref{ubd:I} that 
\begin{align}
\int_{I_t}
\langle (\eta_{\wedge} K_s) |\nabla v_s|^{2} \rangle_{\pi_s} ds   
&\le 4 F(t+R^{2})
+ \frac{9}{2} R^{-2} \bpi(I_t \times \B)
\le C_2\pi_{0}(\B)\delta^{-1} \,,
\label{integ-wpi}
\end{align}
where $C_2= 4 C_1 + 9 \CB$. Since 
$\zeta:={\bf 1}_{\B'} \le 16 \eta$ for $\B':=\B(z,3R/2)$, 
it follows that  
\begin{align*}
\bar{v}_s:=
\frac{\langle v_s \zeta \rangle_{\pi_s}}{\langle \zeta \rangle_{\pi_s}} 
\le \frac{16 F(s)}{\pi_0(\B')} \,,
\qquad
\cE_{K_s,\pi_s} (v_s\zeta,v_s\zeta) \le 8
\langle (\eta_{\wedge} K_s) |\nabla v_s|^{2} \rangle_{\pi_s} \,.
\end{align*}
Recall that under uniform ellipticity, the \abbr{VD} property and (weak) Poincar\'e inequality (\abbr{PI}) of \eqref{poincare}, implies the 
strong-\abbr{PI} where $\B(x_0,r)$ replaces $\B(x_0,2r)$ on the \abbr{RHS} of \eqref{poincare} (see \cite[Cor. A.51]{Ba2} or
\cite[Prop. 3.3.2]{Kum}).
From the strong-\abbr{PI} on $\B'$ and the preceding bounds,
\begin{align}
\langle v_s^2 \zeta \rangle_{\pi_s}  =
\bar{v}_s^2 \langle \zeta \rangle_{\pi_s} +
\langle (v_s-\bar{v}_s)^2 \zeta \rangle_{\pi_s}
\le  \frac{16^2 F(s)^2}{\pi_0(\B')} + 
16 \CP R^2  
\langle (\eta_{\wedge} K_s) |\nabla v_s|^{2} \rangle_{\pi_s} \,.
\label{poinc-zeta}
\end{align}
Combining \eqref{ubd:I}-\eqref{poinc-zeta}, 
we get for some $C_3(C_1,C_2,\Cd,\CP)$ finite 
and all $t\in[T-3R^2,T-2R^2]$,
\begin{align}\label{L2-bound-v}
\int_{I_t \times \B'} v^2 d\bpi \le C_3 R^2 \pi_0(\B) \delta^{-2} \, .
\end{align}
Applying the $L^2$-mean value of Prop. \ref{csrw-ML1} to 
the sub-solution $v$ on $I_t \times \B(x,R/2))$ 
together with the \abbr{VD} property of $\pi_0$, 
we get for $C_i=C_i(\CP,\Cd,\alpha_e,\CB)$ finite 
and all $(t,x)\in Q(T-3R^2,T-2R^2;z,R)$,
\begin{align*}
v^2(t,x) & \le\frac{C_4 R^{-2} }{\pi_{0}\left(\B(x,R/2)\right)}
\int_{I_t \times \B(x,R/2)} v^2 d \bpi 
 \le\frac{C_5 R^{-2}}{\pi_{0}(\B)}\int_{I_t \times \B'}
v^2 d \bpi \;\le C_6^2 \delta^{-2},
\end{align*}
using \eqref{L2-bound-v} in the last step.
That is, \eqref{eq:H-large-measure} holds with $\varepsilon= \exp(-C_6/\delta)>0$.
\end{proof}

Adapting the derivation of \cite[Lemma 4.3]{Gr} (out of \cite[Lemma 4.1]{Gr}),
yields the following consequence of Lemma \ref{first-gr}.
\begin{lem}\label{first-gr2}
For $\{\blG_t\}$ as in Theorem \ref{harnack} there exist 
finite $\eta=\eta(\CP,\Cd,\alpha_e,\CB)$ and $R_0=R_0(\CP,\Cd,\alpha_e,\CB)$
such that for any $R \ge R_0$ and all $z$, $T$, $u(\cdot,\cdot)$ as in Lemma \ref{first-gr},
\begin{equation}\label{eq:small-measure}
\frac{\bpi (Q(T-R^{2},T;z,R) \cap u^{-1}([1,\infty)))}{\bpi(Q(T-R^{2},T;z,R))} \le \eta \;\; \& \;\;
u(T-R^2,z)\ge 2
\quad \Longrightarrow \quad
\sup_{Q(T-2R^{2},T;z,R)} u \ge 4 \,.
\end{equation}
\end{lem}

\begin{remark} 
An alternative and quicker approach by Fabes-Stroock utilizes 
the weighted Poincar\'{e} inequality in a different way (e.g. \cite{FS,Ba,BK}).
It relies on having a-priori that  
\[
K_{t-s,t}(x,x)\le\frac{C}{\pi_{0}(\B(x,\sqrt{s}))}\ \quad 
\& \quad \inf_{s,t,y} \sum_{x\in\B(y,C\sqrt{s})}K_{t-s,t}(x,y) > 0 
\]
(which take the role of \eqref{eq:H-large-measure} in proving 
the first growth lemma). However, lacking a uniform in $y$ 
lower bound on $\sum_{x\in\B(y,C\sqrt{s})}K_{t-s}(x,y)$, 
prevents using this approach in our time-varying setting. 
\end{remark}

Under uniform volume doubling condition, the first growth lemma implies second growth lemma, following the same proof as \cite[Lemma 4.2]{Gr} verbatim.

\begin{lem}\emph{[Second growth lemma]}\label{second-gr}
For $\{\blG_t\}$ of Theorem \ref{harnack}, some
$\theta=\theta(\CP,\Cd,\alpha_e,\CB)$ finite,
$c=c(\delta,\CP,\Cd,\alpha_e,C_{0})>0$, 
$u(\cdot)$ as in Lemma \ref{first-gr}
and $T' \in [T-(R/2)^{2}+r^2,T]$,
\[
\frac{\bpi (Q(T'-r^{2},T';z,R) \cap u^{-1}([1,\infty)))}
{\bpi (Q(T'-r^{2},T';z,r))} \ge \delta \quad 
\Rightarrow \quad
u(T-4R^{2},z)\ge c\left(\frac{\bpi(Q(T'-r^{2},T';z,r))}
{\bpi(Q(T-R^{2},T;z,R))}\right)^{\theta}\,.
\]
\end{lem} 

\noindent
{\emph{Proof of Theorem \ref{harnack} [sketch, following \cite{Gr}].}}
In case of manifolds, the derivation of \abbr{PHI} 
from the first and second growth lemmas is standard in the literature. 
We sketch here the adaptation for discrete graphs of the
argument provided near the end of \cite[Section 4]{Gr},
where one is restricted to choose cylinders of radii at least $R_0 \ge 1$.
Specifically, with $Q(R)$ denoting the cylinder $Q(T-R^2,T;z,R)$, our
goal is to show that for some $\gamma(\CP,\Cd,\alpha_e,\CB)>0$ 
if $\sup_{Q(T-4R^2,T-3R^2;z,R)}\; u = 1$ for some 
positive solution $u(\cdot,\cdot)$ on $Q(8R)$, then necessarily 
$u(T-48R^2,z)\ge \gamma$.
To this end, for $R_0 \ge 1$ of Lemma \ref{first-gr2}
set $\eta_0:=\Cd^{-R_0}/(\CB R_0^2)$ and fix the largest 
$\eta \in (0,\eta_0]$ for which Lemma \ref{first-gr2} holds. 
By Lemma \ref{first-gr}, if
$\bpi (E_0) \ge \delta \bpi (Q(4R))$ for $E_0:=Q(4R)\cap u^{-1}([2^{-1},\infty)$
and $\delta:=\eta/(64 \CB \Cd^4)$,
then $u(T-48R^2,z)\ge \varepsilon$ for some 
$\varepsilon=\varepsilon(\delta,\CP,\Cd,\alpha_e,\CB)>0$, verifying 
the claimed \abbr{PHI} in this case. Next, suppose to the contrary that 
$\delta \bpi (Q(4R)) > \bpi(E_0)$. Recall our assumption that
$u(T-t_0,x_0)=1$ for some $t_0 \in [3R^2,4R^2]$ and $x_0 \in \B(z,R)$,
whereas
by the uniform volume doubling condition and our choice of $\delta$, 
\begin{equation}\label{eq:tianyi-contr}
\eta \bpi (Q(T-t_0-(R/2)^2,T-t_0;x_0,R/2)) \ge \delta \bpi (Q(4R)) > 
\bpi(E_0) \,.
\end{equation}
Next let $r_0$ be the maximum integer $r \le R/2$ for which  
\[
\bpi (Q(T-t_0-r^{2},T-t_0;x_0,r) \cap u^{-1}([2^{-1},\infty)))
\ge \eta \bpi(Q(T-t_0-r^{2},T-t_0;x_0,r)) \,.
\]
Having $u(T-t_0,x_0)=1$, we deduce from the uniform volume doubling condition 
and our choice of $\eta \le \eta_0$ that necessarily $r_0\ge R_0$. Further, 
in view of \eqref{eq:tianyi-contr} also $r_0< R/2$ and 
employing Lemma \ref{first-gr2}, we have that
\[
u(T-t_1,x_1) =\sup_{Q(T-t_0-2(r_0+1)^{2},T-t_0;x_0,r_0+1)} u \ge 2 \,.
\]
This procedure is iterated in the same way as at \cite[end of Section 4]{Gr}, 
with the only the change being that 
 $r_k$ is defined be the maximum integer $r$ such that 
\[
\frac{\bpi (Q(T-t_k-r^{2},T-t_k;x_k,r) \cap u^{-1}([2^{k-1},\infty)))}{\bpi(Q(T-t_k-r^{2},T-t_k;x_k,r))} \ge \eta.
\]
The proof concludes as in \cite{Gr} by choosing a good index $k$ and applying Lemma \ref{second-gr}
to the pair of cylinders $Q(T-t_k-r_k^2,T-t_k;x_k,r_k)$ and $Q(4R)$.
\qed

\medskip
\noindent
{\emph{Proof of Proposition \ref{ppn:holder}. }} 
Fixing $y_j \in \B(z,R)$ and $(T-s_j) \in [R^2,4R^2]$ such that 
$s_2 \ge s_1$, we consider nested time-space cylinders 
$Q(i):=[s_1,s_1+R_i^2] \times \B(y_1,R_i)$ 
for $R_i=2^i$, $i \ge 0$ and the corresponding
$M(i):=\sup_{Q(i)} \{u\}$, $m(i):=\inf_{Q(i)} \{u\}$ and $w(i):=M(i)-m(i)$.
With $i_2 := \sup\{ i : Q(i) \subseteq Q \}$ 
we have that $w(i_2) \le M(i_2) \le \sup_Q \{u\}$. Similarly, 
setting $r:=|s_2-s_1|^{1/2} \vee d(y_1,y_2)$ and
$i_1 := \inf\{ i : r \le R_i \}$ 
we have that $(s_j,y_j) \in Q(i_1)$, $j=1,2$, hence
$|u(s_2,y_2)-u(s_1,y_1)| \le w(i_1)$. Clearly $R_{i_1} \le 2r$, 
while $R \le 2 R_{i_2}$ (since $(T-s_1) \ge R^2$
and $\B(y_1,R) \subseteq \B(z,2R)$). The inequality \eqref{eq:holder}
is trivial unless $r \le R/4$ and thereby $i_1 \le i_2$. 
It thus suffices to show that $w(i-1) \le (1-\gamma) w(i)$ for
some $\gamma(\theta_j)$ from \eqref{eq:harnack} and all $i \le i_2$
(as then $w(i_1)\le (1-\gamma)^{i_2-i_1}w(i_2)$, 
yielding \eqref{eq:holder} for $2^{-h} = 1-\gamma$).
To this end, consider for (non-negative) 
solutions $u-m(i)$ and $M(i)-u$, 
the \abbr{PHI} in $Q(i) \subseteq Q$ with
$\theta_1=\theta_2=\frac{1}{\sqrt{2}}$, $\theta_3=\frac{\sqrt{3}}{2}$
and $\theta_4=1$,
to compare the solution at $(s_1 +2 R_{i-1}^2,y_1)$ 
with its infimum over $Q(i-1)$.
Setting $v(i)=u(s_1+ 2 R_{i-1}^2,y_1)$, these
comparisons yield by the \abbr{phI} that
\begin{align*}
\gamma (v(i)-m(i))&\le  m(i-1)-m(i),\\
\gamma (M(i)-v(i))& \le M(i)-M(i-1),
\end{align*}
and hence $w(i-1)\le (1-\gamma)w(i)$, as claimed.
\qed

\end{section}

\begin{section}{From \abbr{GHKU} and \abbr{PHI} to Gaussian lower bounds}
\label{sect:conseq} 

In this section we establish the matching \abbr{GHKE} of 
\eqref{ghke} 
out of the (weaker) upper bound \eqref{ghku-pi} and the \abbr{PHI} \eqref{eq:harnack}. To this end, we start with the following 
elementary fact.
\begin{lem}\label{lem:mass-conc}
Suppose $\Gamma_{\tau}: V \mapsto \R_+$ are such that  
\begin{equation}\label{eq:mod-ghku}
\Gamma_{\tau}(z) \le \frac{C}{v(\sqrt{\tau})}e^{-\rho_\tau(x,z)/C} \,,
\end{equation}
for some $C<\infty$, $x \in V$, 
a doubling function $v(r) \ge \CV^{-2} |\B(x,r)|$ and 
\begin{equation}
\label{dfn:rho_t}
\rho_t(x,z) := d(x,z) \Big(\frac{d(x,z)}{t} \wedge 1 \Big) \;, 
\qquad \qquad \;\; x,z \in V\,, \;\; t \ge 0 \,.
\end{equation}
Then, for some $c'(\CV,C)$ finite, 
\begin{equation}\label{eq:tail-level}
\Gamma_\tau(\B(x,R)^c) \le c'  e^{-R/(2C)} \,,
\qquad \forall R \ge \tau \ge 1\,, \;\; \forall x \in V \,.
\end{equation}
Further, if $\Gamma_\tau$ are probability measures, then 
for $b>0$ and some $\kappa(\CV,C,b)<\infty$,
\begin{equation}\label{eq:pos}
\inf_{\tau \ge b, x \in V} 
\, \{ \, \Gamma_{\tau}(\B(x,\kappa \sqrt{\tau})) \, \} \, 
\ge\frac{1}{2} \,.
\end{equation}
\end{lem}
\begin{proof} Note that 
$1 - e^{-(1 + 2 (\ell \wedge \tau))/(C\tau)} \le 3 \ell/(C\tau)$ 
for 
$\ell \ge 1$. From \eqref{eq:mod-ghku} we thus get
after summation by parts 
that 
\[
\Gamma_{\tau} (\B(x,R)^c) 
\le\frac{C}{v(\sqrt{\tau})}
\sum_{\ell>R} |\partial \B(x,\ell)| e^{-\ell(\ell \wedge \tau)/
(C \tau)} 
\le 3 \CV^2 
\sum_{\ell > R} \frac{v(\ell) \ell}{v(\sqrt{\tau}) \tau}
 e^{-\ell (\ell \wedge \tau)/(C \tau)} \,.
\]
With $v(\ell)$ doubling, one has 
$v(\ell) \le v(\sqrt{\tau}) (2 \ell/\sqrt{\tau})^{c_v}$
for $c_v := \log_2 \CV$ and any $\ell \ge \sqrt{\tau}$. 
Hence, 
for some $c'(\CV,C)$ and any $R \ge \tau \ge 1$,
\begin{equation*}
\Gamma_{\tau} (\B(x,R)^c) \le 3 \CV^2 
\int_R^\infty (4 u)^{c_v+1} e^{-u/C} du \le c' e^{-R/(2C)} 
\end{equation*}
as claimed in \eqref{eq:tail-level}. Similarly, 
for $\kappa=\kappa(\CV,C,b)$ large enough and all $\tau \ge b$,
\begin{equation*}
\Gamma_{\tau} (\B(x,\kappa \sqrt{\tau})^c) \le 3 \CV^2  \int_{\kappa}^\infty 
(4u)^{c_v+1} e^{-u(u\wedge \sqrt{b})/C} du  \le \frac{1}{2} \,,
\end{equation*}
yielding \eqref{eq:pos} in case $\Gamma_\tau(V)=1$. 
\end{proof}

We next utilize the fact  
that for $\{K_t\}$ of \eqref{n-kernel}, 
the transition probabilities $K_{\cdot,t}(\cdot,y)$ of the \abbr{CSRW} 
and \abbr{DTRW}, are 
solutions of \eqref{dfn:sol-cyl}.

\begin{lem}\label{lem:csrw-sol}
Fixing Borel measurable $\{\bPi_s\}$,  
$t \ge 0$ and $z_\star \in V$,
the functions $u(s,x) := K_{s,t}(x,z_\star)$ for  
the \abbr{CSRW} and \abbr{DTRW} associated with \eqref{n-kernel},
solve the corresponding heat equation \eqref{dfn:sol-cyl}
on the time-space cylinder $Q(0,t;z,R)$ for any $R \ge 1$
and $z \in V$.  
The \abbr{PHI} then implies that:\\
(a). For some
$\gamma(\varphi,\delta) \in (0,1)$, any $\delta \in (0,1)$, 
$\varphi \ge 1/(1-\delta)$, all $t \ge \tau \ge 1$,  
and $x_1,x_2,z_\star \in V$,  
\begin{equation}\label{eq:phi-comp}
d(x_1,x_2) \le 2 [2 \varphi \sqrt{\tau}] \quad \Longrightarrow \quad
K_{t-\tau,t}(x_2,z_\star) \ge \gamma K_{t-\delta \tau,t}(x_1,z_\star) \,,
\end{equation}
where for \abbr{dtrw} we further assume 
that $\tau \in \N$ and $d(x_1,x_2) \le (1-\delta) \tau \in \N$.
\newline
(b). Suppose in addition that
$\CB:=\sup_{s,x}\left\{\frac{\pi_s(x)}{\pi_0(x)}\right\}$ is finite, 
$s\mapsto\pi_s(\cdot)$ is non-decreasing and the \abbr{DTRW} is 
uniformly lazy. Then, 
\begin{align}\label{dfn:D}
s \mapsto D_{s,t}(y):=
\sum_{x\in V}\pi_s(x) K_{s,t}(x,y)^2 
\end{align}
is non-decreasing on $[0,t]$. Further, for 
$C_1=C_1(\gamma,\alpha_l,\CB)$ finite,
\begin{align}\label{eq:D-doubling}
C_1 D_{t-2s,t}(y) \ge D_{t-s,t}(y) \,, \quad \forall y \in V, s \ge 0 \,.
\end{align}
\end{lem}

\begin{remark}\label{d-bar-doubling} From part (b) we have that 
the non-decreasing $\bar D(s) := \pi_t(y)/D_{t-s,t}(y)$
is a doubling function, with $\bar D(0)=1$ such that 
$\bar D(2s) \le C_1 \bar D(s)$ for all $y \in V$ and $2s \le t$.
\end{remark}

\begin{proof} (a). 
From \eqref{def:k-mn} it
immediately follows that for \abbr{dtrw} the non-negative 
$u(k,x)=K_{k,t}(x,z_\star)$ 
satisfies \eqref{dfn:sol-cyl} on $Q(0,t;z,R)$ of \eqref{dfn:cyl}
(with $\partial_{-s} u(s,\cdot) = u(s-1,\cdot)-u(s,\cdot)$).
Similarly, in case of \abbr{csrw}, it follows  
from \eqref{def:k-st} that $s \mapsto u(s,x)=K_{s,t}(x,z_\star)$ is 
an absolutely continuous, 
solution of \eqref{dfn:sol-cyl} 
on $Q(0,t;z,R)$.  
Next,
if $d(x_1,x_2) \le 2R$ for integer $R:= [2 \varphi \sqrt{\tau}] \ge 1$, then 
$x_1,x_2 \in \B(z,R)$ for some $z \in V$. Further, for $\tau \ge 1$ we 
have $R+1 \ge 2 \varphi \sqrt{\tau} \ge 2/(1-\delta)$, hence 
the \abbr{PHI} applies for  
$T=t$, $\tau_2=\tau$ and
$\tau_1=\delta\tau$, with 
$\varphi \theta_1 = \sqrt{\delta}/2$,
$\varphi \theta_2 = \sqrt{\delta}/(1+\delta) < 1/2$,
$\varphi \theta_3 = 1/2$,
$\varphi \theta_4 = 1/(1+\delta)$
and the corresponding $\gamma=\gamma(\theta_i) \in (0,1)$.
\newline
(b). While proving Lemma \ref{gaffney}, we showed that $s \mapsto E_s(u)$ 
of \eqref{def:J_t} is non-decreasing whenever 
$u_s(x) \ge 0$ solves \eqref{dfn:sol-cyl}
for some $u_t \in \cC_0(V)$. By part (a) 
this applies to $s \mapsto D_{s,t}(y)$ which corresponds
to $u_s(x)=K_{s,t}(x,y)$ and $f_s(x) \equiv 1$. Next, from
\eqref{eq:phi-comp} for $x_1=x_2
$ 
we have that 
$K_{t-2s,t}(x,y)\ge \gamma K_{t-s,t}(x,y)$ for 
$\gamma(2,1/2)>0$, all $x,y \in V$, $2s \in [1,t]$,
yielding that 
$C_1 D_{t-2s,t}(y) \ge D_{t-s,t}(y)$ 
for $C_1 = \CB \gamma^{-2}$ finite.
Lowering $\gamma \le \alpha_l$ for lazy \abbr{DTRW}, or
$\gamma \le e^{-1}$ for \abbr{csrw}, we further have
that $K_{t-1,t}(y,y) \ge \gamma
$, hence $C_1 D_{t-1,t}(y) \ge D_{t,t}(y)$, extending
\eqref{eq:D-doubling} to all $2s \in [0,t]$.
\end{proof}

Utilizing the \abbr{ghku} of Theorem \ref{thm-main}(a) 
as well as Lemmas \ref{lem:new-weight}, \ref{lem:mass-conc} 
and \ref{lem:csrw-sol}(b), we next establish a moment generating
bound which is key for getting matching \abbr{ghke}.
\begin{lem}\label{lem:d-mgf}
In the setting of Theorem \ref{thm-main}(b),
for dynamics \eqref{def:k-mn} and \eqref{def:k-st}, the 
functions $\rho_t(\cdot)$
of \eqref{dfn:rho_t} and $D_{s,t}(\cdot)$ of \eqref{dfn:D}, 
some $\theta_0(\gamma,\CV,\alpha_l)>0$ and 
$C_2(\gamma,\CV,\alpha_l)$ finite,   
\begin{align}\label{bound:D-by-J} 
D_{\tau,2\tau}(y;\theta_0) := \sum_{z\in V} \pi_{\tau} (z) K_{\tau,2\tau}(z,y)^2 
e^{\theta_0 \rho_\tau(z,y)} \le C_2 D_{\tau,2\tau}(y) \,, 
\qquad \forall y \in V \,, \; \tau \ge 1 \,.
\end{align}
\end{lem}
\begin{proof} Fixing $y \in V$, $\tau \ge 1$ and $t=2\tau$, let 
\begin{equation}\label{dfn:I-sr}
J(\ell,s):=\sum_{z \in V} \pi_{t-s}(z) K_{t-s,t}(z,y)^2 {\bf 1}_{\{d(z,y) \ge \ell\}}
\,, \qquad  
\ell \ge 0, \; s \in [0,t] \,.
\end{equation}
Since $\partial_{-\ell} \{\ell (\frac{\ell}{\tau} \wedge 1)\} \ge -2 \ell/\tau$ 
and $1-e^{w} \le -w$, it follows after summation by parts that 
\[
D_{\tau,2\tau}(y;\theta) = \sum_{\ell=0}^\infty
e^{\theta \ell (\frac{\ell}{\tau} \wedge 1)} (J(\ell,\tau)-J(\ell+1,\tau))
\le J(0,\tau) + 2 \theta \sum_{\ell=1}^\infty J(\ell,\tau) \frac{\ell}{\tau} 
e^{\theta \ell (\frac{\ell}{\tau} \wedge 1)} \,.
\]
With $J(0,\tau)=D_{\tau,2\tau}(y)$ we thus get \eqref{bound:D-by-J}
upon showing that 
\begin{equation}\label{bd:I-taur}
J(\ell,\tau) \le C_2' J(0,\tau) e^{-2 \theta_0 \ell (\frac{\ell}{\tau} 
\wedge 1)} \,,
\end{equation}
for some $\theta_0(\gamma,\CV,\alpha_l)$ positive,
$C_2'(\gamma,\CV,\alpha_l)$ finite and all 
$\ell \ge \ell_0(\gamma,\CV,\alpha_l)$ finite, whereupon 
taking $C_2' \ge e^{2 \kappa \theta_0}$ it suffices 
to show \eqref{bd:I-taur} for $\ell > \sqrt{\kappa \tau}$. To this end,
we proceed by adapting the proof of \cite[Prop. 5.4]{CGZ} to 
handle both time-varying $\pi_t(\cdot)$ and the \abbr{csrw}
(see also \cite[Lemma 4.1]{Fo} for  
\abbr{csrw} with constant 
conductances). 
First apply Lemma \ref{lem:new-weight} for 
$u_s(\cdot):=K_{t-s'+s,t}(\cdot,y)$, 
with $u_{s'} \in \cC_0(V)$,
and the 
Lipschitz function $\rho(\cdot;\ell'):=\ell'+ 1 - d(\cdot,y) \wedge \ell'$, 
to deduce that for some $\eta=\eta(\alpha_l)>0$,
any $s' \in [0,t]$, $\ell' \ge 0$ and $y \in V$,
\begin{align*}
s \mapsto E_{s}(\ell',s'):= \sum_{z\in V} \pi_{t-s'+s}(z) K_{t-s'+s,t}(z,y)^2 e^{-\eta (s+1) I\big( \rho(z;\ell')/(s+1)\big)} \,,
\end{align*}
is non-decreasing on $[0,s']$. Further, $I(1)=1$ and 
$\rho(z;\ell')=1$ whenever 
$d(z,y) \ge \ell'$ while $\rho(z;\ell') \ge \ell'-\ell$ whenever 
$d(z,y) < \ell \le \ell'$. Hence, for any $s \in [0,s']$ and $\ell \le \ell'$,
\begin{align}
e^{-\eta} J(\ell',s')\le E_{0}(\ell',s') \le E_{s'-s}(\ell',s') 
\le J(\ell,s) + J(0,s) e^{-\eta (s'-s+1) 
I\big(\frac{\ell'-\ell}{s'-s+1}\big)} \,.
\label{iter-I}
\end{align}
We then get \eqref{bd:I-taur} upon recalling
Remark \ref{d-bar-doubling} that due to the \abbr{phi} the function 
$s\mapsto \bar D(s) = J(0,0)/J(0,s)$ is non-decreasing and $C_1$-doubling, 
hence regular in the sense of \cite[Def. 5.1]{CGZ}.
Indeed, \cite[(5.11)]{CGZ} is derived from  
\cite[(5.10)]{CGZ} by iterating \eqref{iter-I} for
consecutive terms of  
the sequence
$\ell_j = \ell/2 + \ell/(j+1)$, $s_j = \tau 2^{-(j-1)}$, 
starting at $(\ell,\tau)$ when $j=1$, and stopping 
at $j_0 := \min\{j \ge 1: \ell_j > s_j \}$ (since in 
their case, of \abbr{dtrw}, one has 
that $J(R,s)=0$ whenever $R > s$).
For such parameters $(\ell'-\ell)/(s'-s+1) \le \ell'/s' \le 1$,
hence $I(r)=r^2$ (even for \abbr{csrw}), and taking 
$\kappa \ge \theta_0^{-1} ( \log (2 C_1)+ \eta)$ makes \cite[CASE 1]{CGZ} 
hold here as well. Thus, the only difference is that for the \abbr{csrw} 
we still have to bound the last term 
of the iteration,
$e^{j_0 \eta} J(\ell_{j_0},s_{j_0})$, by
the \abbr{rhs} of \eqref{bd:I-taur}. For this task 
apply Lemma \ref{lem:mass-conc} to 
$\Gamma_s(z):=\pi_{t-s}(z) K_{t-s,t}(z,y)^2$, in which 
case \eqref{eq:tail-level} amounts to 
\[
J(R,s) \le c' e^{- R/C} \,, \qquad 
\forall R \ge s \ge 1 
\]
(for $C$ of \eqref{ghku-pi} and $c'(\CV,C)$ finite). Next, recall 
\eqref{eq:D-doubling} that 
$J(0,\tau) \ge (C_1 \CV)^{-1} \tau^{-c_1}$ for some $c_1$ 
finite and all $\tau \ge 1$. With $\ell_{j_0} \ge s_{j_0} \vee \ell/2$ and
$e^{j_0 \eta} \le \tau^{\eta/\log 2}$ (since $s_{j_0} \ge 1$), 
it thus follows that for $c_1'=c_1 + \eta/\log 2$ and some $C_1'$ finite,
\[
e^{j_0 \eta} J(\ell_{j_0},s_{j_0}) 
\le C_1' \, J(0,\tau) \, \tau^{c_1'} \, e^{-\ell/(2C)} \,,
\]
which for $\ell \ge \sqrt{\tau}$ and $\theta_0 \le (5C)^{-1}$, is 
further bounded by the \abbr{rhs} of \eqref{bd:I-taur}.
\end{proof}

\noindent
\emph{Proof of Theorem \ref{thm-main}(b).}
Proceeding to derive the matching \abbr{ghke} of \eqref{ghke}, 
since our assumptions apply for $(K_{s+r},\pi_{s+r}): r \in [0,t-s] \}$,
it suffices to do so for $s=0$ and fixed $x,y \in V$ such that 
$d(x,y) \le t$. 
\newline
$\bullet$ {\bf Step I: Improved \abbr{ghku}.}
Recall the \abbr{GHKU} \eqref{ghku-pi} implying that 
\eqref{eq:mod-ghku} holds for $K_{s,s+\tau}(x,\cdot)$ with 
$C<\infty$ independent of $s,\tau \ge 0$
and $x \in V$. Further, by the triangle inequality, we have for the 
non-increasing $t \mapsto \rho_t(x,z)$ of \eqref{dfn:rho_t} that 
\begin{equation}\label{eq:tri-rho}
\frac{1}{2} \rho_t(x,y) \le \rho_t(x,z) + \rho_t(y,z) \,, \qquad 
\forall z \in V, t \ge 0 \,.
\end{equation}
Hence, 
setting $t=2\tau$, $\tau \ge 1$ and 
$\theta_1 = \frac{1}{2} (\theta_0 \wedge C^{-1})$,
by Chapman-Kolmogorov and \eqref{eq:tri-rho}, followed by  
Cauchy-Schwartz, \eqref{bound:D-by-J}, Lemma \ref{lem:d-mgf}
and the inequality \eqref{eq:D-doubling}, we arrive at
\begin{align}
K_{0,t}(x,y)^2 e^{\theta_1 \rho_t(x,y)}
& \le \Big[ \sum_{z\in V} K_{0,\tau}(x,z)e^{\theta_1 \rho_t(x,z)}
K_{\tau,t}(z,y) e^{\theta_1 \rho_t(y,z)} \Big]^2 \nonumber\\
&\le \CV
D_{\tau,t}(y;2\theta_1)
\sum_{z\in V} K_{0,\tau}(x,z)^2e^{2\theta_1 \rho_\tau(x,z)}
\nonumber \\ 
& \le \CV C_2 D_{\tau,t}(y) \frac{C}{v(\sqrt{\tau})} 
\sum_{z\in V} K_{0,\tau}(x,z) \le \frac{C'}{v(\sqrt{t})} D_{0,t}(y) \,,  
\label{eq:c-s}
\end{align}
where $C' = \CV^2 C_2 C_1 C$.
Applying the same argument on $[\tau,2\tau]$ instead of $[0,t]$,
yields that for any $y \in V$ and $\tau \ge 1$,
\[
\Gamma_\tau(z) := \frac{\pi_{\tau}(z) K_{\tau,2\tau}(z,y)^2}
{D_{\tau,2\tau}(y)} 
\le \frac{\CV C'}{v(\sqrt{\tau})} e^{-\theta_1 \rho_{\tau}(z,y)} \,.
\]
In view of \eqref{dfn:D}, these $\Gamma_\tau(\cdot)$ are 
probability measures on $V$, hence by Lemma \ref{lem:mass-conc} there
exists $\kappa(\CV, \CV C' \vee \theta_1^{-1},1)$ finite such that
for $R=\kappa \sqrt{\tau}$, 
\begin{equation}\label{eq:improv1}
\Gamma_\tau(\B(y,R)) D_{\tau,t}(y) 
\ge \frac{1}{2} D_{\tau,t}(y) \ge \frac{1}{2 C_1} D_{0,t}(y)
\end{equation}
(using the \abbr{rhs} of \eqref{eq:D-doubling} 
for the last inequality). By the 
definition \eqref{evol-meas}
and Lemma \ref{lem:csrw-sol}(a) (at
$x_1=z$, $x_2=x$, $z_\star=y$),
we have for $\gamma(\kappa/2,1/2)>0$ and $R$ as above,
\begin{equation}\label{eq:improv2}
\Big[ \frac{\mu_{0,t}(y)}{\pi_0(\B(y,R))} \Big]^2 \ge 
\inf_{x \in \B(y,R)} \{ K_{0,t}(x,y)^2 \} 
\ge \gamma^2 
D_{\tau,t}(y) \frac{\Gamma_\tau(\B(y,R))}
{\pi_\tau(\B(y,R))} 
\end{equation}
(where for \abbr{dtrw} we restrict to $\tau \ge 2 R$). 
In view of the assumed uniform volume growth with $v(r)$ doubling,
from \eqref{eq:improv1} and \eqref{eq:improv2} it follows that 
for some $C_3(\CV,C_1,\kappa)$ finite,
\begin{equation}\label{eq:improv3}
\mu_{0,t}(y)^2 \ge \frac{\gamma^2}{2C_1 \CV^3} 
 v(R) D_{0,t}(y) 
\ge C_3^{-1} v(\sqrt{t}) D_{0,t}(y) \,.
\end{equation}
For \abbr{dtrw}, our derivation of \eqref{eq:improv3} required 
$t \ge 2 (2 \kappa)^2$, but with $K_t(\cdot,\cdot)$ 
uniformly elliptic, one easily extends
\eqref{eq:improv3} to all $t \ge 0$
upon increasing $C_3$ to some $C_3(\bar \alpha)$ finite.
Finally, combining \eqref{eq:c-s} and \eqref{eq:improv3} we have 
that for 
$C_\star := \sqrt{C_3 C'} \vee (2/\theta_1)$ finite,
\begin{align}\label{eq:ubd-interm}
K_{0,t}(x,y) \le \frac{C_\star \mu_{0,t}(y)}{v(\sqrt{t})}
e^{-\rho_t(x,y)/C_\star} \,,
\end{align}
as stated in the \abbr{rhs} of \eqref{ghke}.

\noindent
$\bullet$ {\bf Step II: matching \abbr{GHKL}.}
With \eqref{eq:ubd-interm} holding for $K_{s,s+t}(\cdot,\cdot)$,
$s \ge 0$, it yields the bound \eqref{eq:pos} for $b=1/2$, the
probability measures $K_{s,s+\tau}(x,\cdot)$ and 
$\Gamma_{\tau}(\cdot):=\pi_{s} (\cdot) 
K_{s,s+\tau}(\cdot,y)/\mu_{s,s+\tau}(y)$,
some $\kappa(\CV,C,b) \ge 2$, all $x,y \in V$, $s \ge 0$ and $\tau\ge b$.
Fixing $\varphi \ge 2(1+\kappa^2)$, $\delta=
1/2$ and $\gamma
\in (0,1)$ as in 
\eqref{eq:phi-comp}, we further have that for all 
$x,y \in V$, $t \ge \tau \ge 1$ and
$r \le 2 \varphi \sqrt{\tau}$ (with $4 r \le \tau \in \N$ in case
of \abbr{dtrw}),
\begin{align}\label{eq:dlbd}
K_{t-\tau,t}(x,y) \ge \gamma \sup_{z \in \B(x,2r)} 
\{ K_{t-\delta \tau,t} (z,y) \} \,.
\end{align}
Setting $n_\star = 1$ for \abbr{csrw} and
$n_\star = \lceil  (8 \varphi)^2 \rceil$ for \abbr{dtrw},
\eqref{eq:dlbd} applies when $r=[2 \varphi \sqrt{\tau}]$ and
$\tau \ge n_\star$. Further, if 
$d(x,y) \le r$ then $\B(y,r) \subseteq \B(x,2r)$,
so $r=[2 \varphi \sqrt{\tau}] \ge \kappa \sqrt{\tau}$ yields 
by \eqref{eq:pos}  
\begin{align}\label{eq:dlbd2}
K_{t-\tau,t}(x,y) \ge \frac{\gamma}{\pi_{t-\delta \tau}(\B(y,r))}
\sum_{z \in \B(y,r)} \pi_{t-\delta \tau} (z) K_{t-\delta\tau,t}(z,y) \ge 
\frac{\gamma \, \mu_{t-\delta \tau,t}(y)}{2 \CV v(2 \varphi\sqrt{\tau})} \,.
\end{align}
With $v(\cdot)$ volume doubling and
$s \mapsto \mu_{s,t}(y)$ is non-decreasing, taking 
$\tau=t \ge n_\star$ in \eqref{eq:dlbd2} yields the
\abbr{ghkl} for near-diagonal 
$d(x,y) \le 2 \varphi \sqrt{t}$. It 
extends to all $d(x,y) \le t < n_\star$ since only $y=x$ is 
relevant for $t<1$ (and the \abbr{ghkl} then trivially 
holds), and for \abbr{DTRW} having uniformly elliptic 
conductances implies that $K_{0,t}(x,y) \ge (\alpha_e)^{n_\star}$  
whenever $d(x,y) \le t  < n_\star$.

Considering hereafter 
$d(x,y) \in [2 \varphi \sqrt{t},t]$ and $t \ge n_\star$,
fix integers $R=[2 \varphi t/d(x,y)] \ge 
4$ 
and $\ell=\lceil d(x,y)/R \rceil \ge 
4$. We
further find $x_i \in V$ with $x_0=x$ and $x_\ell=y$  
such that $d(x_i,x_{i+1}) \le R$ for $0 \le i \le \ell-1$. Setting 
$\tau = t/\ell \ge 2$ 
(or its integer part for the \abbr{dtrw}), let 
$t_0=0$, $t_{2\ell}=t$ and $t_{2i-1}:=(t_{2i}+t_{2(i-1)})/2$,
with $t_{2i}-t_{2(i-1)}=\tau$ for \abbr{csrw}, or
in $\{\tau,\tau+1\}$ for \abbr{dtrw} (as needed).
It is easy to check that $\kappa \sqrt{\tau} \le R 
\le 2 \varphi \sqrt{\tau}
$,
and further that the extra requirement $4 R \le \tau \in \N$ which we need
in case of the \abbr{dtrw}, holds whenever $d(x,y) \le t/6$. For
such $x,y,t$
we get by Chapman-Kolmogorov and \eqref{eq:dlbd} followed by \eqref{eq:pos}, 
that 
\begin{align*}
K_{0,t}(x,y) &\ge \sum_{\{
z_i \in \B(x_i,2R)\}} 
K_{0,t_2}(x,z_1)
\Big[\prod_{i=2}^{\ell-1} K_{t_{2(i-1)},t_{2i}}(z_{i-1},z_i) 
\Big] K_{t_{2(\ell-1)},t}(z_{\ell-1},y)  \\
\ge & {\gamma}^{\ell-1}
\Big[ \sum_{z \in \B(x,R)} K_{0,t_2}(x,z) \Big] 
 \prod_{i=2}^{\ell-1} 
\Big[ \sum_{z \in \B(x_{i-1},R)} K_{t_{2i-1},t_{2i}}(x_{i-1},z) 
\Big] \, K_{t_{2\ell-1},t}(x_{\ell-1},y) 
\\
\ge & 
\left(\frac{{\gamma}}{2}\right)^{\ell-1} 
K_{t_{2\ell-1},t}(x_{\ell-1},y) \ge \left(\frac{\gamma}{2}\right)^\ell \frac{\mu_{t-\delta \tau',t}(y)}{\CV 
v(2 \varphi \sqrt{\tau'})} \,,
\end{align*}
where $d(x_{\ell-1},y) \le R$ so the last inequality 
is merely \eqref{eq:dlbd2} 
for $\tau':=t- t_{2\ell-1} \in [\tau/2,\tau]$ and 
$\kappa \sqrt{\tau'} \le R \le 2 \varphi \sqrt{\tau'}$.
Consequently, with $\tau \le t$, $v(\cdot)$ volume doubling
and $\ell \ge d(x,y)^2/(2 \varphi t)$, 
\begin{align}\label{eq:ghkl-mu}
K_{0,t}(x,y) &\ge
\left(\frac{\gamma}{2}\right)^{\ell} \frac{\mu_{0,t}(y)}{\CV
v(2 \varphi \sqrt{t})}
\ge \frac{\mu_{0,t}(y)}{C v(\sqrt{t})}\exp\{-C d(x,y)^2/t\}\,,
\end{align}
for some $C=C(\CV,\gamma,\varphi)$ finite. Note that for
$d(x,y)/t \in [1/6,1]$ and uniformly elliptic \abbr{dtrw}
we have that
$K_{0,t}(x,y) \ge (\alpha_e)^t \ge e^{-C d(x,y)^2/t}$,  
where $C = 1 \vee 36 \log (1/\alpha_e)$.
Recalling that $\mu_{0,t}(y) \le 
1 \le C v(\sqrt{t})$ 
this extends the validity of  \eqref{eq:ghkl-mu} to all 
$d(x,y) \le t$.
\qed

\end{section}

\begin{section}{The perturbative regime}\label{sect:proof-ppn}
\noindent
For $\{K_t\}$ of \eqref{n-kernel}
the transition kernels $\{K_{s,t}\}$ are unchanged by 
re-scaling the conductances 
\begin{equation}
\wh{\pi}_{u;v}(x,y)=e^{a_u-a_v} \, \pi_u(x,y)\,\quad\qquad\forall (x,y)\in E 
\label{dfn:pi-hat}
\end{equation}
In particular, for $\{a_t\}$ of \eqref{pertur-regime} one has
that $a_{u'}-a_u \ge \rho_{\bpi}(u',u)$ for all $u' \ge u$ and 
hence $u \mapsto \wh{\pi}_{u;v}(x)$ is non-decreasing 
(for each $x \in V$). More generally, working under 
the framework of Example \ref{inc-ex}, the Nash profiles re-scale as 
\[
\cN_{Q_u,\wh{\pi}_{u;v}}(\sss)=\cN_{Q_u,\pi_u}(e^{a_v-a_u} \, \sss),
\]
yielding an on-diagonal transition density upper bound
when $u \mapsto \wh{\pi}_{u;v}(x)$ are non-decreasing and
\eqref{cond:pert} holds. Also, if \eqref{cond:pert} applies
for $a_t$ of \eqref{pertur-regime}, then
$\mu_{s,t}(\cdot)$ of \eqref{evol-meas} are uniformly 
bounded below provided $s/t$ is.
\begin{lem}\label{per-n} 
Consider $\{K_t\}$ of \eqref{n-kernel} with non-decreasing 
$u \mapsto \wh{\pi}_{u;v}(x)$ of \eqref{dfn:pi-hat},
such that the non-decreasing $t \mapsto a_t$ satisfies \eqref{cond:pert}. 
\newline
(a). Suppose $c_\star := \inf_{t,x} \{\pi_t(x)\} > 0$ and 
for some finite $\CN$, $\sss_0$ and 
non-decreasing $N(\cdot)$, 
\begin{equation}
N(\sss)\ge\sup_t\{\cN_{K_t^2,\pi_t}(\sss)\}\,, \qquad 
\inf_{\sss \ge \sss_0} \Big\{ \frac{N(\CN\,\sss)}{N(\sss)} \Big\} \ge 2  \,.
\label{eq:nash-unif-bd}
\end{equation}
Then, for the dynamic \eqref{def:k-mn},  
$\psi(\st)=1/F^{-1}(\st;c_\star,N(\cdot))$, 
$\CN'=\CN'(\CN,\sss_0/c_\star)$ and
$F(\cdot)$ of \eqref{eq:N-psi},
\begin{align}
\sup_{x,y \in V} \Big \{ \frac{K_{s,t}(x,y)}{\pi_t(y)} \Big\} 
\le e^{2A}\CN'
\psi\Big(\frac{t-s}{6}\Big) \,, \qquad \forall s \in [0,t] \,.
\label{eq:psi-h}
\end{align}
For the dynamic \eqref{def:k-st}, replace 
$\cN_{K_t^2,\pi_t}(\sss)$ by $2 \cN_{K_t,\pi_t}(\sss)$ 
in \eqref{eq:nash-unif-bd} and
$\psi(\frac{t-s}{6})$ 
on the \abbr{RHS} 
of \eqref{eq:psi-h} 
by $\E [\psi(Z)]$, where $Z \sim \frac{1}{3}\rm{Poisson}(t-s)$.
\newline
(b). If $\{a_t\}$ of \eqref{pertur-regime} satisfies \eqref{cond:pert}, 
then for $\mu_{s,t}$ of \eqref{evol-meas}, under
either 
\eqref{def:k-mn} or \eqref{def:k-st}, 
\begin{equation}
\mu_{s,t} (y)\ge 
e^{-\gamma A} \pi_t(y)\,,
\quad \qquad \forall \, y\in V, \; \gamma \in \N\,,
\;\;\;\; (t+1) \le 2^\gamma (s+1) \,.
\label{eq:ref-meas}
\end{equation}
\end{lem}

\begin{proof} (a). Fixing $s \in [0,t]$, set 
$v=(t+s)/2$, or its integer part in case of \abbr{dtrw}.
Note that $t-v \ge (t-s)/2$ and $(t+1) \le 4(v+1)$. Further, with 
$v \in [s,t]$, we have that $K_{s,t}=K_{s,v} K_{v,t}$
and since the Markov kernel $K_{s,v}$ is an $L^{\infty}$-contraction, 
the \abbr{lhs} of \eqref{eq:psi-h} is bounded above by 
$\Vert K_{v,t} \Vert_{L^{1}(\pi_t)\to L^{\infty}(\pi_v)}$. To 
bound the latter quantity, consider for $u \ge v$ the non-decreasing  
$u \mapsto \wh{\pi}_{u;v} (\cdot)$. 
Since $\underline{\wh{\pi}_{u;v}}\ge\underline{\pi_u} \ge c_\star$ and
\[
\cN_{K_u^{2},\wh{\pi}_{u;v}}(\sss)=\cN_{K_u^{2},\pi_u}(e^{a_v-a_u} \, 
\sss)\le N(\sss) \,,
\]
applying Theorem \ref{nash-inc}(a) for the 
dynamics \eqref{def:k-mn} and 
$\left\{ (K_u,\wh{\pi}_{u;v}), u \in [v,t] \right\}$, we have that 
\begin{align}\label{pert-1-infty}
\left\Vert K_{v,t}\right\Vert _{L^{1}(\pi_t)\to L^{\infty}(\pi_v)}
&=
e^{a_t-a_v} \,
\left\Vert K_{v,t}\right\Vert _{L^{1}(\wh{\pi}_{t;v})\to L^\infty(\wh{\pi}_{v;v})} 
\le
e^{2A} \, \CN'
\psi\Big(\frac{t-v}{3}\Big) 
\end{align}
yielding \eqref{eq:psi-h}. Replacing Theorem \ref{nash-inc}(a) 
by Theorem \ref{nash-inc}(b), the analogous argument applies for the 
dynamic \eqref{def:k-st}.

\smallskip
\noindent
(b). Fixing $x_s=x$ and $x_t=y$,
we have from \eqref{n-kernel} that for the dynamic \eqref{def:k-mn}, 
\begin{align}
K_{s,t}(x,y) 
& =\sum_{\{x_{s+1},\ldots,x_{t-1}\}}\;\prod_{r=s+1}^{t}\frac{\pi_{r}(x_{r})}{\pi_{r}(x_{r-1})}K_{r}(x_{r},x_{r-1})
  \ge\eta_s(t)\frac{\pi_t(y)}{\pi_s(x)}\big[K_t\cdots K_{s+1}\big](y,x)\,,\label{eq:revers} 
\end{align}
where under \eqref{pertur-regime},
\[
\eta_s(t) :=\prod_{r=s}^{t-1}
\inf_{z\in V}\Big\{\frac{\pi_{r}(z)}{\pi_{r+1}(z)}\Big\}\ge
\prod_{r=s}^{t-1} e^{-(a_{r+1}-a_r)} = 
e^{-(a_t-a_s)} \,.
\]
Multiplying 
\eqref{eq:revers} by $\pi_s(x)$ and
summing over $x$ we see that 
$\mu_{s,t}(y)/\pi_t(y) \ge e^{-(a_t-a_s)} 
$, which by \eqref{cond:pert} is further bounded below by
$e^{-\gamma A}$ whenever $(t+1) \le 2^\gamma (s+1)$
(see \eqref{doub:pert}).
Next, 
recall Remark \ref{mon-mu-m-n} that $\mu_{s,t}$ of 
the \abbr{csrw} is the expected value over
$L
\!\!\sim$Poisson($t-s$) 
and jump times $s=T_0' < T'_{1} < \cdots < T'_L \le t$ of the
value $\mu^{(\omega)}_{0,L}$ for the \abbr{dtrw}
using $\{(K_{T'_m},\pi_{T'_m})\}$. Since 
$a_{T'_{m+1}}-a_{T'_m} \ge \rho_{\bpi} (T'_{m+1},T'_m)$ 
for all $m \in \N$, due to \eqref{pertur-regime}, 
by the preceding argument, for each $\omega$,
\[
\mu_{0,L}^{(\omega)}(y) \ge e^{-(a_{T'_L}-a_{T'_0})} \pi_{T'_L}(y)
\ge e^{-(a_t-a_s)} \pi_t(y) \,.
\]
Thus, having \eqref{cond:pert} for $t \mapsto a_t$, implies that  
\eqref{eq:ref-meas} holds also for the dynamic \eqref{def:k-st}.
\end{proof}

\smallskip
\noindent
{\emph {Proof of Proposition \ref{perturbative}.}} 
(a)
Recall from the proof of Lemma \ref{P+VG-to-Nash} that for 
\abbr{csrw} or uniformly lazy \abbr{dtrw} on
$\blG_t$, the assumed uniform Poincar\'e inequality and volume growth $v(r)$ 
with $v(r)$ doubling yield the Nash profile bound 
for $N(\cdot)$ of \eqref{eq:nash-profile-vd}. Following the rest of 
the proof of Lemma \ref{P+VG-to-Nash}, while applying 
Lemma \ref{per-n}(a) instead of Theorem \ref{nash-inc}, we deduce that
the on-diagonal \abbr{ghku} bound \eqref{eq:diag-ubd-nash-rev} also holds here, 
for some $C'(A,\CP,\CV,\alpha_l)$ finite. Next, 
adapting the proof of Proposition \ref{prop:gauss-ubd}
we proceed to deduce for $T\ge 1$
the $1 \to \infty$ norm 
bound similar to \eqref{eq:op-s-norm} for 
the operator $K_{0,2T}^\theta$.
To this end, we use here the dual $\wh{K}_{s,t}^{\star}$  
of $K_{s,t}:L^{2}(\wh{\pi}_{s;T})\to L^{2}(\wh{\pi}_{t;T})$
for the re-scaled non-decreasing conductances $\wh{\pi}_{r;T}$
of \eqref{dfn:pi-hat}, and $r \ge T$.
Replacing \eqref{1-to-2}, we have 
\[
\left \Vert K_{T,2T}^\theta \right \Vert_{L^1(\pi_{2T}) \to L^2(\pi_T)} 
= 
e^{a_{2T}-a_T} \, \left \Vert K_{T,2T}^\theta \right \Vert_{L^1(\wh{\pi}_{2T;T}) 
\to L^2(\wh{\pi}_{T;T})}
\le e^{A} 
\left \Vert (\wh{K}^\star_{T,2T})^{-\theta} \right \Vert_{L^2(\wh{\pi}_{T;T}) \to 
L^\infty(\wh{\pi}_{2T;T})} 
\]
so it suffices
to bound via Proposition \ref{prop:hs}, 
the $L^{2}(\wh{\pi}_{T;T})\to L^{\infty}(\wh{\pi}_{2T;T})$
norm of $(\wh{K}_{T,2T}^{\star})^{-\theta}$ as in \eqref{duality}
and the
$L^{2}(\pi_T)\to L^{\infty}(\pi_{0})$ norm of $K_{0,T}^{\theta}$ as in \eqref{2-to-infty}. For the $2 \to \infty$
bound on $(\wh{K}_{T,2T}^{\star})^{-\theta}$, 
recall \eqref{dual-contr}
that since $r \mapsto \wh{\pi}_{r;T}$ are non-decreasing, 
$\| \wh{K}^\star_{s,t}\|_{L^\infty(\wh{\pi}_{s;T}) \to 
L^\infty(\wh{\pi}_{t;T})} \le 1$ and Lemma \ref{gaffney} provides 
the $2 \to 2$ analog of \eqref{2-to-2-star} for 
$(\wh{K}^\star_{s,t})^{-\theta}$. Further, the
$1\to\infty$ bound of \eqref{eq:diag-ubd-nash-rev} on $K_{s,t}$ 
yields the $2\to\infty$ bound \eqref{eq:on-diagonal} with 
$\varphi(\tau)=C' v(\sqrt{\tau})^{-1/2}$ and in conjunction 
with \eqref{1-to-infty-star} for $\wh{\pi}_{r;T}$, 
allows us to apply Proposition \ref{prop:hs} with 
this choice of $\varphi(\tau)$ for the 
adjoint operators $\wh{K}^\star_{s,t}$. Similarly to \eqref{duality}
it yields that for some $C_2(A)$ and $C_1$ finite,  
\begin{align*}
\Big\| \big(\wh{K}_{T,2T}^{\star}\big)^{-\theta}\Big\| _{L^{2}(\wh{\pi}_{T;T})\to L^{\infty}(\wh{\pi}_{2T;T})} 
\le C_2 v(\sqrt{T})^{-1/2} \exp\left(C_{1}\chi(\theta) T\right).
\end{align*}
Turning next to the $2\to\infty$ bound on $K_{0,T}^{\theta}$, 
utilizing the non-decreasing measures $u\mapsto\wh{\pi}_{u;v}$
Lemma \ref{gaffney} provides the $2\to2$ bound 
on $K_{v,u}^{\theta}$ with $\wh{\pi}_{u;v}$ replacing $\pi_u$
(and $a_u \equiv 0$). Thus,  
\begin{equation*}
\left\Vert K_{v,u}^{\theta}\right\Vert _{L^{2}(\pi_u)\to L^{2}(\pi_v)}=
e^{(a_u-a_v)/2}
\left\Vert K_{v,u}^{\theta}\right\Vert _{L^{2}(\wh{\pi}_{u;v})\to L^{2}(\wh{\pi}_{v;v})}\le \exp\big( a_u - a_v+\chi(\theta)(u-v) \big)
\,.
\end{equation*}
The non-decreasing $t \mapsto a_t$ satisfies \eqref{cond:pert},
so Proposition \ref{prop:hs} establishes the bound
\eqref{2-to-infty} which in turn yields the
heat kernel upper bound \eqref{ghku-pi} (see 
our proof of Theorem \ref{thm-main}(a)).

\medskip
\noindent
(b). As in the proof of Theorem \eqref{thm-main}(b) it suffices 
to establish the relevant \abbr{ghkl} for $K_{0,2T}$. 
To this end,  
thanks to \eqref{eq:ref-meas} and the assumed uniform volume growth
with $v(\cdot)$ doubling, we have that 
$\{\mu_{s,t}(\cdot):  s,t \in [T,2T] \}$ are $c$-stable with respect 
to $\pi(x) \equiv 1$ (for $c=c(\CV,A)$ and all $T$). Moreover, 
in part (a) we established the bound
\eqref{ghku-pi}, therefore $\{K_{s,s+t} : s,s+t \in [T,2T]\}$ 
satisfy the improved \abbr{ghku} of \eqref{eq:ubd-interm}.
The \abbr{phi} is invariant to the re-scaling \eqref{dfn:pi-hat}
and $a_t \le a_T + A$ for $t \in [T,2T]$. Hence, considering 
Theorem \ref{harnack} for $(K_{t},\widehat{\pi}_{t;T})$, yields for the 
\abbr{csrw} that $(K_{t},\pi_t)$ satisfy the \abbr{phi} on $[T,2T]$,
while for \abbr{dtrw} we assumed such \abbr{PHI}, as well
as uniformly elliptic conductances. Proceeding 
as in {\bf Step II} of the proof of Theorem \ref{thm-main}(b),
yields the \abbr{ghkl} of \eqref{ghke} for
some $\Cbu=\Cbu (\CP,\gamma,\CV,\bar \alpha,A)$ finite 
and $K_{T,2T}$ (omitting \abbr{wlog} the bounded factor $\mu_{T,2T}$). 
From the upper bound \eqref{ghku-pi} it further follows, 
as in \eqref{eq:pos}, that for some $\kappa(\CP,\CV,\alpha_l,A)$ finite,  
\begin{equation}\label{eq:lbd-k}
\sum_{z \in B(x,\kappa \sqrt{T})} K_{0,T}(x,z) \ge \frac{1}{2} \,,
\qquad \forall x \in V, \;\; T \ge 1 \,.
\end{equation}
With $d(z,y)^2 \le 2 d(x,z)^2 + 2 d(x,y)^2$, we get by
combining \eqref{ghke} and \eqref{eq:lbd-k}, that 
\begin{align*}
K_{0,2T}(x,y) & 
   \ge\frac{1}{2} \inf_{z\in\B(x, \kappa \sqrt{T})} \{ K_{T,2T}(z,y) \} 
 \ge \frac{e^{-2 \Cbu \kappa^2}}{2 \Cbu v(\sqrt{T})} e^{-2 \Cbu d(x,y)^2/T} \,,
\end{align*}
whenever $d(x,y)\le T-\kappa \sqrt{T}$.
With $\blG_t$ uniformly elliptic, 
increasing $\Cbu$ (in terms of $\alpha_e$),   
such \abbr{GHKL} extends to all 
$d(x,y) \in [T-\kappa \sqrt{T}, 2T]$ (as we have 
seen already after \eqref{eq:ghkl-mu}). 
\qed

\medskip
\begin{proof}[Proof of Proposition \ref{ppn:conv}] We refine 
the counter-example provided in \cite[Proposition 1.4]{HK} 
for $\bG=\Z$, by fixing $\eta,\delta_n \in (0,1/2)$, {$\delta_0=0$}
and setting the uniformly bounded
\begin{align*}
\pi_n(x,x+1) = 1 + (-1)^{n+x} \eta \,, \qquad \pi_n(x,x)=1 - (-1)^{n+x} 
\delta_n \,. 
\end{align*}
Then, $\pi_n(x)=3-(-1)^{n+x}\delta_n$ satisfy
\eqref{pertur-regime} with $a_{n+1} - a_n  
\le \frac{2}{5} (\delta_n+\delta_{n+1})$, and  
$K_n(x,y)$ of \eqref{n-kernel} satisfies
\eqref{dfn:unif-lazy}-\eqref{dfn:unif-elliptic}
with $\bar \alpha=1/7$. The process $\{X_n\}$
induces on types $A$ and $B$ that correspond to $n+X_n$ being even 
and odd, respectively, the in-homogeneous 
$\{A,B\}$-valued Markov chain of transition probabilities:
\begin{align*}
q_{n}(A,B) & =\frac{1-\delta_n}{3-\delta_n},\quad q_{n}(A,A)=\frac{2}{3-\delta_n},\\
q_{n}(B,A) & =\frac{1+\delta_n}{3+\delta_n},\quad q_{n}(B,B)=\frac{2}{3+\delta_n}.
\end{align*}
The uniformly bounded increments $X_{n+1}-X_n$ are zero 
on transitions between types $A$ and $B$ and otherwise 
they are $\pm 1$-valued of mean $\Delta_{n}(A)=\frac{2\eta}{3-\delta_n}$  
and $\Delta_{n}(B)=-\frac{2\eta}{3+\delta_n}$,
when at time $n$ the type is $A$ or $B$, respectively. 
Note that 
$$
v_n := \frac{1}{2} \sum_{i=1}^{n} ( \Delta_{i}(A) + \Delta_{i}(B)) 
\ge 
\frac{2 \eta}{9} a_n = O(n^{1/2+\iota}) \,,
$$ 
and starting at $X_0=0$ (i.e. at type $A$), since 
$q_{n}(B,A)>1/3>q_{n}(A,B)$ for all $n$, such $\{A,B\}$-valued 
Markov chain induces the drift $\bE X_n \ge v_n$. Thus, 
from the 
concentration of the pair-empirical 
$\{A,B\}$-valued measure around its limit $[2/6,1/6,1/6,2/6]$, 
we deduce that for some $C>0$,  
\begin{align}\label{eq:cnt-ghk}
p_n := \bP(|X_{n}|\le C n^{(1+\iota)/2})\le C^{-1} \exp(-C n^{\iota}) \,.
\end{align}
It is easy to check that if the lower bound of \eqref{ghk-lbd} held
for the uniformly bounded above and below $\pi_n(\cdot)$, then 
necessarily $\inf_n \{p_n\} >0$ in contradiction with \eqref{eq:cnt-ghk}.
Further, even if only the upper bound of \eqref{ghk-ubd} held 
for $\pi_n(\cdot)$, then since $\iota>0$ necessarily 
$\bP(|X_n| > C n^{(1+\iota)/2}) \to 0$, again contradicting 
\eqref{eq:cnt-ghk}. 
\end{proof}

\end{section}

\medskip
\noindent
{\bf{Acknowledgments.}} 
We thank Takashi Kumagai for proposing to
pursue the \abbr{GHK} in our aim to verify \cite[Conj. 7.1]{ABGK}.
This work further benefited from discussions (of A.D.) 
with Martin Barlow and (of T.Z.) with Laurent Saloff-Coste.

\end{document}